\documentclass{amsart}
\usepackage{amssymb, latexsym}
\usepackage{graphics}

\newcommand{\N}{{{\mathbf N}}}

\newcommand\R{{{\mathbf R}}}
\newcommand\D{{{\mathbf D}}}
\newcommand\T{{{\operatorname T}}}
\renewcommand\P{{{\mathbf P}}}

\newcommand\C{{{\mathbf C}}}
\renewcommand\O{{{\mathcal O}}}
\renewcommand\H{{\operatorname{H}}}
\renewcommand\Im{{\operatorname{Im}}}
\renewcommand\Re{{\operatorname{Re}}}

\newcommand\diag{{\operatorname{diag}}}

\newcommand\low{{\operatorname{low}}}
\newcommand\high{{\operatorname{high}}}
\newcommand\bounded{{\operatorname{bounded}}}
\newcommand\local{{\operatorname{local}}}
\newcommand\med{{\operatorname{med}}}
\newcommand\lin{{\operatorname{lin}}}
\newcommand\Ball{{{\mathbf B}}}
\renewcommand\L{{\mathcal L}}
\newcommand\eps{\varepsilon}
\renewcommand\div{{\operatorname{div}}}

\theoremstyle{plain}
  \newtheorem{theorem}[subsection]{Theorem}
  
  \newtheorem{proposition}[subsection]{Proposition}
  \newtheorem{lemma}[subsection]{Lemma}

\theoremstyle{remark}
  \newtheorem{remark}[subsection]{Remark}

\theoremstyle{definition}
  \newtheorem{definition}[subsection]{Definition}

\include{psfig}

\begin{document}

\title[Global well-posedness of Maxwell-Klein-Gordon]{Global well-posedness of the
Maxwell-Klein-Gordon equation below the energy norm}
\author{Markus Keel}
\address{Department of Mathematics, U. Minnesota, Minneapolis, MN 55455}
\email{keel@math.umn.edu}

\author{Tristan Roy}
\address{Dapartment of MAthematics, UCLA, Los Angeles CA 90095-1555}
\email{triroy@math.ucla.edu}

\author{Terence Tao}
\address{Department of Mathematics, UCLA, Los Angeles CA 90095-1555}
\email{tao@math.ucla.edu} \subjclass{35J10,42B25}

\vspace{-0.3in}
\begin{abstract}
We show that the Maxwell-Klein-Gordon equations in three dimensions
are globally well-posed in $H^s_x$ in the Coulomb gauge for all $s >
\sqrt{3}/2 \approx 0.866$.  This extends previous work of
Klainerman-Machedon \cite{kl-mac:mkg} on finite energy data $s \geq
1$, and Eardley-Moncrief \cite{eardley} for still smoother data.  We
use the method of almost conservation laws, sometimes called the
``I-method",  to construct an almost conserved quantity based on the
Hamiltonian, but at the regularity of $H^s_x$ rather than $H^1_x$.
One then uses Strichartz, null form, and commutator estimates to
control the development of this quantity. The main technical
difficulty (compared with other applications of the method of almost
conservation laws) is at low frequencies, because of the poor
control on the $L^2_x$ norm. In an appendix, we demonstrate the
equations' relative lack of smoothing - a property that presents
serious difficulties for studying rough solutions using other known
methods.
\end{abstract}

\maketitle \setcounter{tocdepth}{1} \tableofcontents

\section{Introduction}

\subsection{The (MKG-CG) system}

Let $\R^{1+3}$ be Minkowski space endowed with the usual metric
$\eta := \diag(-1,1,1,1)$. Let $\phi: \R^{1+3} \to \C$ be a
complex-valued field, and let $A_\alpha: \R^{1+3} \to \R$ be a real
one-form. Here and in what follows, we use Greek indices to denote
the indices of Minkowski space, and Roman indices to denote spatial
indices (e.g. $\alpha, \beta, \gamma \in \{0,1,2,3\}; i,j,k \in
\{1,2,3\})$, raised and lowered in the usual manner.  One can then
think of $A$ as a $U(1)$ connection, and can define the
\emph{covariant derivatives} $D_\alpha$ by
\begin{equation}\label{conn-def}
D_\alpha \phi := (\partial_\alpha + i A_\alpha) \phi.
\end{equation}
We can define the curvature $F_{\alpha\beta}$ of the connection $A$
as the real anti-symmetric tensor
\begin{equation}\label{curv-def}
 F_{\alpha\beta} := \frac{1}{i} [D_\alpha, D_\beta] = \partial_\alpha A_\beta - \partial_\beta A_\alpha.
\end{equation}

The \emph{(massless) Maxwell-Klein-Gordon equations} for a complex
field $\phi$ and a one-form $A_\alpha$ are given by
\begin{equation*}\tag{MKG}
\begin{split}
\partial^\beta F_{\alpha \beta} &= \Im( \phi \overline{D_\alpha \phi} ) \\
D_\alpha D^\alpha \phi &= 0;
\end{split}
\end{equation*}
which are the Euler-Lagrange equations for the Lagrangian
$$
\int_{\R \times \R^3}  \frac{1}{2} D_{\alpha}\phi \overline{
D^{\alpha} \phi} + \frac{1}{4}F_{\alpha \beta} F^{\alpha \beta}\ dx
dt.
$$
Throughout this paper we follow the convention that repeated indices
are summed over their range. (For example, here $\partial^\beta
F_{\alpha \beta} := \sum_{\beta= 0}^3 \partial^\beta F_{\alpha
\beta}$ for each $\alpha$.) We split $A$ into the temporal component
$A_0$ and the spatial component $\underline{A} := (A_1, A_2, A_3)$.
We similarly split the covariant spacetime gradient $D_\alpha$ into
the covariant time derivative $D_0 = \partial_t + i A_0$ and the
covariant spatial gradient $\underline{D} := (D_1, D_2, D_3) =
\nabla_x + i \underline{A}$.

The Maxwell-Klein-Gordon system of equations has the gauge
invariance
$$
\phi \mapsto e^{i\chi} \phi; \quad A_\alpha \mapsto A_\alpha -
\partial_\alpha \chi
$$
for any (smooth) potential function $\chi: \R^{1+3} \to \R$.  From
this and some elementary Hodge theory, one can place this system of
equations in the \emph{Coulomb gauge} $\nabla_x \cdot \underline{A}
= 0$. In this gauge the Maxwell-Klein-Gordon equations become the
following overdetermined elliptic-hyperbolic system of equations
(see \cite{kl-mac:mkg}):
\begin{align}
\Delta A_0 & = - \Im(\phi \overline{D_0 \phi}) \label{a0-eq}\\
\partial_t \nabla_x A_0 &= - (1-\P) \Im (\phi \overline{\nabla_x \phi}) + (1-\P) (\underline{A} |\phi|^2)\label{a0t-eq}\\
\Box \underline{A} &= - \P \Im (\phi \overline{\nabla_x \phi}) + \P (\underline{A} |\phi|^2)\label{ba-eq}\\
\Box \phi &= - 2i (\P \underline{A}) \cdot \nabla_x \phi
+ 2 iA_0 \partial_t \phi + i(\partial_t A_0) \phi + |\underline{A}|^2 \phi - A_0^2 \phi \label{bphi-eq}\\
\nabla_x \cdot \underline{A} &= 0; \label{div0-eq}
\end{align}
here $\Box$ is the standard d'Lambertian
$$ \Box := \partial_\alpha \partial^\alpha = -\partial_t^2 + \Delta$$
and $\P := \Delta^{-1} d^* d$ is the spatial Leray projection onto
divergence-free vector fields
$$ \P \underline{A} := \Delta^{-1} \nabla_x \times (\nabla_x \times \underline{A}); \quad (1-\P) \underline{A} = \Delta^{-1} \nabla_x (\nabla_x \cdot \underline{A}).$$
Observe that $\P$ can be written as a polynomial combination of 
Riesz transforms $R_j := |\nabla_x|^{-1} \partial_j$, where
$|\nabla_x| := \sqrt{-\Delta}$.

We refer to the system \eqref{a0-eq}-\eqref{div0-eq} collectively as
(MKG-CG). We shall write $\Phi := (A_0,\underline{A},\phi)$ to
denote the entire collection of fields in (MKG-CG), and use
$\underline{\Phi} := (\underline{A}, \phi)$ to isolate the
``hyperbolic'' or ``dynamic'' component of these fields.  (From
\eqref{a0-eq} we see that $A_0$ obeys an elliptic equation rather
than a hyperbolic one.)

We can study the Cauchy problem for (MKG-CG) by specifying the
initial data\footnote{Here and in the sequel we use $\phi[t]$ as
short-hand for $(\phi(t), \phi_t(t))$.} $\Phi[0]$.  Although we
specify initial data for $A_0$, it is essentially redundant
(assuming some mild decay conditions on $A_0$ at infinity) since by
\eqref{a0-eq}, \eqref{a0t-eq}, \eqref{div0-eq} the data $\Phi[0]$
must obey the compatibility conditions
\begin{equation}\label{compat}
\begin{split}
\div \underline{A}(0) &= \div \partial_t \underline{A}(0) \; = \;  0\\
\Delta A_0(0) &= - \Im\left(\phi(0) \overline{D_0 \phi(0)}\right)\\
\partial_t \nabla_x A_0(0) &= - (1-\P) \Im\left(\phi(0) \overline{\underline{D} \phi(0)}\right).
\end{split}
\end{equation}
In Section \ref{elliptic-sec} we show how these conditions allow
$A_0$ to be reconstructed from $\underline \Phi$.

\begin{remark}
If we ignore the ``elliptic part'' of the equations, then
heuristically the system (MKG-CG) takes the schematic
form\footnote{The $\O()$ notation is made precise in the second
paragraph of Section \ref{caricature-sec} below.  For now, the
notation can be taken to mean ``terms that look like".  }
\begin{equation}\label{bad-caricature}
\Box \Phi = \O( \Phi \nabla_{t,x} \Phi ) + \O( \Phi \Phi \Phi )
\end{equation}
although this caricature does not capture the full structure of the
equation.  Indeed, in \cite{klainerman:nulllocal} it was observed
that the interplay between \eqref{div0-eq} and the bilinear terms in
\eqref{ba-eq}, \eqref{bphi-eq} allow us to write the most important
components of the quadratic portion $\O(\Phi \nabla_{t,x} \Phi)$ of
the nonlinearity in terms of the null forms $ Q_{jk}(\varphi,\psi)
:=
\partial_j \varphi \partial_k \psi - \partial_k \varphi \partial_j \psi$.
We shall return to this point in Section \ref{caricature-sec}.
\end{remark}

\begin{remark}
The system (MKG-CG) is invariant under the scaling
\begin{equation}\label{mkg-scaling}
\Phi(t,x) \mapsto \frac{1}{\lambda} \Phi(\frac{t}{\lambda},
\frac{x}{\lambda})
\end{equation}
which suggests that the natural scale-invariant space for the
initial data $\Phi[0]$ is $\dot H^{1/2} \times \dot H^{-1/2}$.
\end{remark}

For any\footnote{We remark that the condition $s > 1/2$, besides
being natural from scaling considerations, is also important for
making sense of the non-linearity $\O( \Phi \nabla_{t,x} \Phi )$ and
the compatibility conditions \eqref{compat}, since when $s < 1/2$
one cannot make sense of a product of an $H^s_x$ function and an
$H^{s-1}_x$ function, even in the sense of distributions.  We do not
consider here the delicate issue of what happens at the critical
regularity $s=1/2$.} $s > 1/2$, define the norm $[H^s]$ on initial
data\footnote{As is usual for wave equations, regularity in time and
regularity in space are essentially equivalent, so we always expect
$\partial_t \Phi$ to have one lower degree of spatial regularity
than $\Phi$.} by
\begin{equation} \label{hsbracket}
\| \Phi[0] \|_{[H^s]} := \| \nabla_{x,t} \Phi(0) \|_{H^{s-1}_x} + \|
\underline{\Phi}(0) \|_{H^s_x} \end{equation} where $H^s_x$ is the
usual inhomogeneous Sobolev norm $\| u \|_{H^s_x} := \| \langle
\nabla\rangle^s u \|_{L^2_x}$, not to be confused with the
homogeneous Sobolev norm $\| u \|_{\dot H^s_x} = \| |\nabla_x|^s u
\|_{L^2_x}$.  Here and in the sequel we use $\langle x \rangle$ as
short-hand for $(1 + |x|^2)^{1/2}$.  We refer to $[H^1]$ in
particular as the \emph{energy class}.

\begin{remark}
The energy class $[H^1]$ is almost the $H^1_x \times L^2_x$ norm of
$\Phi[0]$, a difference being that we do not place $A_0$ in $L^2_x$.
Indeed, even if $\underline{\Phi}$ is smooth and compactly
supported, we see from \eqref{a0-eq} and the fundamental solution of
the Laplacian that $A_0$ might only decay as fast as $O(1/|x|)$ at
infinity, which is not in $L^2_x$.  Thus we see that the non-local
nature of the Coulomb gauge causes some difficulties\footnote{On the
other hand, we do not have these issues with the time derivative
$\partial_t A_0$.  Indeed, if $\underline{\Phi}(0) \in H^1_x$ then
from \eqref{a0t-eq} and some Sobolev embedding we see that $\nabla_x
\partial_t A_0 \in L^{6/5}$, and hence that $\partial_t A_0 \in
L^2_x$.} with the low frequency component of $A_0$.  Although these
difficulties will cause much technical inconvenience, they are not
the main enemy in the low regularity theory, and we recommend that
the reader ignore all mention of low frequency issues at a first
reading.  In particular, the reader should initially ignore the
technical distinctions between the inhomogeneous Sobolev norm
$H^s_x$ and the homogeneous counterpart $\dot H^s_x$.
\end{remark}

\subsection{Prior results}

The following local and global well-posedness results are known.
Regarding global solutions, if the initial data $\Phi[0]$ is smooth
and obeys the compatibility conditions \eqref{compat}, then there is
a unique smooth global solution of (MKG-CG) with initial data
$\Phi[0]$; see \cite{eardley}.  Furthermore, one has global
well-posedness in the energy class and above:

\begin{theorem}\label{h1-gwp}\cite{kl-mac:mkg}  Let $s \geq 1$, and $\Phi[0] \in [H^s]$ obey the compatibility conditions \eqref{compat}.  Then there exists a global solution $\Phi$ to (MKG-CG) with initial data $\Phi[0]$.  Furthermore, for each time $t$ the solution map $\Phi[0] \mapsto \Phi[t]$ is a continuous map from $[H^s]$ to $[H^s]$, and we have the bound
\begin{equation}\label{poly}
\| \Phi[T] \|_{[H^s]} \leq C(s,\|\Phi[0]\|_{[H^s]}) \langle
T\rangle^{C(s)}
\end{equation}
for all $T > 0$, where $C(s,\|\Phi[0]\|_{[H^s]})$ and $C(s)$ are
positive quantities depending only on $s, \|\Phi[0]\|_{[H^s]}$ and
on $s$ respectively.
\end{theorem}

One of the key ingredients in obtaining global well-posedness (as
opposed to just local well-posedness) is the well-known fact that
the flow (MKG-CG) preserves the Hamiltonian
\begin{equation}\label{hamil-def}
\begin{split}
\H[\Phi[t]] := \int& \frac{1}{2} |\nabla_x A_0(t) - \partial_t
\underline{A}(t)|^2
+ \frac{1}{2} |\nabla_x \times \underline{A}(t)|^2\\
&+ \frac{1}{2}|D_0 \phi(t)|^2 + \frac{1}{2} |\underline{D}
\phi(t)|^2\ dx.
\end{split}
\end{equation}
This Hamiltonian is clearly non-negative.  In the Coulomb gauge
$\nabla_x \cdot \underline{A} = 0$ the Hamiltonian turns out to be
roughly equivalent to $\| \Phi[t] \|_{[H^1]}^2$; see Section
\ref{hamil-sec}.

Theorem \ref{h1-gwp} includes also  local well-posedness in $[H^s]$
in the range $s \geq 1$.  This condition for local well-posedness
has been lowered to $s > 3/4$ by Cuccagna \cite{cuccagna} (see also
Theorem \ref{hs-lwp} below), and down to the near-optimal value of
$s > 1/2$ in \cite{machedon} (see also a similar result for a model
problem in \cite{kl-mac:optimal}, and \cite{kl-tar:ym},
\cite{selberg:mkg} for analogous results in higher dimensions). Our
results here shall rely primarily on the local theory in
\cite{cuccagna} and not on the more sophisticated techniques in
\cite{machedon}.

\subsection{Main result}

The purpose of this paper is to consider the corresponding question
of global well-posedness below the energy class.  Our main result is

\begin{theorem}\label{main} Theorem \ref{h1-gwp} also holds in the range $1 > s > \sqrt{3}/2$.
\end{theorem}

\begin{remark} This result was announced previously by the first and third authors in the smaller range $1 > s > 7/8$.  Prior to the finalization of this paper, we had announced this result for the improved range $s > 5/6$, but some of
the estimates used in that argument turned out to be incorrect.
\end{remark}

These results can be compared with the theory for the nonlinear wave
equation
\begin{equation*}\tag{NLW}
\Box \varphi = |\varphi|^2 \varphi  
\end{equation*}
(compare with \eqref{bad-caricature}).  This equation has the same
scaling \eqref{mkg-scaling} as (MKG-CG) but is simpler due to the
lack of derivatives in the nonlinearity.  For this equation one has
local well-posedness all the way down to the critical regularity $s
\geq 1/2$ (though at $s=1/2$ the time of existence depends on the
profile of the data, and not just its norm), and global
well-posedness for small $\dot H^{1/2}$ data, see e.g.
\cite{lindbladsogge, kapitanski, sogge:wave}.  For large data global
well-posedness is known for $s \geq \frac{13}{18}$ \cite{roy1},
extending previous work that had gotten to $s > \frac{3}{4}$
\cite{kpv:gwp}, \cite{gallagher}, \cite{chemin}.  Since the local
well-posedness theory for (MKG-CG) has been improved in
\cite{machedon} to nearly match that for (NLW), one might therefore
hope to improve the results here to, say $s > 3/4$, although this is
by no means automatic and we will not do so here given that the
argument is already quite lengthy.

Our proof proceeds by the method of almost conservation laws,
sometimes called the ``I-method",  introduced in \cite{keel:wavemap}
and in the earliest versions of the present work (see e.g.
\cite{ckstt:2a} for use of the method in a more straightforward
context than the present one).   The basic idea is to introduce a
special smoothing operator $I=I_N$ of order $1-s$ depending on a
large parameter $N$, and then consider the quantity $\H[I\Phi[t]]$,
which turns out to be finite (but large) for $\Phi[t] \in [H^s]$.
When $s=1$ then $I$ is the identity and $\H[I\Phi[t]]$ is exactly
conserved.  When $s < 1$ we do not have exact conservation, but we
will be able to show (using a modified local well-posedness theory)
that $\H[I\Phi[t]]$ is ``almost conserved'' in that its derivative
is very small (indeed, it will be bounded by a negative power of
$N$).  This will allow us to control the solution for long times (a
positive power of $N$).  Letting $N$ go to infinity we obtain the
result.  Unfortunately the operator $I$ maps $H^s_x$ to $H^1_x$ with
a large operator norm (like $O(N^{1-s})$), and when $s$ is large this
loss of $N^{1-s}$ can overwhelm the almost-conservation of
$\H[I\Phi[t]]$, which is why we have the rather artificial
restriction $s>\sqrt{3}/2$. Subsequent refinements of the
``I-method'' in \cite{ckstt:1}-\cite{ckstt:4} suggest that this
restriction can be lowered by adding additional ``damping
correction'' terms to $\H[I\Phi[t]]$ to reduce the size of the
derivative, but we will not pursue these matters here. Certainly we
do not expect $\sqrt{3}/2$ to be the sharp threshold of global
well-posedness.  (For instance, many, though unfortunately not all,
of the components of the argument are also valid in the regime $s >
5/6$, and some parts are even valid in the range $s>3/4$.)

These results are similar to those of the earlier work of
Bourgain\cite{borg:book} and later authors in obtaining global
well-posedness for nonlinear wave, Schr\"odinger and KdV-type
equations below the energy norm.  However the methods are slightly
different; instead of using a smoothing operator $I$, the method of
Bourgain relies on truncating the solution $u$ at frequency $N$ into
a low frequency component and a high frequency component, and
controlling the evolution of the two components separately (except
for some periodic adjustments at regular intervals).  This approach
gives much better control on the solution (for instance, the method
shows the high frequencies behave almost like the linear flow), but
requires ``extra smoothing estimates'' on the nonlinear component of
the solution, in particular placing that component in the energy
class even when the solution is in a rougher Sobolev space.  For the
equation (MKG-CG)  these extra smoothing estimates are not available
for the worst term in the nonlinearity, namely $\P(\underline{A})
\cdot \nabla_x \phi$, mainly because of the derivative $\nabla_x$;
indeed in the appendix we will give an argument that shows that this
extra smoothing fails for $[H^s]$ solutions for any $s < 1$.
Fortunately, the I-method can circumvent this problem by using
commutator estimates as a substitute for extra smoothing estimates.
See \cite[\S 3.9]{tao} for some further discussion.

\subsection{Organisation of the paper}  After setting some general notation in
Section \ref{notation-sec}, we describe some useful elliptic
estimates for $A_0$ in Section \ref{elliptic-sec}, we are able to
investigate the Hamiltonian in Section \ref{hamil-sec}, and show
that this Hamiltonian largely controls the $[H^1]$ norm of $\Phi$.
This allows us to begin the proof of Theorem \ref{main} in Section
\ref{global-sec}, where we reduce matters to showing a standard
local well-posedness result
 (Theorem \ref{hs-lwp}) as well as an almost conservation law
 (Proposition \ref{main-prop}) for the modified Hamiltonian $\H[I \Phi]$.
  To achieve these tasks, we write the (MKG-CG) equation
in Section \ref{caricature-sec} into a more schematic form which
will be more convenient to manipulate.  After recalling some
$H^{s,b}$ theory in Sections \ref{function-sec}-\ref{bil-sec}, we
are then quickly able to establish the local well-posedness result
in Section \ref{local-sec}.  To prove the almost conservation law,
we will then need  a modified local well-posedness result
(Proposition \ref{i-local}), established in Section \ref{mod-sec}.
We then differentiate the Hamiltonian in Section
\ref{hamildiff-sec}, leading to a number of commutator terms we need
to control.  The terms arising from cubic nonlinearities are
relatively easy and are dealt with in Section \ref{cubic-sec}.  The
terms arising from bilinear nonlinearities are rather complicated
and we shall deal with them \emph{en masse} using some specialized
notation in Section \ref{freq-sec}, which allows us to deal with the
non-null form bilinear terms in Section \ref{n1-sec} and the null
form terms in Section \ref{n0-sec}.

Finally, in the appendix we demonstrate why the system (MKG-CG) does
not have the smoothing effect necessary for Bourgain's Fourier
truncation method \cite{borg:book} to be applicable.

\subsection{Acknowledgements}

The first named author was supported by the NSF, and the Sloan and
McKnight Foundations. The third author is supported by a grant from
the MacArthur Foundation, by NSF grant DMS-0649473, and by the NSF
Waterman award.  We thank the referee for a careful reading of the paper.

\section{Notation}\label{notation-sec}

We fix an exponent $s$, which will usually be in the range
$\sqrt{3}/2 < s < 1$.  We use $C$ to denote various large constants
depending on $s$ and on some other quantities which we will indicate
in the sequel. We use $A \lesssim B$ (or $A = O(B)$) to denote the
estimate $A \leq CB$ and $A \ll B$ to denote the estimate $A \leq
C^{-1} B$. We use $a+$ and $a-$ to denote expressions of the form
$a+\eps$ and $a-\eps$, where $0 < \eps = \eps(s) \ll 1$ denotes a
small number; the implicit constants $C$ referred to above are
allowed to depend on $\eps$.  Note that $a$ may be negative; thus for instance $-\frac{1}{2}+ = -(\frac{1}{2}-)$ is a number slightly larger than $-\frac{1}{2}$.

Given any Banach space $X$ and any injective linear operator $T: X \to Y$,
let $TX$ denote the Banach space $\{ Tu: u \in X \}$ with norm
$$ \| u \|_{TX} := \|T^{-1} u \|_X.$$
If $X$ is a Banach space, we shall use $\Ball(X)$ to denote the unit
ball $\Ball(X) := \{ f \in X: \|f\|_X \leq 1 \}.$  Thus
$r\Ball(X)=\Ball(rX)$ is the ball of radius $r$, $$\Ball(X) +
\Ball(Y) = \Ball(X+Y) = \{ f + g: f \in \Ball(X), g \in \Ball(Y)
\}$$ is the unit ball of the Banach space $X+Y$, $T\Ball(X) =
\Ball(TX)$ is the unit ball of $TX$, etc. We use the embedding
notation
$$ X \subseteq Y$$
to denote the estimate $\Ball(X) \subseteq C \Ball(Y)$, or
equivalently that $f \in Y$ and $\|f\|_Y \lesssim \|f\|_X$ for all
$f \in X$.  We will need this rather unusual notation due to our
reliance on compound spaces $X+Y$ in some of our later arguments.

We use $\hat \phi$ to denote the spatial Fourier transform
$$ \hat \phi(t,\xi) := \int_{\R^3} e^{-i x \cdot \xi} \phi(t,x)\ dx$$
and define fractional derivative operators in the usual manner:
$$ \widehat{|\nabla_x|^s \phi}(t,\xi) := |\xi|^s \hat \phi(t,\xi); \quad
\widehat{\langle\nabla_x\rangle^s \phi}(t,\xi) :=
\langle\xi\rangle^s \hat \phi(t,\xi).$$

We recall the Sobolev multiplication laws on $\R^3$.  Specifically,
we have
\begin{equation}\label{sobolev-mult}
\| \phi \psi \|_{H^s_x} \lesssim \| \phi \|_{H^{s_1}_x} \| \psi
\|_{H^{s_2}_x}
\end{equation}
whenever $s_1 + s_2 \geq 0$, $s \leq \min(s_1,s_2)$ and $s < s_1 +
s_2 - \frac{3}{2}$. See e.g. \cite{tao:xsb}.  Of course, the
implicit constant here depends on $s_1,s_2,s$. A special case of
this inequality is

\begin{lemma}\label{mult}  If $s > 3/4$, then
$\| u v \|_{\dot H^{-1}_x} \lesssim \| u \|_{H^s_x} \| v
\|_{H^{s-1}_x}$.
\end{lemma}

\begin{proof}
By duality it suffices to show that $ \| u w \|_{H^{1-s}} \lesssim
\| u \|_{H^s_x} \| w \|_{\dot H^1_x}$. If $\hat w$ is supported in
the region $|\xi| \gtrsim 1$ then $\dot H^1_x$ is equivalent to
$H^1_x$ and the claim follows from \eqref{sobolev-mult} (since $s >
3/4$).  Thus we assume that $\hat w$ is supported on the region
$|\xi| \ll 1$.

Taking the fractional derivative $\langle \nabla_x \rangle^{1-s}$
and applying the fractional Leibnitz rule (taking Fourier transforms
and assuming $\hat u$, $\hat w$ to be real and non-negative if
desired) we reduce to showing that
$$ \| u \langle \nabla_x \rangle^{1-s} w \|_{L^2_x} \lesssim \| u \|_{H^s_x} \| w \|_{\dot H^1_x}$$
and
$$ \| (\langle \nabla_x \rangle^{1-s} u) w \|_{L^2_x} \lesssim \| u \|_{H^s_x} \| w \|_{\dot H^1_x}$$
Since $w$ only has low frequencies, the operator $\langle \nabla_x
\rangle^{1-s}$ is harmless when applied to $w$, and it will suffice
to prove the latter inequality.  But we can use H\"older to measure
$\langle \nabla_x \rangle^{1-s} u$ in $L^3_x$ and $w$ in $L^6_x$,
and the claim follows from Sobolev and the assumption $3/4 < s \leq
1$.
\end{proof}

We shall also use the Sobolev embeddings $\dot H^{1/2}_x \subseteq
L^3_x$, $\dot H^{3/4}_x \subseteq L^4_x$, $\dot H^1_x \subseteq
L^6_x$, and $\dot H^{3/2-}_x \cap \dot H^{3/2+}_x \subseteq
L^\infty_x$ extremely frequently in the sequel.  Of course these
homogeneous Sobolev embeddings imply various inhomogeneous Sobolev
embeddings, e.g. that $H^s_x \subseteq L^3_x$ whenever $s \geq 1/2$.

As mentioned in the introduction, one of the technical difficulties
with (MKG-CG) is that it is not always possible to control the low
frequency portion\footnote{In the introduction, it was only $A_0$
which had difficulty getting into the $L^2$ norm, and the $[H^s]$
norm allowed us to control $\underline \Phi$ in $L^2$.  However, for
the global existence argument we shall need to rescale the fields
$(A_0, \underline \Phi)$ by a large dilation factor $\lambda$.  This
rescaling is needed to make the (subcritical) Hamiltonian small, but
it also makes the (supercritical) $L^2_x$ norm large.  While it is
possible to continue using the $L^2_x$ norm, it leads to inferior
numerology (in particular, the range of possible $s$ is greatly
reduced) so we shall avoid doing so, using other (non-supercritical)
Lebesgue spaces such as $L^6_x$ as substitutes.}  of the fields
$(A_0, \underline{\Phi})$ satisfactorily in $L^2_x$ norm.  To get
around this we shall estimate the low frequencies in other $L^p_x$
norms.  In this section we develop some of the theory of
frequency-localized Lebesgue spaces.

\begin{definition}
If $1 < p \leq \infty$ and $R > 0$, we define the space $\L^p_R$ to be
the subspace of $L^p(\R^3)$ consisting of those functions whose
Fourier support is contained in the ball $|\xi| \leq R$.  (We keep
the $L^p$ norm structure on this subspace $\L^p_R$.)
\end{definition}

We will use very specific instances of these spaces such as
$\L^6_1$, $\L^3_2$, and $\L^\infty_{10}$.

Observe that if $R$ is bounded, then derivatives are bounded on
$L^p_R$:
\begin{equation}\label{nabla}
\nabla_x \L^p_R \subseteq \L^p_R.
\end{equation}
This is clear since $\nabla$ is equivalent to a standard symbol of
order 0 on frequencies $|\xi| \leq R$.  From this and Sobolev
embedding (or Bernstein's inequality) we see that
\begin{equation}\label{lpr-bernstein}
\L^p_R \subseteq \L^q_{R'}
\end{equation}
whenever $p \leq q$ and $R \leq R'$.

The functions in $\L^p_R$ are thus very smooth (in fact, they are
analytic).  The $p$ exponent thus does not measure regularity, but
instead controls the decay at infinity.

From H\"older's inequality we have
\begin{equation}\label{lpr-holder}
\L^p_R \cdot \L^q_{R'} \subseteq \L^{r}_{R +R'}
\end{equation}
whenever $1/r = 1/p + 1/q$, since the frequency support of a product
is contained in the sum of the frequency supports of the factors.

In particular, if $R$ is bounded, then functions in $\L^p_R$ are
bounded, and so
$$ \L^p_R \cdot L^2 \subseteq L^2.$$
From \eqref{nabla} and the Leibnitz rule we thus have
\begin{equation}\label{lpr-mult}
\L^p_R \cdot H^s_x \subseteq H^s_x
\end{equation}
for all integer $s \geq 0$.  By duality this also holds for integer
$s \leq 0$.  By complex interpolation this thus holds for all real
$s$.

Finally, we prove an ``energy estimate'' for the $\L^p_R$ spaces.
Let us restrict spacetime to a slab $[t_0-\delta, t_0+\delta] \times
\R^3$ for some $0 < \delta \ll 1$ and some time $t_0$.  Suppose $u$
is such that $u[t_0] \in \L^p_R$ and $\Box u \in L^1_t \L^p_R$, with
$R$ bounded.  Then we have $\nabla_{x,t} u \in C^0_t \L^p_R$ with
\begin{equation}\label{lpr-energy}
\| \nabla_{x,t} u \|_{C^0_t \L^p_R} \lesssim \| u[t_0] \|_{\L^p_R} +
\| \Box u \|_{L^1_t \L^p_R}.
\end{equation}
Indeed, this follows from the Duhamel formula
$$ u(t) = \cos((t-t_0) \sqrt{-\Delta}) u(t_0) + \frac{\sin((t-t_0) \sqrt{-\Delta})}{\sqrt{-\Delta}} u_t(t_0) + \int_{t_0}^t
\frac{\sin((t'-t_0) \sqrt{-\Delta})}{\sqrt{-\Delta}} \Box u(t')\
dt'$$ and the fact that $\nabla_x$, $\cos((t-t_0) \sqrt{-\Delta})$
and $\frac{\sin((t-t_0) \sqrt{-\Delta})}{\sqrt{-\Delta}}$ are
equivalent to symbols of order 0 for frequencies $\leq R$ and times
$t \in [t_0-\delta, t_0+\delta]$.

\section{The elliptic theory of $A_0$}\label{elliptic-sec}

In this section we develop some elliptic theory for how the
connection component $A_0$ depends on $\phi$ and $A$.  We shall
establish a smoothing effect that allows us to place $A_0[t]$ in
$\dot H^1_x \times L^2_x$ even if $\phi[t]$ and $A[t]$ are merely in
$H^s_x \times H^{s-1}_x$ for some $3/4 < s \leq 1$.

The equations \eqref{a0-eq}, \eqref{a0t-eq} for a fixed time $t$ can
be rewritten as
\begin{equation}\label{a0-alt}
\begin{split}
(-\Delta + |\phi|^2) A_0 &= \Im(\phi \overline{\phi_t}) \\
\nabla_x A_{0,t} &= - (1-\P) \Im (\phi \overline{\nabla_x \phi}) +
(1-\P) (\underline{A} |\phi|^2).
\end{split}
\end{equation}
We view \eqref{a0-alt} as a linear elliptic system for two unknown
fields $A_0, A_{0,t}$ in terms of data $\phi, \phi_t, A$; the $t$
subscript here should be viewed as simply a label, thus $A_{0,t}$
and $\phi_t$ are not being interpreted here as the time derivatives
of $A_0$ or $\phi$.

Our main result here is as follows.

\begin{proposition}[$A_0$ estimates]\label{a0-est}  Let $3/4 < s \leq 1$, let $\phi, \underline{A} \in H^s_x$, and $\phi_t \in H^{s-1}_x$ (we do not assume here that $\phi, \underline{A}, \phi_t$ solve (MKG-CG) or obey any compatibility conditions).  Then there exists unique $A_0 \in H^1_x$ and $A_{0,t} \in L^2_x$ obeying \eqref{a0-alt}, and we also have the bounds
\begin{equation}\label{cz}
\| A_0 \|_{\dot H^1_x} \lesssim \| \phi \overline{\phi_t} \|_{\dot
H^{-1}_x}
\end{equation}
and
\begin{equation}\label{nab-a0}
\| A_0 \|_{\dot H^1_x}, \| A_{0,t} \|_{L^2_x} \lesssim \sum_{j=2}^3
(\| \phi \|_{H^s_x} + \| \underline{A} \|_{H^s_x} + \| \phi_t
\|_{H^{s-1}_x})^j.
\end{equation}
Furthermore, if $\phi', \underline{A}' \in H^s_x$ and
$\underline{A}'_t \in H^{s-1}_x$, and $A'_0, A'_{0,t}$ are the
associated solutions to \eqref{a0-alt}, then we have the local
Lipschitz bound
\begin{equation}\label{a0-diff}
\| A_0 - A'_0 \|_{\dot H^1_x}, \| A_{0,t} - A'_{0,t} \|_{L^2_x}
\lesssim M^5 [ \| \phi - \phi' \|_{H^s_x} + \|\underline{A} -
\underline{A}'\|_{H^s_x} + \| \phi_t - \phi'_t \|_{H^{s-1}_x} ]
\end{equation}
where
\begin{equation}\label{m-def}
M := 1 + \| \phi \|_{H^s_x} + \| \underline{A} \|_{H^s_x} + \|
\phi_t \|_{H^{s-1}_x} + \| \phi' \|_{H^s_x} + \| \underline{A}'
\|_{H^s_x} + \| \phi'_t \|_{H^{s-1}_x}.
\end{equation}
\end{proposition}

\begin{remark}
If $\phi$ is suitably small (e.g. if $\| \phi \|_{L^3_x} \ll 1$)
then we can iterate away the linear term $A_0 |\phi|^2$, and
controlling $A_0$ in terms of $\phi$ is straightforward.  However we
do not assume any smallness condition on $\phi$, and so we must
proceed with some care.  In particular, we must augment perturbation
theory with some variational methods\footnote{Alternatively, one
could localize in space and use standard elliptic regularity theory.
However this requires the usage of local Coulomb gauges, which have
their own attendant technicalities, see e.g.
\cite{klainerman:yangmills}, \cite{uhlenbeck}.}.
\end{remark}

\begin{proof}
The Schr\"odinger operator $-\Delta + |\phi|^2$ maps $H^1_x$ to
$H^{-1}_x$ (using \eqref{sobolev-mult}), and is clearly positive
definite.  From Lemma \ref{mult} we have $\Im(\phi
\overline{\phi_t}) \in H^{-1}_x$.  From \eqref{a0-alt} we thus
conclude that $A_0$ is unique as claimed.  The uniqueness for
$A_{0,t}$ is obvious.

To prove the remaining claims in the proposition, it suffices by the
usual density arguments to verify the case when $\phi$, $\phi_t$,
$\underline{A}$ are smooth and rapidly decreasing in space, which we
shall now assume throughout.

From \eqref{a0-alt} and standard Euler-Lagrange theory, we see that
$A_0$ can be now be constructed as the unique minimizer in $H^1_x$
of the convex functional\footnote{The expression \eqref{a0-minimize}
can also be interpreted as the component of the Hamiltonian
\eqref{hamil-def} which depends on $A_0$.}
\begin{equation}\label{a0-minimize}
L_{\phi,\phi_t}(A_0) := \frac{1}{2} \int_{\R^3} |\nabla_x A_0|^2 +
|\phi_t + i A_0 \phi|^2\ dx.
\end{equation}
This gives existence of $A_0$.  The existence of $A_{0,t}$ is clear
from Hodge theory, since the right-hand side of the second equation
in \eqref{a0-alt} is curl-free.  We shall now use this variational
formulation to establish the bounds \eqref{nab-a0}, \eqref{a0-diff}.

At first glance it seems we are in trouble when $s<1$ because
\eqref{a0-minimize} cannot be controlled by the $H^s_x \times
H^{s-1}_x$ norm of $\phi[t]$.  However, we may ``renormalize''
$L_{\phi,\phi_t}(A_0)$ by defining
\begin{align*}
\tilde L_{\phi,\phi_t}(A_0) &:= L_{\phi,\phi_t}(A_0) - L_{\phi,\phi_t}(0) \\
&= \int_{\R^3} \frac{1}{2} |\nabla_x A_0|^2 + A_0 \Im( \phi
\overline{\phi_t} ) + \frac{1}{2} |A_0|^2 |\phi|^2.
\end{align*}
In particular, since $\tilde L_{\phi,\phi_t}(A_0) \leq 0$, we have
$$ \int_{\R^3} |\nabla_x A_0|^2
\lesssim \left|\int_{\R^3} A_0 \Im( \phi \overline{\phi_t})\right|
$$ so by Cauchy-Schwarz we have \eqref{cz}. From Lemma \ref{mult}
see that $A_0$ verifies \eqref{nab-a0}.  The claim for $A_{0,t}$ is
much simpler, following easily from \eqref{a0-eq}, Lemma \ref{mult},
Sobolev and H\"older.

It remains to establish \eqref{a0-diff}.  We fix
$\phi,A,\phi_t,\phi',A',\phi'_t$ (and hence $A_0, A'_0$).  From
\eqref{nab-a0} we have
\begin{equation}\label{ah1}
\| A_0 \|_{\dot H^1_x} + \| A'_0 \|_{\dot H^1_x} \lesssim M^2.
\end{equation}
If we write $A_0 = A'_0 + h$ and $A_{0,t} = A'_{0,t} + h_t$, our
task is to show that
\begin{equation}\label{abomb}
\| h \|_{\dot H^1_x} + \| h_t \|_{L^2_x} \lesssim M^5 ( \| \phi -
\phi' \|_{H^s_x} + \|A-A'\|_{H^s_x} + \| \phi_t - \phi'_t
\|_{H^{s-1}_x} ).
\end{equation}

We begin with the estimation of $\| h \|_{\dot H^1_x}$. From the
variational characterisation of $A_0$ we have
$$ \tilde L_{\phi', \phi'_t}(A'_0) -  \tilde L_{\phi', \phi'_t}(A_0) \leq 0$$
and thus by the triangle inequality
\begin{equation}
\tilde L_{\phi,\phi_t}(A'_0) - \tilde L_{\phi,\phi_t}(A_0) \lesssim
|\tilde L_{\phi,\phi_t}(A_0) - \tilde L_{\phi',\phi'_t}(A_0)| +
|\tilde L_{\phi,\phi_t}(A'_0) - \tilde L_{\phi',\phi'_t}(A'_0)|.
\end{equation}
Since
$$ \tilde L_{\phi,\phi_t}(A'_0) - \tilde L_{\phi,\phi_t}(A_0) = \frac{1}{2} \int_{\R^3}|\nabla_x h|^2 + |h|^2 |\phi|^2 \gtrsim \| h \|_{\dot H^1_x}^2$$
we thus see that it suffices to show that
\begin{equation}\label{l-claim}
|\tilde L_{\phi,\phi_t}(A_0) - \tilde L_{\phi',\phi'_t}(A_0)|
\lesssim M^5 \| \phi[t] - \phi'[t] \|_{H^s_x \times H^{s-1}_x}
\end{equation}
and similarly for $A_0$ replaced by $A'_0$.  Using the definition of
$\tilde L$, we may estimate the left-hand side of \eqref{l-claim} by
$$ \lesssim \int_{\R^3} |A_0| |\phi \overline{\phi_t} - \phi' \overline{\phi'_t}|
+ |A_0|^2 \left| |\phi|^2 - |\phi'|^2 \right|.$$ Splitting $\phi
\overline{\phi_t} - \phi' \overline{\phi'_t}$ as a sum of two
differences and using \eqref{m-def} and Lemma \ref{mult}, we obtain
$$ \| \phi \overline{\phi_t} - \phi' \overline{\phi'_t}\|_{\dot H^{-1}_x} \lesssim M \| \phi[t] - \phi'[t] \|_{H^s_x \times H^{s-1}_x}.$$
Also, from \eqref{ah1} and Sobolev we have
$$ \| |A_0|^2 \|_{L^3_x} = \| A_0\|_{L^6_x}^2 \lesssim \|A_0 \|_{\dot H^1_x}^2 \lesssim M^4$$
and from Sobolev we have
$$ \| |\phi|^2 - |\phi'|^2\|_{L^{3/2}_x} \lesssim (\| \phi \|_{L^3_x} + \| \phi'\|_{L^3_x})
\| \phi - \phi' \|_{L^3_x} \lesssim M \| \phi - \phi' \|_{H^s_x} $$
and the claim \eqref{l-claim} follows.    This yields the desired
bound \eqref{abomb} for $\| h \|_{\dot H^1_x}$.  The analogous claim
for $h_t$ follows from \eqref{a0-eq}, H\"older, Sobolev, and Lemma
\ref{mult} as before.  This gives \eqref{abomb} and hence
\eqref{a0-diff} as desired.
\end{proof}

\begin{remark} Heuristically, Proposition \ref{a0-est} allows us to eliminate $A_0$ from (MKG-CG), and think of this system as an evolution purely in $\underline \Phi$.  Indeed from the above analysis one morally has $A_0 \approx \Delta^{-1} (\underline \Phi \nabla_x \underline \Phi)$.  However we shall keep $A_0$ explicit in our computations.
\end{remark}

\begin{remark} Proposition \ref{a0-est} asserts that $A_0$ is somewhat smoother than $\underline \Phi$: it is in $\dot H^1_x$ even though $\phi$ is merely in $H^s_x$.  However we cannot place $A_0$ in $H^1_x$ or even in $L^2_x$ because of the slow decay of $A_0$ at infinity mentioned earlier.
\end{remark}

\begin{remark}\label{smooth-remark} Proposition \ref{a0-est} implies that rough initial data
$\Phi[0] \in [H^s]$  (see \eqref{hsbracket}) obeying \eqref{compat}
can be approximated arbitrarily closely in $[H^s]$ norm by smooth
initial data $\Phi'[0]$, also obeying \eqref{compat}.  We sketch the
argument as follows.  Given $\Phi[0] \in [H^s]$ we can first
approximate $\underline \Phi[0]$ in $H^s_x \times H^{s-1}_x$ norm by
a smooth $\underline \Phi'[0]$ which still obeys the divergence-free
condition $\nabla_x \cdot \underline A'[0] = 0$.  Then we construct
$A'_0(0)$ as above, and $\partial_t A'_0(0)$ by \eqref{a0t-eq}.
From \eqref{a0-diff} we see that $\nabla_{x,t} A'_0(0)$ is close to
$\nabla_{x,t} A_0(0)$ in $L^2_x$, and hence in $H^{s-1}_x$.
Combining all these estimates we thus see that $\Phi'[0]$ is close
to $\Phi[0]$ in $[H^s]$ norm as desired.
\end{remark}

\section{Fixed-time Hamiltonian estimates}\label{hamil-sec}

Using the elliptic theory for $A_0$ and the machinery of
frequency-localized spaces, we are now ready to understand  the
Hamiltonian \eqref{hamil-def}.

From \eqref{hamil-def} and the triangle inequality we have
$$ \H[\Phi[t]]
\lesssim \| \nabla_x A_0(t) \|_{L^2_x}^2 + \| \nabla_{x,t}
\underline{A}(t) \|_{L^2_x}^2 + \| \nabla_{x,t} \phi(t) \|_{L^2_x}^2
+ \| A_0(t) \phi(t) \|_{L^2_x}^2 + \| \underline{A}(t) \phi(t)
\|_{L^2_x}^2.$$ From H\"older's inequality and the Sobolev embedding
$\dot H^1_x \subseteq L^6_x$ we thus have
\begin{equation}\label{h-bound}
\H[\Phi[t]] \lesssim (\| \nabla_x A_0(t) \|_{L^2_x}^2 + \|
\nabla_{x,t} \underline{\Phi}(t) \|_{L^2_x}^2) (1 + \| \phi(t)
\|_{L^3_x}^2).
\end{equation}

This allows us to control the Hamiltonian by the $[H^1]$ norm.

Now we look at the converse: given that the Hamiltonian is finite,
what bounds can we place on $ \Phi$?  This question was studied in
\cite{kl-mac:mkg}, however in that paper some $L^2_x$ control on
$\phi$ and $A$ was also assumed at time zero.  We will not be able
to use such control as the $L^2_x$ norm is
supercritical\footnote{One might consider adding a mass term to
(MKG-CG) and posing the same global well-posedness questions.  It
seems likely that one has similar results for the massive (MKG-CG),
however the argument would be technically more complicated due to
the Schr\"odinger-like behaviour of low frequencies.  Also, the mass
term is still supercritical and so this does not solve the
difficulties of using the $L^2$ norm.} and so will behave badly with
respect to a rescaling argument which we will use later.
Fortunately, we can still obtain good control on $\Phi$ without the
$L^2_x$ norm, although some odd things happen at low frequencies.

\begin{lemma}[Hamiltonian controls $H^1$]\label{hamil-control}  Let $t \in \R$ be fixed. Suppose that $A_0(t)$, $\underline{\Phi}[t] = (\underline A(t), \phi[t])$ are $[H^1]$ functions such that $\nabla_x \cdot \underline A[t] = 0$ and $\H[\Phi[t]] \lesssim 1$.
Then we have the estimates
\begin{align}
\Phi(t) &\in C \Ball(H^1_x + \L^6_1)\label{phi0}\\
\nabla_{x,t} \underline{\Phi}(t) &\in C\Ball(L^2_x + \L^3_2) \label{gradphi0} \\
\nabla_x A_0(t), \nabla_{x,t} \underline{A}(t) &\in C \Ball(L^2_x)
\label{a}.
\end{align}
\end{lemma}

Informally, control of the Hamiltonian allows one to place most of
$A_0(t)$ and $\underline \Phi[t]$ in $H^1_x \times L^2_x$, except
for the low frequency component, which is only in $L^6_x$ or
$L^3_x$.  The hypothesis that $A_0, \underline{\Phi}$ are $[H^1]$ functions is a purely qualitative hypothesis; the constants $C$ do not depend on the $[H^1]$ norms of these functions.

\begin{proof}
From the hypothesis and \eqref{hamil-def} we have
\begin{align}
\nabla_x A_0(t) - \partial_t \underline{A}(t) & \in C\Ball(L^2_x) \label{a0-energy}\\
\nabla_x \times \underline A(t) & \in C\Ball(L^2_x) \label{aj-energy}\\
D_0 \phi(t) & \in C\Ball(L^2_x) \label{phi0-energy}\\
\underline{D} \phi(t) & \in C\Ball(L^2_x). \label{phij-energy}
\end{align}

From \eqref{aj-energy} and the hypothesis $\nabla_x \cdot
\underline{A}(t) = 0$ we have $
 \underline A(t) \in C \Ball(\dot H^1_x).
$ Also, by taking divergence-free and curl-free components of
\eqref{a0-energy} using the hypothesis $\nabla_x \cdot \partial_t
\underline{A}(t) = 0$ we see that
$$ \nabla_x A_0(t), \partial_t A_j(t) \in C \Ball(L^2_x).$$
Combining these estimates together we obtain \eqref{a}. Using the
embedding $\dot H^1_x \subseteq H^1_x + \L^6_1$ which comes from
applying Sobolev embedding to the low frequencies of $\dot H^1_x$,
we thus see that $A_0$ and $\underline{A}$ satisfy the required
estimates \eqref{phi0}, \eqref{gradphi0}.

It remains to show the corresponding estimates for $\phi$.  We begin
with the pointwise identity
$$ 2 |\phi(t)| \partial_j |\phi(t)| = \partial_j( |\phi(t)|^2 ) = 2 \Re(\phi(t) \overline{\partial_j \phi(t)}) = 2 \Re(\phi(t) \overline{D_j \phi(t)})$$
for $j=1,2,3$.  In particular we have the ``diamagnetic inequality''
$$ |\partial_j |\phi(t)|| \leq |D_j \phi(t)|.$$
From \eqref{phij-energy}, the Sobolev embedding $\dot H^1_x
\subseteq L^6_x$, and the trivial observation that $|\phi(t)|$ and
$\phi(t)$ have the same $L^6_x$ norm we thus have
\begin{equation}\label{l6-bound}
\phi(t) \in C \Ball(L^6_x).
\end{equation}
Also, from our estimates on the $A_j$ and $A_0$ and Sobolev
embedding we have
\begin{equation}\label{aj-6}
A_0(t), A_j(t) \in C \Ball(L^6_x).
\end{equation}
By H\"older we thus have
$$ A_0(t) \phi(t), A_j(t) \phi(t) \in C \Ball(L^3_x).$$
Combining this with \eqref{phi0-energy}, \eqref{phij-energy} we
obtain
\begin{equation}\label{phi-23}
\partial_t \phi(t), \partial_j \phi(t) \in C\Ball(L^2_x + L^3_x).
\end{equation}
On the other hand, if we take the divergence of \eqref{phij-energy}
using the hypothesis $\nabla_x \cdot \underline A[t] = 0$ we obtain
$$ \Delta \phi(t) + i A_j(t) \partial_j \phi(t) \in C\Ball(\dot H^{-1}_x).$$
From \eqref{phi-23}, \eqref{aj-6} and H\"older we have
$$ i A_j(t) \partial_j \phi(t) \in C\Ball(L^{3/2}_x + L^2_x).$$
From Sobolev we have $L^{3/2}_x, L^2_x, \dot H^{-1}_x \subseteq
H^{-1}_x$.  From this and the previous we thus have
$$ \Delta \phi(t) \in C \Ball(H^{-1}_x).$$
We now divide $\phi(t)$ smoothly into a low frequency component
supported on $|\xi| \leq 1$, and the remainder supported on $|\xi|
\geq 1/2$.  From \eqref{l6-bound} and the above equation we see that
the low frequency part is in $\L^6_1$ and the remainder is in
$H^1_x$, so $\phi(t)$ obeys \eqref{phi0}.

It only remains to show that $\partial_t \phi(t)$ obeys
\eqref{gradphi0}.  From \eqref{phi0-energy} it suffices to show that
$$ A_0(t) \phi(t) \in C\Ball(L^2_x + \L^3_2).$$
Since we have already shown that $A_0(t)$, $\phi(t)$ obey
\eqref{phi0}, the claim then follows from \eqref{sobolev-mult},
\eqref{lpr-mult}, and \eqref{lpr-holder}.
\end{proof}

\section{Global well-posedness: preliminary reduction}\label{global-sec}

We now begin the proof of Theorem \ref{main}.  Fix $\sqrt{3}/2 < s <
1$, and fix the initial data $\Phi[0]$ obeying the hypotheses of the
theorem.  Let $T_*$ denote the maximal time of existence for which
one can construct a solution $\Phi$ in $[H^s]$; our objective is to
show that $T_*$ is infinite.

In Section \ref{local-sec} we shall prove the following local
well-posedness result (essentially due to Cuccagna \cite{cuccagna}):

\begin{theorem}[Local well-posedness]\label{hs-lwp} Let $5/6 < s \leq 1$ and $M > 0$.  Let $\Phi[0] \in M\Ball([H^s])$ obey the compatibility conditions \eqref{compat}.  Then there exists a solution $\Phi$ to (MKG-CG) with initial data $\Phi[0]$ on the time interval $[-T,T]$ for some $T = T(M) > 0$.  Furthermore, for each time $t \in [-T,T]$ the solution map $\Phi[0] \mapsto \Phi[t]$ is a Lipschitz map from $M\Ball([H^s])$ to $[H^s]$ (with the Lipschitz constant depending on $s$ and $M$).
\end{theorem}

\begin{remark} In view of the work of Cuccagna\cite{cuccagna}, the local existence theorem here should in fact extend to the range $s>3/4$, and a possibly weakened version of this local existence theorem should also hold in the range $s > 1/2$ thanks to the work of Machedon and Sterbenz\cite{machedon}.  However, to avoid technicalities we will restrict ourselves to the case $s>5/6$ (which covers the range $s > \sqrt{3}/2$ that our main theorem covers).
\end{remark}

Assume this theorem for the moment.  Then we have $T_* > 0$.
Furthermore, if $T_*$ is finite, then Theorem \ref{hs-lwp} forces
one to have $\lim_{t \to T_*} \| \Phi[t] \|_{[H^s]} = +\infty$.
Thus if we can prove the polynomial growth bound \eqref{poly} for $t
< T_*$, we will have obtained global well-posedness.

By another application of Theorem \ref{hs-lwp} and a standard
limiting argument (using Remark \ref{smooth-remark}) we may assume
that $\Phi[0]$ is smooth and $[H^1]$, in which case we have a global
smooth and $[H^1]$ solution from the results in the introduction
(\cite{eardley, kl-mac:mkg}). Thus it will suffice to prove
\eqref{poly} for global smooth solutions.

Henceforth our constants $C$ are allowed to depend on $\| \Phi[0]
\|_{[H^s]}$, thus for instance $ \| \Phi[0] \|_{[H^s]} \lesssim 1$.

Fix the time $T$ in \eqref{poly}.  In view of Theorem \ref{hs-lwp}
we may assume $T \gtrsim 1$. As is usual in applications of the
$I$-method, we will need to rescale the equation using
\eqref{mkg-scaling}, replacing $\Phi$ by the rescaled solution
$$ \Phi^{(\lambda)}(t,x) := \frac{1}{\lambda} \Phi(\frac{t}{\lambda}, \frac{x}{\lambda})$$
for some large $\lambda = \lambda(T) \gg 1$ to be chosen later. Note
that $\Phi^{(\lambda)}$ also solves (MKG-CG).  In order to obtain
\eqref{poly} at time $T$ we will need to control $\Phi^{(\lambda)}$
at time $\lambda T$.

We would like to use the Hamiltonian $\H[\Phi^{(\lambda)}[t]]$
defined in \eqref{hamil-def}.  Unfortunately we do not have enough
regularity on $\underline{\Phi}$ to ensure this Hamiltonian is
finite since $s < 1$.  (On the other hand, $A_0$ has enough
regularity thanks to \eqref{nab-a0}.)  To get around this difficulty
we shall use the method of almost conservation laws. 

We pick a large number $N = N(T) \gg 1$ to be specified later.  Let
$m(\xi)$ be a smooth radial positive symbol such that $m(\xi) = 1$
for $|\xi| \leq N$ and $m(\xi) = |\xi|^{s-1}/N^{s-1}$ for $|\xi| >
2N$, and let $I$ be the Fourier multiplier
$$ \widehat{If}(\xi) := m(\xi) \hat f(\xi).$$
Thus $I$ is the identity for bounded\footnote{We have two frequency
cutoffs in our argument, one at 1 and one at $N$.  To avoid
confusion as to what ``low'' and ``high'' frequency are, we refer to
frequencies $|\xi| \lesssim 1$ as \emph{low}, frequencies $1 \ll
|\xi| \ll N$ as \emph{medium}, and frequencies $|\xi| \gtrsim N$ as
\emph{high}.  We will also refer to frequencies $|\xi| \ll N$ as
\emph{bounded}, and frequencies $|\xi| \gg 1$ as \emph{local}.}
frequencies $|\xi| \ll N$ and is smoothing of order $1-s$ for high
frequencies $|\xi| \gtrsim N$.  Observe that the convolution kernel
of $I$ is integrable, thus $I$ is bounded on every
translation-invariant Banach space.  Furthermore, we have the
smoothing estimates
\begin{equation}\label{i-loss}
\| u \|_{H^s_x} \lesssim \| I u \|_{H^1_x} \lesssim N^{1-s} \| u
\|_{H^s_x}
\end{equation}
\begin{equation} \label{i-loss-hom}
\|I u \|_{\dot{H}^1_x} \lesssim N^{1-s} \| u \|_{\dot{H}^s_x}.
\end{equation}
We will use $\H[I\Phi^{(\lambda)}[t]]$ as a substitute for the
Hamiltonian.  Unfortunately, the loss of $N^{1-s}$ in \eqref{i-loss}
would make the modified Hamiltonian large.  However, the scaling
parameter $\lambda$ (combined with the fact that the energy
regularity $H^1_x$ is sub-critical) can be used to rescale the
Hamiltonian to be small again.  More precisely, we have

\begin{lemma}[Rescaled Hamiltonian is small]\label{small-hamil}  Suppose we choose $N \gg 1$, $\lambda \gg 1$ so that
\begin{equation}\label{n-lambda}
\lambda^{1/2-s} N^{1-s} \ll 1.
\end{equation}
Then we have $\H[I\Phi^{(\lambda)}[0]] \leq 1$.
\end{lemma}

\begin{proof}
Observe that
$$ \| I\phi^{(\lambda)}(0) \|_{L^3_x} \lesssim \| \phi^{(\lambda)}(0) \|_{L^3_x} = \| \phi(0) \|_{L^3_x} \lesssim \| \phi(0) \|_{H^s_x} \lesssim 1$$
so by \eqref{h-bound} it will suffice to show that $\| \nabla_x I
A_0^{(\lambda)}(0) \|_{L^2_x} \ll 1$ and $\|\nabla_{x,t} I
\underline{\Phi}^{(\lambda)}(0)\|_{L^2_x} \ll 1$. The former
estimate is easy, in fact by \eqref{i-loss-hom}, \eqref{nab-a0} we
have
$$ \| \nabla_x I A_0^{(\lambda)} \|^{2}_{L^2_x} \lesssim N^{2(1-s)}
\lambda^{1-2s} \| \Phi[0] \|_{[H^s]}^2.$$ For the latter estimate we
use \eqref{i-loss-hom}:
$$ \| \nabla_{x,t} I \underline{\Phi}^{(\lambda)}(0) \|_{L^2_x} \lesssim
N^{1-s} \| \nabla_{x,t} \underline{\Phi}^{(\lambda)}(0)
\|_{\dot{H}^{s-1}_x} \lesssim N^{1-s} \lambda^{1/2-s} \|
\nabla_{x,t} \underline{\Phi}(0) \|_{\dot{H}^{s-1}_x} $$

and the claim follows from \eqref{n-lambda} since $\| \Phi[0]
\|_{[H^s]} \lesssim 1$.
\end{proof}

Thus we can make the modified Hamiltonian small at time zero.  In
order to control the modified Hamiltonian at later times we use the
following almost conservation law for the modified Hamiltonian:

\begin{proposition}[Almost conservation law]\label{main-prop}  Let $(A_0, \underline \Phi)$ be a global smooth solution to (MKG-CG), and suppose $t_0$ is a time such that
\begin{equation}\label{hi-eps}
\H[I\Phi[t_0]] \leq 2.
\end{equation}
Then we have
\begin{equation}\label{h-shift}
\H[I\Phi[t]] = \H[I\Phi[t_0]] + O \left( \frac{1}{N^{ ( s - 1/2 )-}}
\right)
\end{equation}
for all $t \in [t_0 - \delta/2, t_0 + \delta/2]$, where $N^{0-} \ll
\delta \ll 1$ is a small constant depending only on $s$ and $N$.
\end{proposition}

\begin{remark}
The error of $O \left( \frac{1}{N^{ ( s - 1/2 )-}} \right)$
corresponds to the restriction $s > \sqrt{3}/2$, but is not optimal.
In particular it seems
feasible that one could improve this error to $O(N^{-1/2+})$, which
would in principle allow us to obtain global well-posedness for $s >
5/6$. For the equation (NLW), an error of
$O(N^{-1+})$ is attainable, which corresponds to the regularity
$s>3/4$ (cf. \cite{kpv:gwp}). By combining this conservation law
with additional techniques, the global well posedness of (NLW) was
extended to the range $s>13/18$ in \cite{roy1}, with a further gain
to $s>7/10$ in \cite{roy2} (in the spherically symmetric case).
\end{remark}

The proof of Proposition \ref{main-prop} is rather lengthy and
(together with Theorem \ref{hs-lwp}) will occupy Sections
\ref{mod-sec}-\ref{n0-sec}.  For now, we see how this proposition
implies \eqref{poly} and hence Theorem \ref{main}.

From Lemma \ref{small-hamil} and $O(N^{(s- 1/2)-})$ applications of
Proposition \ref{main-prop} we can obtain the estimate
\begin{equation}\label{sup-hamil}
\sup_{0 \leq t \leq T} \H[I\Phi^{(\lambda)}[\lambda t]] \leq 2
\end{equation}
provided that
\begin{equation}\label{lt}
\lambda T \ll N^{(s- 1/2)-}.
\end{equation}
A little algebra shows that we can choose $\lambda \gg 1$, $N \gg 1$
so that \eqref{lt} and \eqref{n-lambda} simultaneously hold, so long
as $s> \sqrt{3}/2$.  Furthermore, both $\lambda$ and $N$ are at most
polynomial in $T$.

To finish up we use an integration in time argument inspired by a similar argument from \cite{klainerman:nulllocal}.
From \eqref{sup-hamil} and Lemma \ref{hamil-control} we have that
\begin{equation}\label{nabla-I}
\nabla_{x,t} I \underline \Phi^{(\lambda)}(t) \in C \Ball(L^2_x +
\L^3_2)
\end{equation}
for all $0 \leq t \leq T$ (note that $I \underline A^{(\lambda)}$ is
automatically divergence free).  Also, from \eqref{i-loss} we have
\begin{equation}\label{phi-0}
I \underline \Phi^{(\lambda)} (0) \in CT^C \Ball(H^1_x) \subseteq
CT^C \Ball(L^2_x).
\end{equation}
Using the fundamental theorem of calculus
$$ I \underline \Phi^{(\lambda)} (t) = I \underline \Phi^{(\lambda)} (0) + \int_0^t \partial_t I \underline \Phi^{(\lambda)} (t')\ dt'$$
we thus see that
$$ I \underline \Phi^{(\lambda)} (t) \in CT^C \Ball(L^2_x + \L^3_2).$$
Splitting $\Phi$ smooth into low frequencies $|\xi| \leq 4$ and a
remainder term $|\xi| \geq 2$ and using \eqref{nabla-I} we obtain
$$ I \underline \Phi^{(\lambda)} (t) \in CT^C \Ball(H^1_x + \L^3_2).$$
On the other hand, from Lemma \ref{hamil-control} we have
$$ A_0^{(\lambda)} (t) \in C\Ball(H^1_x + \L^6_1).$$
Multiplying the two using the Sobolev embedding $H^1_x \subset L^4$,
we obtain
$$ A_0^{(\lambda)}(t) I \underline \Phi^{(\lambda)}(t) \in CT^C \Ball(L^2_x).$$
But from \eqref{sup-hamil}, \eqref{hamil-def} we have
$$ \partial_t I \phi^{(\lambda)} (t) + i A^{(\lambda)}_0(t) I \phi^{(\lambda)} (t) \in CT^C \Ball(L^2_x)$$
and hence
$$ \partial_t I \phi^{(\lambda)} (t) \in CT^C \Ball(L^2_x).$$
Combining this with \eqref{a} we obtain
$$ \partial_t I \underline \Phi^{(\lambda)}(t) \in CT^C \Ball(L^2_x).$$
Applying the fundamental theorem of calculus and \eqref{phi-0} again
we obtain
$$ I \underline \Phi^{(\lambda)} (\lambda T) \in CT^C \Ball(L^2_x).$$
Combining this with \eqref{nabla-I} again we obtain
$$ I \underline \Phi^{(\lambda)} [\lambda T] \in CT^C \Ball(H^1_x \times L^2_x)$$
which implies from \eqref{i-loss} that
$$ \underline \Phi^{(\lambda)} [\lambda T] \in CT^C \Ball(H^s_x \times H^{s-1}_x).$$
Combining this with \eqref{nab-a0} we obtain
$$ \nabla_{x,t} \Phi^{(\lambda)}(\lambda T) \in CT^C \Ball(L^2_x).$$
Undoing the scaling we thus obtain \eqref{poly} as desired.  This
proves Theorem \ref{main}.

It remains to show Theorem \ref{hs-lwp} and Proposition
\ref{main-prop}.  This will occupy the remainder of the paper.

\section{A caricature for MKG-CG}\label{caricature-sec}

The system (MKG-CG) may appear excessively complicated, due to
vector structures, Riesz transforms, complex conjugates, and
constants such as $2i$.  To clean up some of the clutter we shall
adopt some notational conventions to reduce (MKG-CG) to a
``caricature'' form, which we will then use to prove both Theorem
\ref{hs-lwp} and Proposition \ref{main-prop}.

We adopt the convention that if $A$ is a scalar-, vector-, or
tensor-valued quantity, then $\O( A )$ denotes an expression which
is \emph{schematically} of the form $A$, or more precisely a finite
linear combination of expressions of the form $T_{i} \Re(A_i)$ and
$T'_i \Im(A_i)$, where $A_i$ denotes the various components of $A$
and $T_i$, $T'_i$ either denote constants or Riesz transforms (which
arise due to the presence of the Leray projection $\P$).  We recall
the well-known fact from Calder\'on-Zygmund theory (see e.g.
\cite{stein:small}) that these operators are bounded on $L^p_x$ for
every $1 < p < \infty$. We can then define quadratic schematic
expressions $\O( AB)$ and cubic ones $\O(A B C)$ by using the
convention that $AB$ denotes the tensor product of $A$ and $B$
(viewed as real tensors rather than complex, thus for instance
$\Re(A) \Im(B) = \O(AB)$), etc.  For example, we have
\begin{align*}
\Im( \phi \overline{D_0 \phi} ) &= \O( \phi \partial_t \phi ) + \O( A_0 \phi \phi )\\
&= \O( \underline \Phi \nabla_{x,t} \underline \Phi ) + \O( \Phi \Phi \Phi ) \\
(1-\P) (\Im \phi \overline{\nabla_x \phi} ) &= \O( \phi \nabla \phi ) \\
&= \O( \underline \Phi \nabla_{x,t} \underline \Phi ) \\
(1-\P) (A |\phi|^2 ), \P (A |\phi|^2), |\underline A|^2 \phi, |A_0|^2 \phi &= \O( \Phi \Phi \Phi ) \\
2i A_0 \partial_t \phi &= \O( A_0 \partial_t \underline \Phi ) \\
i(\partial_t A_0) \phi &= \O( (\partial_t A_0) \underline \Phi ) \\
\end{align*}
and we can therefore rewrite (MKG-CG) in the caricature form
\begin{equation}\label{caricature}
\begin{split}
\nabla_x \nabla_{x,t} A_0 & = \O(\N_2) + \O(\N_3)\\
\Box \underline \Phi & = \O(\N_0) + \O(\N_1) + \O(\N_3)
\end{split}
\end{equation}
where the bilinear and trilinear nonlinearities $\N_0$, $\N_1$,
$\N_2$, $\N_3$ are defined as the tensors
\begin{align}
\N_0 &:= (\P \Im (\phi \nabla_x \overline{\phi}), \P(\underline{A}) \cdot \nabla_x \phi) \label{n0-def}\\
\N_1 &:= ((\partial_t A_0) \underline \Phi, A_0 \partial_t \underline \Phi)\label{n1-def}\\
\N_2 &:= \underline \Phi \nabla_{x,t} \underline \Phi \label{n2-def}\\
\N_3 &:= \Phi \Phi \Phi.\label{n3-def}
\end{align}

\begin{remark} The cubic nonlinearity $\N_3$ is relatively easy to deal with.  The nonlinearity $\N_2$ would be dangerous if it were present in the ``hyperbolic'' equation for $\Box \underline \Phi$, but fortunately only affects the ``elliptic'' equation for $A_0$, which has better smoothing effects with which to handle this nonlinearity.  The nonlinearity $\N_1$ is tractable due to the high regularity of $A_0$. The null form $\N_0 = \N_0(\underline \Phi, \underline \Phi)$ is perhaps the most interesting.  It is a special case of the more general quadratic form $\N_2$, or more precisely
\begin{equation}\label{ntriv}
\N_0 = \O(\N_2).
\end{equation}
However, one can express $\N_0$ more carefully as
\begin{equation}\label{n1}
\N_0(\underline{\Phi}, \underline{\Phi'}) = \O( |\nabla_x|^{-1}
Q(\underline{\Phi}, \underline{\Phi'}) ) + \O( Q(|\nabla_x|^{-1}
\underline{\Phi}, \underline{\Phi'}) )
\end{equation}
where $Q$ is the null form $Q(\phi, \psi) := \nabla_x \phi \wedge
\nabla_x \psi$, which can be expressed in components as
$$ Q_{jk}(\phi, \psi) := \partial_j \phi \partial_k \psi
- \partial_k \phi \partial_j \psi.$$ See \cite{kl-mac:mkg} for more
details.
\end{remark}

\begin{remark} The equation
$$\Box \Phi = \O( \N_0(\Phi, \Phi) )$$
is sometimes used as a simplified model for (MKG-CG) (and also for
Yang-Mills equations in the Coulomb gauge); see e.g.
\cite{kl-mac:optimal}, \cite{kl-tar:ym}.  However we will not use
this model equation here.
\end{remark}

\section{Function spaces}\label{function-sec}

We now recall some notation for the function spaces we shall use to
control the nonlinear expressions $\N_0,\N_1,\N_2,\N_3$ properly,
which will be useful both for proving Theorem \ref{hs-lwp} and
Proposition \ref{main-prop}.

Given a spacetime function $\phi: \R \times \R^3 \to \C$, we use
$\tilde \phi$ to denote the spacetime Fourier transform
$$ \tilde \phi(\xi, \tau) := \int_{\R \times \R^3} e^{-i (x \cdot \xi + t\tau)} \phi(x,t)\ dx dt.$$
Of course, the spacetime Fourier transform only makes sense if
$\phi$ is defined globally on $\R \times \R^3$ (as opposed to a
spacetime slab such as $[0,T] \times \R^3$).  In practice this
difficulty is avoided by using the spacetime Fourier transform to
define global function spaces, and then define their local
counterparts by restriction.

If $X$ is a Banach space of functions on $\R^3$, we use $L^q_t X$ to
denote the space of functions whose norm
$$ \| u \|_{L^q_t X} := (\int_\R \| u(t) \|_X^q\ dt)^{1/q}$$
is finite, with the usual modifications when $q=\infty$; we also let
$C^0_t X$ be the space of bounded continuous functions from $\R$ to
$X$ with the supremum norm.  In particular, we have the mixed
Lebesgue spaces $L^q_t L^r_x$ and the energy spaces $C^0_t H^s_x
\cap C^1_t H^{s-1}_x$.  These spaces localise to spacetime slabs $I
\times \R^3$ in the obvious manner.

For any $s, b\in \R$, we denote the space\footnote{These spaces are
also known as $X^{s,b}$ spaces in the literature.} $H^{s,b} =
H^{s,b}(\R \times \R^3)$ of spacetime functions on $\R \times \R^3 $
whose norm
$$ \| u \|_{H^{s,b}} := \| \langle \xi \rangle^s \langle |\xi| - |\tau| \rangle^b \tilde u \|_{L^2_\xi L^2_\tau}$$
is finite. We observe the trivial inclusions $H^{s_2,b_2} \subseteq
H^{s_1,b_1}$ whenever $s_2 \geq s_1$ and $b_2 \geq b_1$.

If $I$ is a bounded time interval, we define the restricted space
$H^{s,b}_I = H^{s,b}(I \times \R^3)$ to consist of the restriction
of $H^{s,b}$ functions to the spacetime slab $I \times \R^3$, with
norm
$$ \| u \|_{H^{s,b}_I} := \inf \{ \| v \|_{H^{s,b}}: v \in H^{s,b}, v|_{I \times \R^3} \equiv u \}.$$

We now formalize the well-known fact that $H^{s,1/2+}$ functions are
``averages'' of free $H^s_x$ solutions to the wave equation (see
e.g. \cite[Proposition 7]{selberg:thesis} or \cite[Lemma 2.9]{tao}).

\begin{lemma}[$H^{s,b}$ decomposes into free solutions]\label{wave-general}  Let $\phi \in H^{s,b}$ for some $b>1/2$.  Then for each $\lambda \in \R$ there exists a global solution $\phi_\lambda$ to the free wave equation $\Box \phi_\lambda = 0$ with $\| \phi_\lambda[t]\|_{H^s_x \times H^{s-1}_x} \lesssim 1$ for all $t$, and a co-efficient $a(\lambda) \in \R$, such that
$$ \phi(t) = \int_\R a(\lambda) e^{i t \lambda} \phi_\lambda(t)\ d\lambda$$
for all $t$, and such that
$$\|a\|_{L^1_\lambda} \lesssim \| \langle \lambda \rangle^b a\|_{L^2_\lambda} \lesssim \| \phi \|_{H^{s,b}}$$
where the implicit constant can depend on $b$.
\end{lemma}

\begin{proof}
Without loss of generality we may assume that the spacetime Fourier
transform $\tilde \phi$ is supported on the upper half-space $\{
(\tau,\xi): \tau \geq 0 \}$.  We then write
$$ \widetilde \phi(\tau,\xi) = \int_\R \widetilde \phi(|\xi|+\lambda, \xi) \delta(\tau - \lambda - |\xi|)\ d\lambda$$
where $\delta$ is the Dirac delta.  If we then define
$$ a(\lambda) := \| \langle \xi \rangle^s \widetilde \phi(|\xi|+\lambda, \xi) \|_{L^2_\xi}$$
and
$$ \widetilde \phi_\lambda(\tau,\xi) := \frac{1}{a(\lambda)} \delta(\tau - |\xi|) \widetilde \phi(|\xi|+\lambda, \xi)$$
we see that all the relevant properties are easily verified except
perhaps for the $L^1_\lambda$ bound on $a$, which we compute using
Cauchy-Schwarz:
$$ \| a\|_{L^1_\lambda} \lesssim \| \langle \lambda \rangle^b a\|_{L^2_\lambda} = \| \langle \tau-|\xi| \rangle^b \langle \xi \rangle^s \widetilde \phi \|_{L^2_\tau L^2_\xi} = \| \phi \|_{H^{s,b}}.$$
\end{proof}

As a particular consequence of this lemma, we see that if one can
imbed $H^s_x \times H^{s-1}_x$ free solutions in a spacetime Banach
space $X$ which is invariant under time modulations $\phi(t) \mapsto
\phi(t) e^{i t\lambda}$, then one can also imbed $H^{s,1/2+}$
solutions into the same space.  In particular we have
\begin{equation}\label{hsd-energy}
H^{s,1/2+} \subseteq L^\infty_t H^s_x
\end{equation}
for any $s \in \R$.  Also, from Strichartz' estimate (see e.g.
\cite{tao:keel}, \cite{sogge:wave}, and the references therein) and
Lemma \ref{wave-general} we have
\begin{equation}\label{hsd-strichartz}
H^{s,1/2+} \subseteq L^q_t L^r_x
\end{equation}
whenever
\begin{equation}
\begin{array}{l}
s \geq 0 \\
\frac{1}{q} + \frac{1}{r} \leq \frac{1}{2}\\
2 < q  \leq \infty \\
\frac{1}{q} + \frac{3}{r} \geq \frac{3}{2} - s
\end{array}
\end{equation}
except at the endpoint $(q,r,s) = (2,\infty,0)$, where the estimate
is known to fail (see \cite{klainerman:nulllocal}).  If time is
localized to an interval, one can also use H\"older in time to lower
the $q$ index.

Finally, we recall

\begin{lemma}[Energy estimate] \cite[Theorem 12]{selberg:thesis}  For any time $t_0$,
any interval $I$ of width $O(1)$ containing $t_0$, any $0 \leq
\sigma \leq 1-b$, $1
> b > 1/2$, and $s \in \R$, we have
\begin{equation}\label{energy}
\| u \|_{H^{s,b}_I} + \| \nabla_{x,t} u \|_{H^{s-1,b}_I} \lesssim \|
u[t_0] \|_{H^s_x \times H^{s-1}_x} + |I|^{\sigma/2} \| \Box u
\|_{H^{s-1,b-1+\sigma}_I}
\end{equation}
whenever the right-hand side is finite.  The implied constant of
course depends on $\sigma, b, s$.
\end{lemma}

Note the factor of $|I|^{\sigma/2}$ on the right-hand side of
\eqref{energy}; this factor will be  very convenient for the large
data theory.  The fact that $\sigma$ is allowed to be as large as
$1-b$ (rather than $1/2$) allows us to reach $s > \sqrt{3}/2 \approx
.866$ rather than $s > 7/8 = .875$.

\section{Bilinear estimates}\label{bil-sec}

To obtain local well-posedness in $H^s_x$, we shall place
$\underline{\Phi}$ in $H^{s,s-}$, $\nabla_{x,t} \underline \Phi$ in
$H^{s-1,s-}$, and $\nabla_{x,t} A_0$ in $H^{1/2+,0}$ (compare with
\cite{cuccagna}).

We are now ready to prove the main bilinear estimates needed to
handle the nonlinear expressions $\N_0$, $\N_1$, $\N_2$.

\begin{proposition}[Bilinear estimates]\label{bilinear}  Let $3/4 < s < 1$.  Then we have the estimates
\begin{align}
\| \eta(t) \phi \nabla_{x,t} \psi \|_{H^{s'+s''-2, 0}} &\lesssim \| \phi \|_{H^{s', 3/4+}} \| \nabla_{x,t} \psi \|_{H^{s''-1, 3/4+}} \label{no-null}\\
\| \eta(t) (\partial_t A_0) \phi \|_{H^{s-1, s-1}} &\lesssim \| \nabla_{x,t} A_0 \|_{H^{1/2+, 0}} \| \phi \|_{H^{s, 3/4+}} \label{a0t-phi}\\
\| \eta(t) A_0 (\partial_t \phi) \|_{H^{s-1, s-1}} &\lesssim \| \nabla_{x,t} A_0 \|_{H^{1/2+, 0}} \| \partial_t \phi \|_{H^{s-1, 3/4+}}  \label{a0-phit}\\
\| \eta(t) \N_0(\phi, \psi) \|_{H^{s-1, s-1}} &\lesssim \| \phi
\|_{H^{s, 3/4+}} \| \psi \|_{H^{s, 3/4+ }} \label{null}
\end{align}
where $\eta$ is any bump function and $s \leq s',s'' \leq 1$ are
exponents which are not both equal to 1, and $A_0$, $\phi$, $\psi$
are arbitrary functions for which the left-hand side makes sense.
(Of course, the implicit constants depend on $s, \eta, s', s''$.)
\end{proposition}

These estimates were essentially proven in \cite{cuccagna}, but we
sketch a proof here based on the bilinear estimates in
\cite{damiano:klainerman}.  For local existence we need to take
$s'=s''=s$, but for global existence we will also need one of $s'$,
$s''$ to equal 1 instead.

\begin{proof}
We first prove \eqref{no-null}.  From \cite[Theorem
1.1]{damiano:klainerman} and the assumption $3/4 < s < 1$, we
have\footnote{More generally, one has $\| \phi \psi \|_{\dot
H^{s,b}} \lesssim \| \phi[0] \|_{\dot H^{s'} \times \dot H^{s'-1}}
\| \psi[0] \|_{\dot H^{s''} \times \dot H^{s''-1}}$ whenever
$s+b=s'+s''-1$, $s'+s'' > 1/2$, $s > -1$, $b \geq 0$, and $s \leq
\min(s',s'')$, with at least one of the latter two inequalities
being strict.  See \cite{damiano:klainerman}.}
$$
\| \phi \nabla_{x} \psi \|_{\dot H^{s'+s''-2, 0}} \lesssim \|
\phi[0] \|_{\dot H^{s'} \times \dot H^{s'-1}} \| \psi[0] \|_{\dot
H^{s''} \times \dot H^{s''-1}}$$ for all global solutions $\Box \phi
= \Box \psi = 0$ to the wave equation, and $\| u \|_{\dot H^{s,b}}
:= \| |\xi|^s ||\xi|-|\tau||^b \tilde u \|_{L^2_\tau L^2_\xi}$ is
the homogeneous $L^2_x$ norm.

Now we prove the variant
\begin{equation}\label{eta-variant}
\| \eta(t) \phi \nabla_{x} \psi \|_{H^{s'+s''-2, 0}} \lesssim \|
\phi[0] \|_{H^{s'} \times H^{s'-1}} \| \psi[0] \|_{H^{s''} \times
H^{s''-1}}.
\end{equation}
If $\phi$ and $\psi$ both have Fourier support on the region $|\xi|
\gtrsim 1$ then this follows from the previous estimate, since
$s'+s''-2$ is negative, and the $\dot H^{s'+s''-2,0}$ norm controls
the $H^{s'+s''-2,0}$ norm, and the $\eta(t)$ cutoff is harmless. Now
suppose that $\phi$ has Fourier support in the region $|\xi|
\lesssim 1$.  Estimating $\nabla_{x,t} \psi$ in $C^0_t H^{s''-1}$
and $\phi$ in $C^0_t H^{10}$ (for instance) it suffices by a
H\"older in time to show the estimate
$$ \| f g \|_{H^{s'+s''-2}_x} \lesssim \| f \|_{H^{10}_x} \| g\|_{H^{s''-1}_x}.$$
But this follows from \eqref{sobolev-mult}.  A similar argument
holds for $\psi$ (in fact here it is easier as the derivative is in
a more favorable location).  This proves \eqref{eta-variant}.

Now suppose that $\phi$ is an arbitrary spacetime function, while
$\psi$ still solves the wave equation $\Box \psi = 0$.  We claim the
estimate

$$
 \| \phi \nabla_{x} \psi \|_{H^{s^{'}+s^{''}-2, 0}}
\lesssim \| \phi \|_{H^{s^{'}, 3/4+}} \| \psi[0] \|_{H^s_x \times
H^{s-1}_x}. 
$$
To see this, we use Lemma \ref{wave-general} to write $\phi(t) =
\int_\R a(\lambda) e^{i t\lambda} \phi_\lambda(t)\ d\lambda$.  From
\eqref{eta-variant}, Minkowski's inequality and the observation that
the norm 
$H^{s^{'} + s^{''} -2, 0}$
is invariant under multiplication by
$e^{it\lambda}$, we have

$$
 \| \phi \nabla_{x,t} \psi \|_{H^{s^{'} + s^{''} -2, 0}}
\lesssim \| a \|_{L^1_\lambda} \| \psi[0] \|_{H^s_x \times
H^{s-1}_x} 
$$
and the claim follows from Lemma \ref{wave-general}.  If we then
apply similar reasoning to $\psi$ we obtain 
\begin{equation}\label{john}
\| \eta(t) \phi \nabla_{x} \psi \|_{H^{s'+s''-2, 0}} \lesssim \| \phi \|_{H^{s', 3/4+}} \| \psi \|_{H^{s'', 3/4+}}.
\end{equation}
To establish \eqref{no-null} in the case where $\psi$ has Fourier support in the region $|\xi| \gtrsim 1$, one applies \eqref{john} with $\psi$ replaced by $\nabla_x \Delta^{-1} \nabla_{t,x} \psi$ and then takes traces.  The final remaining case of \eqref{no-null} to check is when $\psi$ has Fourier support in the region $|\xi| \lesssim 1$.  In that case, the norm $\| \nabla_{t,x} \psi \|_{H^{s''-1,3/4+}}$ controls $\| \nabla_{t,x} \psi \|_{L^\infty_t H^{100}_x}$ (say), while $\| \phi \|_{H^{s',3/4+}}$ controls $\| \phi \|_{L^\infty_t H^{s'}_x}$, and the claim follows from standard Sobolev product estimates.

Now, we prove \eqref{null}.  From \cite[Corollary
13.4]{damiano:klainerman} and \eqref{n1}, and the hypothesis $3/4 <
s < 1$, we have\footnote{The main idea of this corollary is to
exploit the heuristic $Q(\phi, \psi) \sim |\nabla_x|^{1/2} D_-^{1/2}
(|\nabla_x|^{1/2} \phi) (|\nabla_x|^{1/2} \psi)$, where $D_-$ is the
spacetime Fourier multiplier with weight $||\xi|-|\tau||$; this
allows one to reduce the null form estimate to a product estimate
similar to the one used to prove \eqref{no-null}.  See
\cite{damiano:klainerman}, and also Lemma \ref{nf} below.}
$$
\| \N_0(\phi, \psi) \|_{\dot H^{s-1, s-1}} \lesssim \| \phi[0]
\|_{\dot H^s_x \times \dot H^{s-1}_x} \| \psi[0] \|_{\dot H^s_x
\times \dot H^{s-1}_x}$$ when $\Box \phi = \Box \psi = 0$.  We now
prove the variant estimate
\begin{equation}\label{var}
\| \eta(t) \N_0(\phi, \psi) \|_{H^{s-1, s-1}} \lesssim \| \phi[0]
\|_{H^s_x \times H^{s-1}_x} \| \psi[0] \|_{H^s_x \times H^{s-1}_x}.
\end{equation}
Again this estimate follows from the previous when $\phi$ and $\psi$
both have Fourier support in the region $|\xi| \gtrsim 1$ (note that
multiplication by $\eta(t)$ is bounded on any $H^{s,b}$ space).  Now
suppose $\phi$ has Fourier support on $|\xi| \lesssim 1$. Crudely
writing $\N_0(\phi,\psi) = \O( \phi \nabla_x \psi )$ and estimating
the $H^{s-1,s-1}$ norm by the $L^2_t H^{s-1}_x$ norm, we argue as
with \eqref{no-null}, the only difference being that $-1/2+$ has
been replaced by $s-1$.  Similarly for $\psi$.

Now we invoke Lemma \ref{wave-general} again.  We have to be a bit
more careful because the space $H^{s-1,s-1}$ is not invariant under
multiplication by $e^{it\lambda}$, however from the (rather crude)
elementary inequality $\langle a + b \rangle^{s-1} \lesssim \langle
a \rangle^{1-s} \langle b \rangle^{s-1}$ we do have the estimate
$$ \| e^{it\lambda} u \|_{H^{s-1,s-1}} \lesssim \langle \lambda \rangle^{1-s} \| u \|_{H^{s-1,s-1}}.$$
So if $\phi$ does not solve the free wave equation, and we decompose
as in Lemma \ref{wave-general}, then
$$ \| \eta(t) \N_0(\phi, \psi) \|_{H^{s-1, s-1}} \lesssim
\int_\R \langle \lambda \rangle^{1-s}  |a_\lambda(t)| \| \eta(t)
\N_0(\phi_\lambda, \psi) \|_{H^{s-1,s-1}}\ d\lambda.$$ Using
\eqref{var} we thus have
$$ \| \eta(t) \N_0(\phi, \psi) \|_{H^{s-1, s-1}} \lesssim \| \langle \lambda \rangle^{1-s} a \|_{L^1_\lambda} \| \psi[0] \|_{H^s_x \times H^{s-1}_x}.$$
However since $s > 3/4$, we have $3/4+ > 1/2 + (1-s)$, so we can
estimate
$$ \| \langle \lambda \rangle^{1-s} a \|_{L^1_\lambda} \lesssim \| \langle \lambda \rangle^{3/4+} a \|_{L^2_\lambda} \lesssim \| \phi \|_{H^{s,3/4+}}.$$
By arguing similarly for $\psi$ we obtain \eqref{null}.

Finally, we prove \eqref{a0t-phi}, \eqref{a0-phit}.  From the
embeddings $H^{s,3/4+} \subseteq C^0_t H^s_x$, $H^{s-1,3/4+}
\subseteq C^0_t H^{s-1}_x$ (from \eqref{hsd-energy}) and the trivial
embedding $H^{s-1,s-1} \subseteq L^2_t H^{s-1}_x$, it suffices to
show the spatial product estimates
$$ \| f g \|_{H^{s-1}_x} \lesssim \| f \|_{H^{1/2+}_x} \| g \|_{H^s_x}$$
and
$$ \| f g \|_{H^{s-1}_x} \lesssim \| \nabla_x f \|_{H^{1/2+}_x} \| g \|_{H^{s-1}_x}.$$
The first estimate follows directly from \eqref{sobolev-mult}.  To
prove the second we use duality to convert it to
$$ \| f h \|_{H^{1-s}_x} \lesssim \| \nabla_x f \|_{H^{1/2+}_x} \| h \|_{H^{1-s}}.$$
If $f$ has Fourier support in the region $|\xi| \gtrsim 1$ then this
again follows from \eqref{sobolev-mult}, so we may assume $f$ has
support in the region $|\xi| \lesssim 1$.  But then this follows by
applying the fractional Leibnitz rule for $\langle \nabla_x
\rangle^{1-s}$ and the Sobolev embedding $\| \langle \nabla_x
\rangle^{1-s} f \|_{L^\infty_x} \lesssim \| \nabla_x f
\|_{H^{1/2+}_x}$.
\end{proof}

\section{Local existence}\label{local-sec}

We are now finally ready to prove the local existence result,
Theorem \ref{hs-lwp}.

\begin{remark} This result is essentially in \cite{cuccagna}, but for our application we need
local well-posedness for large data as well as small.  One cannot
simply rescale large data to be small because the $L^2_x$ norm is
supercritical in (MKG-CG).  By localizing in time and modifying the
``$b$'' index of the $H^{s,b}$ norms we can control the
``hyperbolic'' component of the large data evolution for short
times.  However for the ``elliptic'' component of the evolution
localizing in time does not help.  One might try localizing in
space, but this is tricky because the non-local Coulomb gauge has
destroyed finite speed of propagation, and one would probably be
forced to use local Coulomb gauges, cf. \cite{klainerman:yangmills}
and \cite{uhlenbeck}. Fortunately, the variational estimates in
Section \ref{elliptic-sec} will allow us to avoid these
difficulties.
\end{remark}

\begin{proof}[Proof of Theorem \ref{hs-lwp}]
Let $T > 0$ be chosen later.  Let $\Phi[0]$, $\Phi'[0]$ be two
smooth elements of $M\Ball([H^s])$ obeying \eqref{compat}.  Then by
the global well-posedness theory in the introduction we can find
smooth solutions $\Phi$, $\Phi'$ to (MKG-CG) on the slab $[-T,T]
\times \R^3$.  We shall show the estimates
\begin{equation}\label{lipschitz}
\| \Phi - \Phi' \|_{C^0_t [H^s]} \lesssim C(M) \| \Phi[0] - \Phi'[0]
\|_{[H^s]}
\end{equation}
on this slab; the Theorem then follows by a standard limiting
argument, using the remarks at the end of Section \ref{elliptic-sec}
to approximate rough data by smooth data.

Define the norm $\| \Phi \|_X$ on the slab $[-T,T] \times \R^3$ by
\begin{equation}\label{xdef}
 \| \Phi \|_X := \| \underline{\Phi} \|_{H^{s,3/4+}_{[-T,T]}}
+ \| \nabla_{x,t} \underline{\Phi} \|_{H^{s-1,3/4+}_{[-T,T]}} + \eps
\| \nabla_{x,t} A_0 \|_{L^2_t H^{1/2+}_x}
\end{equation}
where $0 < \eps = \eps(M) \ll 1$ is a small number to be chosen
later.  We shall show that
$$ \| \Phi - \Phi' \|_X \lesssim C(M) \| \Phi[0] - \Phi'[0]\|_{[H^s]}$$
which will imply \eqref{lipschitz} by \eqref{hsd-energy} and
\eqref{a0-diff}.

To simplify the exposition we shall just prove the bound
$$ \| \Phi \|_X \lesssim C(M)$$
but the reader may verify that the arguments below can be easily
adapted to differences.

We begin with the $\underline{\Phi}$ component of the $X$ norm.  By
\eqref{energy} and \eqref{caricature} we have
\begin{equation}\label{e-start}
\| \underline{\Phi} \|_{H^{s,3/4+}_{[-T,T]}} + \| \nabla_{x,t}
\underline{\Phi} \|_{H^{s-1,3/4+}_{[-T,T]}} \lesssim M +
T^{(s-3/4+)/2} \sum_{j=0,1,3} \| \N_j \|_{H^{s-1,s-1}_{[-T,T]}}.
\end{equation}
From Proposition \ref{bilinear} (restricted to $[-T,T]$ in the
obvious fashion) we can control the $\N_0$, $\N_1$ nonlinearities:
$$ \| \N_0 \|_{H^{s-1,s-1}_{[-T,T]}} + \|\N_1 \|_{H^{s-1,s-1}_{[-T,T]}}
\lesssim C(\eps) \| \Phi \|_X^2.$$ On the other hand, from
Strichartz \eqref{hsd-strichartz} we have\footnote{It is here that we crucially make use of the hypothesis $s>5/6$.  It is likely that the methods of Cuccagna\cite{cuccagna} can control the cubic terms $\N_3$ by more sophisticated estimates than Strichartz estimates in the larger range $s>3/4$, but we will not need to do so here.  We thank the referee for these points.}
$$ \| \underline{\Phi} \|_{L^6_t L^6_x} \lesssim \| \Phi \|_X$$
Also, we have by Sobolev and \eqref{nab-a0}
\begin{equation}\label{a66}
\begin{split}
\| A_0 \|_{L^6_t L^6_x} &\lesssim \| \nabla_x A_0 \|_{C^0_t L^2_x} \\
&\lesssim \sum_{j=2}^3 ( \| \phi \|_{C^0_t H^s_x} + \| A \|_{C^0_t H^s_x} + \| \phi_t \|_{C^0_t H^{s-1}_x} )^j \\
&\lesssim \sum_{j=2}^3 \| \Phi \|_X^j.
\end{split}
\end{equation}
Combining these estimates we can control the cubic nonlinearity
$\N_3$:
\begin{equation}\label{phi-ss}
\|\N_3\|_{ H^{s-1,s-1}_{[-T,T]}} \lesssim \|\Phi^3\|_{L^2_t L^2_x}
\lesssim \| \Phi \|_{L^6_t L^6_x}^3 \lesssim \sum_{j=3}^9 \| \Phi
\|_X^j.
\end{equation}
Putting all of this together we obtain
\begin{equation}\label{uphi}
\| \underline{\Phi} \|_{H^{s,3/4+}_{[-T,T]}} + \| \nabla_{x,t}
\underline{\Phi} \|_{H^{s-1,3/4+}_{[-T,T]}} \lesssim M +
T^{(s-3/4+)/2} C(\eps,\| \Phi \|_X).
\end{equation}

Now we consider the $A_0$ component of the $X$ norm.  From the
computations in \eqref{a66} we have
$$ \| \nabla_{x,t} A_0 \|_{L^2_t L^2_x} \lesssim \| \Phi \|_X^2$$
so we have
$$ \| \nabla_{x,t} A_0 \|_{L^2_t H^{1/2+}_x} \lesssim \| \Phi \|_X^2 +
\| \nabla_x \nabla_{x,t} A_0 \|_{L^2_t H^{-1/2+}_x}.$$ By
\eqref{caricature} we thus have
$$ \| \nabla_{x,t} A_0 \|_{L^2_t H^{1/2+}_x} \lesssim \| \Phi \|_X^2 +
\| \N_2 \|_{L^2_t H^{-1/2+}_x} + \| \N_3 \|_{L^2_t H^{-1/2+}_x}.$$
We estimate the second term using Proposition \ref{bilinear}
(observing that $s+s-2  > -1/2+$ since $s > 3/4$), and we estimate
the third term by the computations in \eqref{phi-ss}, to obtain
\begin{equation}\label{ao}
\| \nabla_{x,t} A_0 \|_{L^2_t H^{1/2+}_x} \lesssim \| \Phi \|_X^2 +
\| \Phi \|_X^3.
\end{equation}
Combining this with our previous bound for $\underline{\Phi}$ we
obtain
$$ \| \Phi \|_X \lesssim M +
T^{(s-3/4+)/2} C(\eps,\| \Phi \|_X) + \eps C(\| \Phi \|_X).$$ If we
now choose $\eps = \eps(M)$ sufficiently small, and $T = T(\eps, M)$
sufficiently small, we thus see that there is an absolute constant
$C$ such that
$$ \| \Phi \|_X \leq 2CM \implies \| \Phi \|_X \leq CM.$$
By standard continuity arguments this implies that $\| \Phi \|_X
\leq CM$, as desired.  The adaptation of this scheme to differences
is routine (using \eqref{a0-diff} instead of \eqref{nab-a0}) and is
left to the reader.
\end{proof}

\section{Modified local well-posedness}\label{mod-sec}

To finish the proof of Theorem \ref{main}, it only remains to prove
Proposition \ref{main-prop}.  Fix $A_0$, $\underline{\Phi}$, $t_0$.

From a continuity argument we may assume a priori that
\begin{equation}\label{sup-hamil-2}
\sup_{t \in [t_0 - \delta/2, t_0 + \delta/2]} \H[I\Phi[\lambda t]]
\leq 3.
\end{equation}

Our objective is to control the modified Hamiltonian at times $t$
close to $t_0$. In order to do this we must obtain estimates on
$\Phi$ away from $t_0$. One obvious possibility is to combine
\eqref{sup-hamil-2} and Lemma \ref{hamil-control}; this for instance
will give the estimate
\begin{equation}\label{idphi}
(\partial_t + i IA_0) I \phi \in C^0_t L^2_x
\end{equation}
on the slab $[t_0 - \delta/2, t_0 + \delta/2]$.  However these types
of estimates (which basically place $I \Phi$ in $C^0_t [H^1]$) will
not be the right type of estimates (except for the low frequency
component) for estimating the change in the
Hamiltonian\footnote{Basically, the problem is that Sobolev
embedding in three dimensions does not allow $C^0_t [H^1]$ to
control $L^\infty_x$ norms, so that nonlinearities such as $\Phi
\nabla_x \Phi$ cannot be placed in $L^2_x$ norms.  To get around
this we must use Strichartz embeddings and null form estimates,
which in turn necessitates the use of $H^{s,b}$ spaces.  We remark
that in one dimension one can rely purely on Sobolev embedding and
obtain global well-posedness results for nonlinear wave equations
below $H^1$ without using $H^{s,b}$ or Strichartz norms; see
\cite{borg:book}.}; it turns out that we need estimates in $H^{s,b}$
spaces, which are not directly controlled by the Hamiltonian.

One might hope to apply Theorem \ref{hs-lwp}, since the regularity
\eqref{hi-eps} should be enough to put $\Phi$ in $[H^s]$.  However
this is inefficient (basically because there is a significant loss
in using \eqref{i-loss}) and in addition there are some low
frequency issues, because of the error terms in Lemma
\ref{hamil-control}.  So we shall instead require a modified local
well-posedness result which is adapted to the estimates arising from
Lemma \ref{hamil-control}.

From \eqref{hi-eps} and Lemma \ref{hamil-control} we thus have
\begin{align}
\underline \Phi(t_0) &\in C\Ball(I^{-1} H^1_x + \L^6_1)\label{phi-t0}\\
\nabla_{x,t} \underline{\Phi}(t_0) &\in C\Ball(I^{-1} L^2_x +
\L^3_2)\label{gradphi-t0}
\end{align}
(recall that $I$ is the identity on low frequencies). We now extend
this control at time $t_0$ to control on $[t_0-\delta, t_0+\delta]$
with the following proposition, which is the main result of this
section.


\begin{proposition}[Spacetime control on $\underline{\Phi}$]\label{i-local}  Adopt the assumptions
of Proposition \ref{main-prop}, and suppose in addition $K \gg 1$ is a sufficiently
large constant.  We conclude there is a $\delta$ with  
$c(K) N^{0-} \leq   \delta \leq C(K)$, such that 
\begin{equation}\label{phi-spacetime}
\nabla_{x,t} \underline{\Phi}(t_0) \in K\Ball(I^{-1}
H^{0,s-}_{[t_0-\delta,t_0+\delta]} + C^0_t \L^3_2)
\end{equation}
on the slab $[t_0-\delta, t_0+\delta] \times \R^3$.
\end{proposition}
\begin{remark}
Note that the ``$b$'' index is now $s-$ instead of $3/4+$. This
extra regularity in the $b$ index will turn out to be helpful when
proving \eqref{h-shift}. These types of estimates morally follow
from the local well-posedness theory (or more precisely, the
multilinear estimates underlying that theory) using tools such as
\cite[Lemma 12.1]{ckstt:5}, but for various technical reasons it is
not feasible to do so directly, and so we have chosen instead the
following more pedestrian argument.
\end{remark}

\begin{proof}
All our spacetime norms here will be on the slab $[t_0-\delta,
t_0+\delta]$.  Many of the unpleasant technicalities in the
following argument will arise from the low frequency terms $\L^p_R$,
and the reader is advised to ignore all the contributions from these
terms in a first reading as they are not the essential difficulty.

By the continuity method and the smoothness of $\Phi$ it will
suffice to prove this under the \emph{a priori} assumption
\begin{equation}\label{continuity}
\nabla_{x,t} \underline{\Phi} \in 2K\Ball(I^{-1}
H^{0,s-}_{[t_0-\delta,t_0+\delta]} + C^0_t \L^3_2)
\end{equation}
(since this will imply that the space of $\delta$ for which
\eqref{phi-spacetime} holds is both open and closed, if $\delta$ is
restricted to be sufficiently small).

To prove this, we begin by estimating $\underline{\Phi}$ in various
auxiliary norms. Split $\underline{\Phi}$ smoothly into a low
frequency component $\underline{\Phi}_{\low}$ supported on $|\xi|
\leq 2$, and a local component $\underline{\Phi}_{\local}$ supported
on $|\xi| \geq 1$.  From \eqref{continuity} we have
$$ \nabla_x \underline{\Phi}_{\local} \in CK\Ball(I^{-1} H^{0,s-}_{[t_0-\delta, t_0 + \delta]} + C^0_t
\L^3_2),$$ which implies that
$$ \underline{\Phi}_{\local} \in CK\Ball(I^{-1} H^{1,s-}_{[t_0-\delta, t_0 + \delta]} + C^0_t \L^3_2). $$
Also, from Sobolev embedding, \eqref{hsd-energy} and the low
frequency restriction we have
$$ \partial_t \underline{\Phi}_{\low} \in CK\Ball(C^0_t \L^3_2)$$
while from \eqref{phi-t0} and Sobolev embedding we have
$$ \underline \Phi_{\low}(t_0) \in C\Ball(\L^6_1 + \L^3_2)$$
so by the fundamental theorem of calculus we have
$$ \underline \Phi_{\low} \in CK\Ball(C^0_t \L^6_1 + C^0_t \L^3_2).$$
Combining these estimates we obtain
\begin{equation}\label{phi-ball}
\underline \Phi \in CK\Ball(I^{-1} H^{1,s-}_{[t_0-\delta, t_0 +
\delta]} + C^0_t \L^3_2 + C^0_t \L^6_1).
\end{equation}
From Strichartz \eqref{hsd-strichartz} and \eqref{lpr-bernstein} we
have in particular that
$$ \underline \Phi \in CK \Ball(L^6_t L^6_x).$$

Our next task is to obtain estimates on $A_0$. From
\eqref{phi-ball}, \eqref{hsd-energy}, \eqref{i-loss} and Sobolev
embedding we have
$$
\underline \Phi \in CK \Ball(C^0_t H^s_x + C^0_t \L^6_2)$$ while
from \eqref{continuity}, \eqref{hsd-energy}, \eqref{i-loss} we have
$$
\underline \Phi_t \in CK\Ball(C^0_t H^{s-1}_x + C^0_t \L^3_2).$$ 
From
Lemma \ref{mult}, \eqref{lpr-mult}, and \eqref{lpr-holder} we thus
have
$$
\underline \Phi \ \underline \Phi_t \in CK^2 \Ball(C^0_t \dot
H^{-1}_x)$$ so by \eqref{cz} we have
\begin{equation}\label{a0-bound}
A_0 \in CK^2 \Ball(C^0_t \dot H^1_x).
\end{equation}
From Sobolev embedding we thus have
\begin{equation}\label{a-sob}
A_0 \in CK^2 \Ball(C^0_t L^6_x) \subseteq CK^2 \Ball(L^6_t L^6_x).
\end{equation}
Combining this with our $L^6_t L^6_x$ bound on $\underline \Phi$ we
thus have $\Phi \in CK^2 \Ball(L^6_t L^6_x)$. In particular we can
control the cubic nonlinearity $\N_3$:
\begin{equation}\label{cubic}
\N_3 \in CK^6 \Ball(H^{0,0}_{[t_0-\delta, t_0+\delta]}).
\end{equation}
To control the bilinear nonlinearities $\N_0$, $\N_1$, $\N_2$ we
need the following variant of Proposition \ref{bilinear}:

\begin{lemma}[Modified bilinear estimates]\label{bilinear-cor}
\begin{align}
\|\eta(t) \phi \nabla_{x,t} \psi \|_{I^{-2} H^{0,0}} &\lesssim
N^{0+}
\| \phi \|_{I^{-1} H^{1, s-}} \| \nabla_{x,t} \psi \|_{I^{-1} H^{0, s-}} \label{no-null-I}\\
\|\eta(t) (\partial_t A_0) \phi \|_{I^{-1} H^{0, s-1}} &\lesssim
N^{0+} \| \partial_t A_0 \|_{H^{1/2+,0}} \| \phi \|_{I^{-1} H^{1,
s-}}
\label{a0t-phi-I}\\
\|\eta(t) A_0 (\partial_t \phi) \|_{I^{-1} H^{0, s-1}} &\lesssim
N^{0+} \| \nabla_{x,t} A_0 \|_{H^{1/2+,0}} \| \partial_t \phi
\|_{I^{-1} H^{0, s-}}
 \label{a0-phit-I}\\
\|\eta(t) \N_0(\phi, \psi) \|_{I^{-1} H^{0, s-1}} &\lesssim N^{0+}
\| \phi \|_{I^{-1} H^{1, s-}} \| \psi \|_{I^{-1} H^{1, s-}}
\label{null-I}
\end{align}
\end{lemma}

\begin{proof}
We first prove \eqref{null-I}.  By an appropriate Fourier
decomposition\footnote{For instance, one can decompose both $\tilde
\phi$ and $\tilde \psi$ into three components with frequency support
$\lesssim N$, $\sim N$, $\gtrsim N$ respectively, in such a way that
each of the nine interactions falls into one of the four categories
described below.} it will suffice to prove the estimate in the
following four cases:
\begin{itemize}
\item (bounded-bounded interaction) $\tilde \phi$, $\tilde \psi$ are supported on the region $|\xi| \leq N/2$.
\item (high-high interaction) $\tilde \phi$, $\tilde \psi$ are supported on $|\xi| > N/10$.
\item (bounded-high interaction) $\tilde \phi$ is supported on $|\xi| < N/5$ and $\tilde \psi$ is supported on $|\xi| > N/4$.
\item (high-bounded interaction) $\tilde \psi$ is supported on $|\xi| < N/5$ and $\tilde \phi$ is supported on $|\xi| > N/4$.
\end{itemize}

By \eqref{null} (noting that $s- > 3/4+$)  we have
$$ \| \N_0(\phi, \psi) \|_{I^{-1} H^{0, s-1}}
\lesssim N^{1-s} \| \N_0(\phi, \psi) \|_{H^{s-1, s-1}} \lesssim
N^{1-s} \| \phi \|_{H^{s, s-}} \| \psi \|_{H^{s, s- }}. $$ If $\phi$
has high frequency then $\| \phi \|_{H^{s, s-}} \lesssim N^{s-1} \|
\phi \|_{I^{-1} H^{1, s-}}$, otherwise we just have $\| \phi
\|_{H^{s, s-}} \lesssim \| \phi \|_{I^{-1} H^{1, s-}}$. Similarly
for $\psi$. This allows us to deal with the high-high, high-bounded,
and bounded-high cases. In the bounded-bounded case, all the $I$
operators are the identity, so it suffices to show
$$\|\eta(t) \N_0(\phi, \psi) \|_{H^{0, s-1}} \lesssim \| \phi \|_{H^{1+, s-}} \| \psi \|_{H^{1, s- }}$$
where we have used the $N^{0+}$ to gain an epsilon regularity on the
bounded frequency function $\phi$.  We crudely estimate the
$H^{0,s-1}$ norm by the $L^2_t L^2_x$ norm and crudely write the
null form $\N_0(\phi, \psi)$ as $\O( \phi \nabla_x \psi )$.  The
claim then follows from the Strichartz embeddings $H^{1+, s-}
\subseteq L^{2+}_t L^\infty_x$  and $H^{0,s-} \subseteq C^0_t L^2_x$
from \eqref{hsd-strichartz}.

Now we prove \eqref{a0t-phi-I}, \eqref{a0-phit-I}.  In the
bounded-high and high-high cases this follows from
\eqref{a0t-phi}, \eqref{a0-phit} by the same arguments as before
(indeed, we may even weaken the $H^{1/2+,0}$ norm to $I^{-1}
H^{1/2+,0}$).  It remains to consider the bounded-bounded and high-bounded cases.  We begin with the bounded-bounded case.
Estimating the $H^{0,s-1}$ norm by the $L^2_t L^2_x$ norm, it
suffices to show
$$ \|\eta(t) (\partial_t A_0) \phi \|_{L^2_t L^2_x} \lesssim \| \partial_t A_0 \|_{H^{1/2+,0}} \| \phi \|_{H^{1, s-}}
$$
and
$$ \|\eta(t) A_0 (\partial_t \phi) \|_{L^2_t L^2_x} \lesssim \| \nabla_{x,t} A_0 \|_{H^{1/2+,0}}
\| \partial_t \phi \|_{H^{0, s-}}. $$ For the first estimate we use
the Sobolev embedding $H^{1/2+,0} \subseteq L^2_t L^3_x$ and the
Strichartz estimate $H^{1,s-} \subseteq C^0_t L^6_x$ from
\eqref{hsd-strichartz}. For the second we use the Sobolev embedding
$\| A_0 \|_{L^2_t L^\infty_x} \lesssim \| \nabla_{x,t} A_0
\|_{H^{1/2+,0}}$ and the Strichartz estimate $H^{0,s-} \subseteq
C^0_t L^2_x$.   The same argument also deals with the high-bounded case; one can use the theory of paraproducts to cancel the factors of $I$.

Finally we prove \eqref{no-null-I}.  In the high-high case we
compute using \eqref{no-null} with $s'=s''=s$:
$$ \| \phi \nabla_{x,t} \psi \|_{I^{-2} H^{0, 0}}
\lesssim N^{2-2s} \| \phi \nabla_{x,t} \psi \|_{H^{2s-2, 0}}
\lesssim N^{2-2s} \| \phi \|_{H^{s, s- }} \| \nabla_{x,t} \psi
\|_{H^{s-1, s- }} $$
 which is acceptable in the high-high
case.  In the high-bounded case we use \eqref{no-null} with $s'=s$
and $s''=1$:
$$ \| \phi \nabla_{x,t} \psi \|_{I^{-2} H^{0, 0}}
\lesssim N^{1-s} \| \phi \nabla_{x,t} \psi \|_{H^{s-1, 0}} \lesssim
N^{1-s} \| \phi \|_{H^{s, s- }} \| \nabla_{x,t} \psi \|_{H^{0, s-
}}$$ which is acceptable.  For the bounded-high case we similarly
use \eqref{no-null} with $s'=1$ and $s''=s$.  Now we turn to the
bounded-bounded case.  Here we take advantage of the additional
$N^{0+}$ factor; it suffices to prove
$$
\|\eta(t) \phi \nabla_{x,t} \psi \|_{H^{0,0}} \lesssim \| \phi
\|_{H^{1+, s- }} \| \nabla_{x,t} \psi \|_{H^{0, s- }}.
$$
But then this follows from the Strichartz embeddings $H^{1+, s- }
\subseteq L^{2+}_t L^\infty_x$ and $H^{0,s- } \subseteq C^0_t L^2_x$
from \eqref{hsd-strichartz}.
\end{proof}

We can now control $\N_2$.  Indeed we claim
\begin{equation}\label{a-nab}
\N_2 \in CN^{0+} I^{-2} K \Ball(H^{0,0}_{[t_0-\delta, t_0+\delta]})
\end{equation}
To prove this we multiply \eqref{continuity} and \eqref{phi-ball}.
To multiply the two $H^{s,b}$ spaces we use\footnote{The estimates
there were phrased for cutoff functions $\eta$ centered at the
origin, but it is clear from time translation invariance that one
can also use cutoff functions centered at $t_0$.} \eqref{no-null-I}.
To multiply the $\L^p_R$ spaces we use \eqref{lpr-holder},
\eqref{lpr-bernstein} (indeed we can get into $H^{0,0}= L^2_t L^2_x$
for these terms).

It remains to handle the cross terms when $H^{s,b}$ is multiplied
against a $\L^p_R$. By \eqref{continuity} and \eqref{phi-ball} it
suffices to prove the embedding
\begin{equation}\label{i-embed}
(I^{-1} H^{0,0}) \cdot C^0_t \L^p_R \subseteq I^{-1} H^{0,0}
\end{equation}
for $\L^p_R = \L^6_1, \L^3_2$. But this is implied by
\begin{equation}\label{i-2}
(I^{-1} L^2_x) \cdot \L^p_R \subseteq I^{-1} L^2_x
\end{equation}
which easily follows from \eqref{lpr-mult} and a decomposition in to
high and bounded frequencies.

From \eqref{a-nab}, \eqref{cubic}, \eqref{caricature} we have
$$ \nabla_x \nabla_{x,t} A_0 \in CN^{0+} K^7 \Ball(I^{-2} H^{0,0}_{[t_0-\delta, t_0+\delta]}).$$
Split $A_0$ into bounded frequencies $|\xi| \leq 1$ and local
frequencies $|\xi| \geq 1/2$.  For the local component we invert
$\nabla$ in the above, while for bounded frequencies we use Sobolev
embedding, to obtain
\begin{equation}\label{a0-nab}
\nabla_{x,t} A_0 \in CN^{0+} K^7 \Ball(I^{-2} H^{1,0}_{[t_0-\delta,
t_0+\delta]} + L^2_t \L^6_1).
\end{equation}

Now we control $\N_0$, $\N_1$, $\N_3$.  Indeed we claim
\begin{equation}\label{box-phi}
\N_0, \N_1, \N_3 \in C N^{0+} K^{20} \Ball(I^{-1} H^{0,
s-1}_{[t_0-\delta, t_0+\delta]} + L^2_t \L^3_2).
\end{equation}
The cubic term $\N_3$ is acceptable by \eqref{cubic}, so we turn to
$\N_1$.

To control the $A_0 \partial_t \underline \Phi$ component of $\N_1$,
we use \eqref{continuity}.  The $2K\Ball(C^0_t \L^3_2)$ component of
$\partial_t \underline \Phi$ will be acceptable from \eqref{a-sob}, as this places this contribution to $A_0 \partial_t \underline \Phi$ in $C^0_t L^2_x$, which easily embeds into $I^{-2} H^{1,0}_{[t_0-\delta,
t_0+\delta]}$.  So it suffices to show that
$$ A_0 \cdot 2K\Ball(I^{-1} H^{0,s- }_{[t_0-\delta,t_0+\delta]}) \subseteq
C N^{0+} K^{20} \Ball(I^{-1} H^{0, s-1}_{[t_0-\delta,t_0+\delta]}).$$

Split $A_0$ into low frequencies $|\xi| \leq 2$ and local
frequencies $|\xi| \geq 1$. For local frequencies we can use
\eqref{a0-nab} and \eqref{a0-phit-I} (observing that $I^{-2} H^{1,0}
\subseteq H^{1/2+,0}$).  For low frequencies we have
$$ A_0 \in CK^2 \Ball(C^0_t \L^6_2)$$
from \eqref{a0-bound}, and the claim will follow from
\eqref{i-embed}.

Now we control the $\partial_t A_0 \underline \Phi$ component of
$\N_1$.  To do this we multiply \eqref{a0-nab} with
\eqref{phi-ball}.  The product of $I^{-2} H^{1,0}$ and $I^{-1}
H^{1,s- }$ is acceptable from \eqref{a0t-phi-I} (again observing
$I^{-2} H^{1,0} \subseteq H^{1/2+,0}$).  The product of $L^2_t
\L^6_1$ and $C^0_t \L^3_t$ is in $L^2_t L^2_x \subseteq I^{-1}
H^{0,s-1}$ by \eqref{lpr-holder}, while the product of $L^2_t
\L^6_1$ and $C^0_t \L^6_1$ is similarly in $L^2_t \L^3_2$. For the
cross terms between $H^{1/2+,0}$ and $C^0_t L^p_R$ we use
\eqref{i-embed} (since $H^{1/2+,0} \subseteq I^{-1} H^{0,0}$), while
for the cross terms between $I^{-1} H^{1,s-}$ and $L^2_t \L^6_1$ we
use the embedding $I^{-1} H^{1,s- } \subseteq C^0_t (I^{-1} L^2_x)$
followed by \eqref{i-2}.

Finally, we control $\N_0 = \N_0(\underline \Phi, \underline \Phi)$.
We decompose $\underline \Phi$ smoothly into a low frequency
component $\underline \Phi_{\low}$ supported on $|\xi| \leq 4$ and a
local frequency component $\underline \Phi_{\local}$ supported on
$|\xi| \geq 2$.  The contribution of $\N_0(\underline \Phi_{\local},
\underline \Phi_{\local})$ is acceptable from \eqref{null-I} and
\eqref{phi-ball}.  For the remaining terms we shall not exploit the
null structure, and just write $\N_0(\underline \Phi, \underline
\Phi)$ crudely as $\O( \underline \Phi \nabla_x \underline \Phi )$.

Consider $\underline \Phi_{\low} \nabla_{x,t} \underline
\Phi_{\local}$.  From \eqref{continuity} we have $\nabla_{x,t}
\underline \Phi_{\local} \in CK \Ball(I^{-1} H^{0, s-
}_{[t_0-\delta,t_0+\delta]})$, while from \eqref{phi-ball} and
\eqref{lpr-bernstein} we have $\underline \Phi_{\low} \in CK
\Ball(C^0_t \L^6_4)$.  The claim then follows from \eqref{i-embed}.

Now consider $\underline \Phi_{\local} \nabla_{x,t} \underline
\Phi_{\low}$.  From \eqref{continuity} and \eqref{lpr-bernstein} we
have $\nabla_{x,t} \underline \Phi_{\low} \in CK \Ball(C^0_t
\L^3_4)$, while from \eqref{phi-ball} we have $\underline
\Phi_{\local} \in CK \Ball(I^{-1} H^{1,s-}_{[t_0-\delta, t_0 +
\delta]})$. The claim again follows from \eqref{i-embed}.

Finally we consider $\underline \Phi_{\low} \nabla_{x,t} \underline
\Phi_{\low}$.  As mentioned in the previous paragraphs we have
$\underline \Phi_{\low} \in CK \Ball(C^0_t \L^6_4)$ and
$\nabla_{x,t} \underline \Phi_{\low} \in CK \Ball(C^0_t \L^3_4)$.
The claim then follows from \eqref{lpr-holder}.

This completes the proof of \eqref{box-phi}.  From
\eqref{caricature} we then have
$$
\Box \underline \Phi \in CN^{0+} K^{20} \Ball(I^{-1} H^{0,
s-1}_{[t_0-\delta, t_0+\delta]} + L^2_t \L^3_2).
$$
From \eqref{energy}, \eqref{lpr-energy}, \eqref{gradphi-t0} and the
linearity of $\Box$, we thus have
$$
\nabla_{x,t} \underline \Phi \in C \delta^{0+} N^{0+} K^{20}
\Ball(I^{-1} H^{0,s-}_{[t_0-\delta, t_0+\delta]} + C^0_t \L^3_2).
$$

If we choose $K$ sufficiently large, and $\delta$ sufficiently small
depending on $K$, $N$ (but with $\delta \geq C(K) N^{0-}$), then
\eqref{phi-spacetime} follows, as desired.
\end{proof}

We apply the 
above Proposition with a fixed $K$ sufficiently large, i.e. with $K$  an absolute constant,
following our conventions for such constants.  In particular 
henceforth all implicit constants are allowed to depend on $K$.  As
a corollary of the above argument (specifically
\eqref{phi-spacetime}, \eqref{phi-ball}, \eqref{a0-nab},
\eqref{a-sob}) we have the estimates

\begin{align}
\underline \Phi &\in C \Ball(I^{-1} H^{1,s-}_{[t_0-\delta, t_0+\delta]} + C^0_t \L^6_2) \label{phi-bound}\\
\nabla_{x,t} \underline \Phi &\in C \Ball(I^{-1} H^{0,s- }_{[t_0-\delta,t_0+\delta]} + C^0_t \L^3_2) \label{phit-bound}\\
A_0 &\in C\Ball(I^{-2} H^{2,0}_{[t_0-\delta,t_0+\delta]} + C^0_t \L^6_1) \label{A0-bound}\\
\partial_t A_0 &\in C \Ball(I^{-2} H^{1,0}_{[t_0-\delta,t_0+\delta]} + L^2_t \L^6_1). \label{A0t-bound}
\end{align}
In other words, ignoring the technical low frequency issues,
$I\underline \Phi$ lives in $H^{1,s- }$ and $I^2 A_0$ lives in
$H^{2,0}$, and similarly for the time derivatives (but with one
lower order of regularity, of course).

In the remainder of the paper we use the estimates
\eqref{phi-bound}-\eqref{A0t-bound} to obtain \eqref{h-shift}.

\section{Differentiating the Hamiltonian}\label{hamildiff-sec}

Having obtained control on $A_0$, $\Phi$ on the interval
$[t_0-\delta, t_0+\delta]$, we are now ready to begin the proof of
\eqref{h-shift}. We shall use the real inner product $\langle u, v
\rangle := \Re \int_{\R^3} u(x) \overline{v(x)}\ dx$ throughout this
section.  Since $m$ is real and symmetric we observe that $I$ is
self-adjoint: $\langle Iu, v\rangle = \langle u, Iv \rangle$.
Similarly for $I^{-1}$.

Fix $T \in [t_0-\delta/2, t_0+\delta/2]$.  By the Fundamental
theorem of calculus it suffices to show that
\begin{equation}\label{ftoc}
\left|\int_{t_0}^T \frac{d}{dt} \H[I\Phi[t]]\ dt\right| \lesssim
\frac{1}{N^{(s-1/2)-}}.
\end{equation}

Our next task is to expand the expression
\begin{equation}\label{deriv}
\frac{d}{dt} \H[I\Phi[t]].
\end{equation}
If the $I$ were not present then \eqref{deriv} would vanish.  With
the $I$ present, \eqref{deriv} does not vanish completely, but we
will be able to express \eqref{deriv} in terms of commutators of $I$
and other operators.

Before we do so, let us begin with a heuristic discussion, ignoring
the elliptic term $A_0$ and the null structure.  Since (MKG-CG) is
roughly of the form
\begin{equation}\label{b-c}
\Box \Phi = \O( \Phi \nabla_{t,x} \Phi ) + \O( \Phi \Phi \Phi )
\end{equation}
and the Hamiltonian \eqref{hamil-def} is roughly of the form
$$ \H[\Phi] = \int_{\R^3} \O( \nabla_{t,x} \Phi \nabla_{t,x} \Phi ) + \O( \Phi \Phi \nabla_{t,x} \Phi ) +
\O( \Phi \Phi \Phi \Phi)$$ it seems reasonable to expect an identity
roughly of the form
\begin{equation}\label{h-deriv-heuristic}
\partial_t \H[\Phi] = \langle \Phi_t, \Box \Phi - \O( \Phi \nabla_{t,x} \Phi ) - \O( \Phi \Phi \Phi ) \rangle
\end{equation}
for arbitrary $\Phi$ (not necessarily solving (MKG-CG)), since we
know in advance that the Hamiltonian must be preserved by the flow
(MKG-CG).   In particular we expect
$$ \partial_t \H[I\Phi] = \langle I\Phi_t, \Box I\Phi - \O( I\Phi \nabla_{t,x} I\Phi ) - \O(I\Phi I\Phi I\Phi) \rangle.$$
On the other hand, by applying $I$ to \eqref{b-c}, we have
$$ \Box I \Phi - I\O(\Phi \nabla_{x,t} \Phi) - I\O(\Phi \Phi \Phi) = 0.$$
Inserting this into the previous equation, we expect to split $\partial_t
\H[I\Phi]$ as two commutators:
\begin{equation}\label{ht}
 \partial_t \H[I\Phi] = \langle I\Phi_t, \O( I(\Phi \nabla_{t,x} \Phi) - I\Phi \nabla_{t,x} I\Phi \rangle ) + \langle I \Phi_t, \O( I(\Phi \Phi \Phi) - I\Phi I\Phi I\Phi ) \rangle.
 \end{equation}

We now begin the rigorous argument.  The rigorous form of
\eqref{h-deriv-heuristic} is

\begin{lemma}[First variation of Hamiltonian]\label{identity}
If $\Phi$ is arbitrary (not necessarily solving (MKG-CG)), then
$$
\frac{d}{dt} \H[\Phi[t]] = -\langle F_{\mu 0}, \partial^\alpha
{F_\alpha}^\mu + \Im(\phi \overline{D^\mu \phi}) \rangle - \langle
D_0 \phi, D_\alpha D^\alpha \phi \rangle,
$$
where $D_\alpha$ and $F_{\alpha \beta}$ were defined in
\eqref{conn-def}, \eqref{curv-def}.
\end{lemma}

Observe that this quantity vanishes (as expected) if $\Phi$ solves
(MKG), and in particular if it solves (MKG-CG).

\begin{proof}
We recall the \emph{stress-energy tensor}
$$ \T_{\alpha \beta} := {F_\alpha}^\mu F_{\beta \mu} - \frac{1}{4} \eta_{\alpha \beta} F_{\mu \nu} F^{\mu \nu} + \Re(D_\alpha \phi \overline{D_\beta \phi}) - \frac{1}{2} \eta_{\alpha \beta} \Re(D_\mu \phi \overline{D^\mu \phi}),$$
where $\eta_{\alpha\beta}$ is the Minkowski metric.  From
\eqref{hamil-def} we see that
$$ \H[\Phi[t]] = \int_{\R^3} \T_{00}(x,t)\ dx.$$
Thus
$$ \partial_t \H[\Phi[t]] = -\int_{\R^3} \partial^\alpha \T_{\alpha 0}(x,t)\ dx.$$
To compute the integrand we observe that
$$ \partial_\alpha \Re(u \overline{v}) = \Re(D_\alpha u \overline{v}) + \Re(u \overline{D_\alpha v}).$$
Using this we can expand $\partial^\alpha \T_{\alpha 0}$ as
$$ (\partial^\alpha F_\alpha^\mu) F_{0\mu}
+ {F_\alpha}^\mu \partial^\alpha F_{0\mu} - \frac{1}{2} (\partial_0
F_{\mu \nu}) F^{\mu \nu} +
 \Re(D^\alpha D_\alpha \phi \overline{D_0 \phi})
+  \Re(D_\alpha \phi \overline{D^\alpha D_0 \phi}) -  \Re(D_\mu \phi
\overline{D_0 D^\mu \phi}).$$ Collecting terms and relabeling (using
the anti-symmetry of $F$), we can rewrite the above as
$$ - \frac{1}{2} F^{\mu \nu} (\partial_0 F_{\mu \nu} + \partial_\mu F_{\nu 0} + \partial_\nu F_{0 \mu}) + (\partial^\alpha {F_\alpha}^\mu) F_{0\mu}
+ \Re(D^\alpha D_\alpha \phi \overline{D_0 \phi}) + \Re(D_\alpha
\phi \overline{[D^\alpha, D_0] \phi}).$$ The first term vanishes
from the Bianchi identity $dF = ddA = 0$.  The last term can be
simplified as $[D^\alpha, D_0] = i F^\alpha_0$.  After a little more
collecting terms and relabeling, we obtain
$$ \partial^\alpha \T_{\alpha 0} = (\partial^\alpha {F_\alpha}^\mu + \Im(\phi \overline{D^\mu \phi}) F_{0\mu}
+ \Re(D^\alpha D_\alpha \phi \overline{D_0 \phi})$$ and the Lemma
follows.
\end{proof}

In the Coulomb gauge \eqref{div0-eq}, we can use this lemma to
rewrite \eqref{deriv} as
$$
\frac{d}{dt} \H[I\Phi[t]] = -\langle \partial_j IA_0 - \partial_t I
A_j, \Box I A_j + \partial_j \partial_t IA_0 + \Im(I\phi
\overline{\tilde D_j I\phi}) \rangle - \langle \tilde D_0 I \phi,
\tilde D_j \tilde D_j I\phi - \tilde D_0 \tilde D_0 I\phi \rangle$$
where $\tilde D_\alpha := \partial_\alpha + i (IA_\alpha)$. On the
other hand, by applying $I$ to (MKG) we have
$$ \Box I A_j + \partial_j \partial_t I A_0 + I \Im(\phi \overline{D_j \phi}) = 0$$
and
$$ I (D_0 D_0 \phi - D_j D_j \phi) = 0.$$
Thus one can write \eqref{deriv} as a linear combination of the
commutator expressions
\begin{equation}\label{one}
\bigl\langle \partial_j I A_0 - \partial_t I A_j, I \Im(\phi
\overline{D_j \phi}) - \Im(I\phi \overline{\tilde D_j I\phi})
\bigr\rangle,
\end{equation}
\begin{equation}\label{two}
\bigl\langle \tilde D_0 I \phi, I D_0 D_0 \phi - \tilde D_0 \tilde
D_0 I \phi \bigr\rangle,
\end{equation}
and
\begin{equation}\label{three}
\bigl\langle \tilde D_0 I \phi, I D_j D_j \phi - \tilde D_j \tilde
D_j I \phi \bigr\rangle.
\end{equation}
This should be compared with \eqref{ht}.

We now break up \eqref{one}, \eqref{two}, \eqref{three} further.  We
introduce the nonlinear commutators
\begin{align*}
[I, \N_0] &:= I \N_0(\underline \Phi, \underline \Phi) - \N_0(I\underline \Phi, I\underline \Phi)\\
[I, \N_1] &:= \bigl(I(\partial_t A_0 \underline \Phi) - \partial_t IA_0 I\underline \Phi, I(A_0 \partial_t \underline \Phi) - IA_0 I\partial_t \underline \Phi\bigr) \\
[I, \N_2] &:= I(\underline{\Phi} \nabla_x \underline{\Phi}) - I\underline{\Phi} \nabla_x I\underline{\Phi} \\
[I, \N_3] &:= I(\Phi^3) - (I \Phi)^3
\end{align*}
We also define the ``mollified time derivative''
$$ \D_0 \Phi := (\partial_t I \underline{A}, \tilde D_0 I \phi).$$

In later sections we shall prove the estimates
\begin{equation}\label{rk-1}
\left|\int_{t_0}^t \langle \D_0 \Phi, \O( [I, \N_k] )
\rangle\ dt\right| \lesssim \frac{1}{N^{(s-1/2)-}}
\end{equation}
for $k=0,1,3$, as well as the variant
\begin{equation}\label{rk-2}
\left|\int_{t_0}^t \langle \nabla_x I A_0, \O( [I, \N_k] )
\rangle\ dt\right| \lesssim \frac{1}{N^{(s-1/2)-}}
\end{equation}
for $k = 2,3$.  (In fact, with the exception of the null form
estimate \eqref{rk-1} with $k=0$, we will be able to obtain a decay
of $O(N^{-1/2+})$.) For now we show how these estimates imply
\eqref{one}, \eqref{two}, \eqref{three} are bounded by $O(1/N^{(s-1/2)-})$ as required.

We begin with \eqref{one}.  Consider the contribution of $\partial_j
I A_0$.  We write this term crudely as
$$ \langle \nabla_x I A_0, \O( [I, \N_2] ) + \O( [I, \N_3] )\rangle$$
which is acceptable by \eqref{rk-2}.

Now consider the contribution of $\partial_t I A_j$.  By
\eqref{div0-eq} we may freely insert a projection $\P$ on the left
term of the inner product, and hence on the right by
self-adjointness.  From the definition of the null form $\N_0$ we
can thus write this contribution as
$$ \langle \partial_t I \underline A, \O( [I, \N_0] ) + \O( [I, \N_3] ) \rangle$$
which is acceptable by \eqref{rk-1}.

Now we expand out \eqref{two}.  Observe that $I D_0 D_0 \phi -
\tilde D_0 \tilde D_0 I \phi$ can be expanded as a linear
combination of $[I, \N_1]$ and $[I, \N_3]$, so the claim follows
from \eqref{rk-1}.

Now we expand out \eqref{three}.  From \eqref{div0-eq} we have
$$ I D_j D_j \phi = I \Delta \phi + 2 i I (\P \underline A \cdot \nabla_x \phi) + I(|\underline A|^2 \phi)$$
and
$$ \tilde D_j \tilde D_j I\phi = I \Delta \phi + 2 i (\P I \underline A) \cdot (I \nabla_x \phi) + |I \underline A|^2 I \phi$$
so by the definition of the null form $\N_0$ we have
$$ I D_j D_j \phi - \tilde D_j \tilde D_j I\phi = \O( [I, \N_0] ) + \O( [I, \N_3] )$$
and the claim follows from \eqref{rk-1}.

It remains to prove \eqref{rk-1}, \eqref{rk-2}.  We shall do so in
later sections, but for now we give some estimates on $\D_0 \Phi$.

\begin{lemma}[$\D_0 \Phi$ estimate]\label{sigma-bound-lemma}  We have
\begin{equation}\label{sigma-bound}
\D_0 \Phi \in C\Ball(H^{0,s- }_{[t_0-\delta/2,t_0+\delta/2]} + C^0_t
\L^6_{10}).
\end{equation}
\end{lemma}

\begin{proof}
The bound on $\partial_t I \underline{A}$ follows from
\eqref{phit-bound} and \eqref{lpr-bernstein}. Now we bound $\tilde
D_0 I \phi$.  We split $\tilde D_0 I \phi = (\tilde D_0 I
\phi)_{\low} + (\tilde D_0 I \phi)_{\local}$, where the low term has
Fourier support in $|\xi| \leq 10$ and the local term has Fourier
support in $|\xi| \geq 9$.

The low term is acceptable from \eqref{idphi}
and Sobolev, so we consider the local term.  We split $D_0 I \phi = I \phi_t + i
IA_0 I \phi$.  The $I\phi_t$ component is acceptable from
\eqref{phit-bound}, so it remains to control the local frequency
component $(IA_0 I \phi)_{\local}$ of $IA_0 I \phi$.  Because this
is a lower order term, regularity will not be a major problem, but
there will be some other technical issues related to the time
truncation.

Let $\eta(t)$ be a bump function adapted to $[t_0-\delta,
t_0+\delta]$ which equals 1 on $[t_0-\delta/2, t_0+\delta/2]$.  It
will suffice to show that
$$ \| \eta(t) (IA_0 I \phi)_{\local} \|_{H^{0,s- }} \lesssim 1.$$
From the crude estimate $\langle |\xi| - |\tau| \rangle^{s-}
\lesssim |\xi| + |\tau|$ when $|\xi| > 9$, we have
$$ \| \eta(t) (IA_0 I \phi)_{\local} \|_{H^{0,s-}} \lesssim
\| \nabla_{x,t} (\eta(t) IA_0 I \phi) \|_{L^2_t L^2_x}. $$ By the
Leibnitz rule, it thus suffices to show the quantities
\begin{equation}\label{norms}
\| \eta'(t) IA_0 I \phi \|_{L^2_t L^2_x}, \| \eta(t) (\nabla_{x,t} IA_0) I\phi \|_{L^2_t L^2_x}, \| \eta(t) IA_0 (\nabla_{x,t} I \phi) \|_{L^2_t L^2_x}
\end{equation}
are bounded.

From \eqref{A0-bound}, \eqref{A0t-bound} and Sobolev we have the
crude bound
$$A_0, \nabla_{x,t} A_0 \in C\Ball(L^2_t L^6_x + L^2_t L^3_x)$$
while from \eqref{phi-bound}, \eqref{hsd-energy} and Sobolev we have
$$ I\phi \in C \Ball(L^\infty_t L^6_x) \cap C\Ball(L^\infty_t L^3_x).$$
The first two norms of \eqref{norms} are then bounded by H\"older (since $I$ is
bounded on all the above spaces).  For the last norm we instead use
\eqref{A0-bound}, \eqref{lpr-bernstein}, and Sobolev to obtain
$$ A_0 \in C\Ball(L^2_t L^\infty_x) \cap C\Ball(L^2_t L^6_x)$$
while from \eqref{phit-bound}, \eqref{hsd-energy} and Sobolev we
have
$$ \nabla_{x,t} I \underline \Phi \in C\Ball(C^0_t L^2_x + C^0_t L^3_x).$$
The claim again follows from H\"older.
\end{proof}

The only remaining task is to establish the estimates \eqref{rk-1},
\eqref{rk-2} for various values of $k$.

\section{The cubic commutator $[I, \N_3]$}\label{cubic-sec}

We first prove the estimates \eqref{rk-1}, \eqref{rk-2} for the
cubic commutator $[I, \N_3]$, which is the easiest to handle as
there are no derivatives.  Indeed we will not need the full strength
of the commutator structure here, and we can use very crude Lebesgue
space estimates.  In this case we will obtain a decay of $N^{-1/2+}$
instead of just $\frac{1}{N^{(s-1/2)-}}$.

Smoothly divide $\Phi := \Phi_{\bounded} + \Phi_{\high}$, where
$\Phi_{\high}$ has frequency support on $|\xi| > N/10$ and
$\Phi_{\bounded}$ is supported on $|\xi| < N/5$.

We need the following Strichartz-type estimates.

\begin{lemma}[Strichartz estimates]\label{cubic-est}
On the slab $[t_0,T] \times \R^3$, we have
\begin{align}
\D_0 \Phi, \nabla_x IA_0 &\in C\Ball(C^0_t L^6_x + C^0_t L^2_x + L^2_t L^{9/2}_x )\label{sigma-cubic}\\
\Phi &\in C\Ball(L^3_t L^6_x \cap L^4_t L^6_x \cap L^{8/3}_t L^8_x) \label{838-cubic}\\
\Phi_{\high} &\in CN^{-1/2+} \Ball(C^0_t L^2_x) \cap \Ball(C^0_t L^3_x) \cap
CN^{-1/2+}
\Ball(L^4_t L^4_x) \label{i2-cubic}
\end{align}
\end{lemma}

\begin{proof}
The bound on $\D_0 \Phi$ comes from Lemma \ref{sigma-bound-lemma}
and the crude estimate $H^{0,s-} \subseteq C^0_t L^2_x$.  The bound
on $\nabla_x IA_0$ comes from \eqref{A0-bound} and the crude
embedding $I^{-1} H^{1,0} \subseteq L^2_t L^{9/2}_x$ arising from Sobolev embedding (since $s>5/6$).

For $\underline \Phi$, the bounds in \eqref{838-cubic} come from
\eqref{phi-bound}, \eqref{lpr-bernstein}, and the Strichartz
estimates
$$ I^{-1} H^{1,s- } \subseteq H^{s,1/2+} \subseteq L^3_t L^6_x, L^4_t L^6_x, L^{8/3}_t L^8_x $$
from \eqref{hsd-strichartz}.

For $A_0$, we argue differently.  The low frequency component is
acceptable from \eqref{A0-bound}, \eqref{lpr-bernstein}.  For the
medium and high frequency components, the estimate \eqref{A0-bound}
and Sobolev gives $A_0 \in C\Ball(L^2_t L^\infty_x)$, while
\eqref{A0t-bound} and the fundamental theorem of calculus and Sobolev embedding gives $A_0
\in C\Ball(C^0_t L^2_x \cap C^0_t L^3_x)$.  The claim then follows by
interpolation.

For $\underline \Phi$, the bounds in \eqref{i2-cubic} come from
\eqref{phi-bound}, the observation that
$$ \| \underline \Phi_{\high} \|_{H^{1/2,s- }_{[t_0-\delta,t_0+\delta]}} \lesssim N^{-1/2}
\| \underline \Phi \|_{I^{-1} H^{1,s- }_{[t_0-\delta,t_0+\delta]}},
$$ and the Strichartz embeddings (from \eqref{hsd-strichartz})
$$ H^{1/2,s- } \subseteq C^0_t L^2_x, C^0_t L^3_x, L^4_t L^4_x.$$

For $A_0$, we have from \eqref{A0-bound} and Sobolev that
$$ \| A_{0,\high} \|_{L^2_t L^\infty_x} \lesssim \| A_{0,\high} \|_{H^{3/2+,0}_{[t_0-\delta,t_0+\delta]}} \lesssim N^{-1/2+}
\| A_{0,\high} \|_{I^{-2} H^{2,0}_{[t_0-\delta,t_0+\delta]}} \lesssim
N^{-1/2+},$$ while from \eqref{A0t-bound}, the fundamental theorem of
calculus we have
\begin{align*}
\| A_{0,\high} \|_{C^0_t L^2_x} &\lesssim \| A_{0,\high}(t_0)
\|_{L^2_x} +
\| \partial_t A_{0,\high} \|_{H^{0,0}_{[t_0-\delta,t_0+\delta]}} \\
&\lesssim N^{-1/2}
( \| \nabla_x A_0(t_0)\|_{L^2_x} + \| \partial_t A_{0,\high} \|_{I^{-1}H^{1,0}_{[t_0-\delta,t_0+\delta]}})\\
&\lesssim N^{-1/2}.
\end{align*}
The claims then follow from interpolation.
\end{proof}

We now prove \eqref{rk-1}, \eqref{rk-2}.  From \eqref{sigma-cubic}
and H\"older it will suffice to prove that
$$ [I, \N_3] \in C \Ball(L^1_t L^{6/5}_x \cap L^1_t L^2_x \cap L^2_t L^{9/7}_x).$$
We expand out $[I, \N_3]$ as the sum of eight terms
$$ [I, \N_3] = \sum_{a,b,c \in \{bounded, high \}} I(\Phi_a \Phi_b \Phi_c) - (I\Phi_a) (I \Phi_b) (I \Phi_c).$$
When $a=b=c=bounded$ then $I$ acts like the identity everywhere, and
the summand vanishes.  Thus it will suffice to show that
$$ I(\Phi_a \Phi_b \Phi_c), (I\Phi_a) (I \Phi_b) (I \Phi_c) \in C \Ball(L^1_t L^{6/5}_x) \cap C \Ball(L^1_t L^2_x) \cap C \Ball(L^2_t L^{9/7}_x)$$
whenever at least one of $a,b,c$ is equal to $high$.  By symmetry we
may assume $c=high$.  But then the claim follows from
 \eqref{838-cubic}, \eqref{i2-cubic} and H\"older.
(The operator $I$ and the projections $\Phi \mapsto
\Phi_{\bounded}$, $\Phi \mapsto \Phi_{\high}$ are bounded on every
translation-invariant Banach space.  To get $L^2_t L^{9/7}_x$, one places two factors in $L^4_t L^6_x$, and the last factor in an interpolant of $C^0_t L^2_x$ and $C^0_t L^3_x$.)  This completes the proof of
\eqref{rk-1}, \eqref{rk-2} for the cubic commutator $[I, \N_3]$.

\begin{remark}
It is in fact possible to use more Strichartz estimates to improve
the estimate for $[I, \N_3]$ even further to $N^{-1+}$; this would
be consistent with the results for the cubic nonlinear wave equation
in \cite{kpv:gwp}.  The numerology is as follows.  As we saw above,
it suffices to put the three factors $\Phi$ in $\N_3 = \Phi^3$ in
$L^3_t L^6_x$.  The Strichartz embedding \eqref{hsd-strichartz}
allows this if $\Phi$ is in $H^{2/3,1/2+}$.  But $\Phi$ is in
$H^{1,s-}$ (for medium frequencies at least), so there is 1/3 of a
derivative to spare.  Since there are three factors of $\Phi$, we
thus see that there is about a full derivative of surplus regularity
in $[I, \N_3]$.  From \eqref{gain} one then expects to extract a
gain\footnote{Admittedly, in the above argument only one of the
three factors $\Phi$ could be assumed to be high frequency, however
one should still be able to obtain the full gain of $N^{-1+}$ by
playing around with the Strichartz exponents (e.g. putting the high
frequency factor in $L^2_t L^{\infty-}_x$ and the other two in
$C^0_t L^{2+}_x$), or perhaps by using commutator estimates as we do
with the bilinear commutators below.} of $N^{-1+}$, in principle at
least.
\end{remark}
\section{Frequency interactions of bilinear commutators}\label{freq-sec}

In the remainder of the paper we will prove the estimates
\eqref{rk-1} or \eqref{rk-2} for the bilinear commutators $[I,
\N_0]$, $[I, \N_1]$, $[I, \N_2]$.  Ignoring derivatives and null
forms (which will have no bearing on the discussion in this
section), all the expressions on the left-hand side have the form
\begin{equation}\label{com-1}
\left|\int_{t_0}^T \langle u(t), I(v(t)w(t)) - I(v(t)) I(w(t))
\rangle\ dt\right|.
\end{equation}
On the other hand, the functions $u$, $v$, $w$ which appear here
have different behavior at low, medium and high frequencies.  The
purpose of this section is to decompose the above trilinear
expressions in terms of these three frequency components.

We smoothly split $u = u_{\low} + u_{\med} + u_{\high}$, where $\hat
u_{\low}$ is supported on $|\xi| < 20$, $\hat u_{\med}$ is supported
on $10 < |\xi| < N/5$, and $\hat u_{\high}$ is supported on $|\xi| >
N/10$.  Similarly decompose $v$ and $w$.  We can then split
\eqref{com-1} into 27 terms of the form
\begin{equation}\label{com-piece}
\left|\int_{t_0}^T \langle u_a(t), I(v_b(t) w_c(t)) - I(v_b(t))
I(w_c(t)) \rangle\ dt\right|
\end{equation}
where $a,b,c \in \{low, med, high \}$.

This may look like a lot of terms, but fortunately most of these
terms are zero.  For instance, if neither of $b$ or $c$ is high,
then $I$ acts like the identity and \eqref{com-piece} vanishes.  So
we may assume at least one of $b$, $c$ is high.

Next, we claim that if one of $a, b, c$ is low frequency, then
\eqref{com-piece} vanishes unless the other two indices is high
frequency.  To see this, suppose (for instance) that $a$ was low
frequency and $b$ was low or med frequency. Then we can integrate by
parts and rewrite the above as
$$ \left|
\int_{t_0}^T \langle I(u_a(t)), v_b(t) w_c(t) \rangle - \langle
I(u_a(t) \overline{Iv_b(t)}), w_c(t) \rangle\ dt \right|
$$
at which point all the $I$s act like the identity and so
\eqref{com-piece} vanishes.  Similarly for other permutations.

From this discussion we see that of the 27 terms in the
decomposition, only 9 are non-zero, and they are listed in Figure
\ref{freq-fig}.

\begin{figure}
\begin{tabular}{|l|l|l|l|} \hline
$a$ & $b$ & $c$ & \\
\hline
low & high & high & Lemma \ref{low-freq}\\
high & low & high & Lemma \ref{low-freq}, Lemma \ref{com}\\
high & high & low & Lemma \ref{low-freq}, Lemma \ref{com}\\
\hline
med/high & med & high & Lemma \ref{tri-lemma}\\
med/high & high & high & \eqref{theta}, Lemma \ref{tri-lemma}\\
med/high & high & med & \eqref{theta}, Lemma \ref{tri-lemma}\\
\hline
\end{tabular}
\caption{List of possible frequency interactions for \eqref{com-1},
and the Lemmas and estimates which are useful in each case (although
for the null form $\N_0$ the analysis is more complicated than the
above table suggests).  In most cases, it will not be so important
to distinguish between the cases $a=\med$ and
$a=\high$.  One can also eliminate the med-med-high case by Fourier support considerations, though this does not significantly simplify the argument.}\label{freq-fig}
\end{figure}

We now discuss qualitatively how each of the six cases in Figure
\ref{freq-fig} will be estimated.  In all of the cases, the main
challenge in proving \eqref{rk-1} or \eqref{rk-2} is to obtain the
decay factor $\frac{1}{N^{(s-1/2)-}}$; it is relatively
straightforward to prove these estimates without this decay factor,
but then Proposition \ref{main-prop} will only let us control the
Hamiltonian for times $T = O(1)$, which will not give us global
well-posedness for any $H^s_x$.

To obtain this decay we must use the fact that the high frequencies
are small if measured in rough norms.  In particular, we will make
frequent use of the simple estimate
\begin{equation}\label{gain}
\| u \|_{H^{s-\theta}_x} \lesssim N^{-\theta} \|u \|_{H^s_x}
\end{equation}
valid for all $u$ such that $\hat u$ is only supported in ``high''
frequencies $|\xi| \gtrsim N$, and all $\theta \geq 0$ and $s \in
\R$.  Thus if there is a high frequency term present, we can
sacrifice some of its regularity to obtain the desired gain in $N$.

This will be fairly straightforward in the first three cases of
Figure \ref{freq-fig}, when there are two high frequency terms,
because the low frequency term is smooth and easily estimated.  In
fact for these low frequency cases one can usually improve the decay
estimate to $N^{-1/2+}$ or better.  We use the following two lemmas
to handle the low frequency cases.

\begin{lemma}\label{low-freq}  We have
$$ \eqref{com-1}
\lesssim \| u \|_{L^{q_1}_t H^{s_1}_x} \| v \|_{L^{q_2}_t \L^p_R} \|
w \|_{L^{q_3}_t H^{s_3}_x}$$ for any $u \in L^{q_1}_t H^{s_1}_x$, $v \in L^{q_2}_t \L^p_R$, $w \in L^{q_3}_t H^{s_3}_x$ with $1 \leq p,q_1,q_2,q_3 \leq
\infty$, $1/q_1 + 1/q_2 + 1/q_3 \leq 1$, $s_1 + s_3 \geq 0$, and $R
= O(1)$.  Similarly for permutations of $u$, $v$, $w$.
\end{lemma}

\begin{proof}
By \eqref{lpr-bernstein} we may take $p=\infty$; by lowering $s_3$
if necessary we may assume $s_3 = -s_1$. By a H\"older in time (and
discarding all appearances of the operator $I$, which is bounded on
every Lebesgue and Sobolev space) it suffices to prove the spatial
estimate
$$ | \langle u, vw \rangle|
\lesssim \| u \|_{H^{s_1}_x} \| v \|_{\L^\infty_R} \| w
\|_{H^{-s_1}_x}.$$ We perform a Littlewood-Paley  decomposition $u =
\sum_{k \geq 0} u_k$, where $\hat u_k$ is supported in the region
$\langle \xi \rangle \sim 2^k$.  Similarly split $w = \sum_{k' \geq
0} w_{k'}$.  Observe from the Fourier support of $v$ that $\langle
u_k, v w_{k'} \rangle$ vanishes unless $k' = k + O(1)$.  By H\"older
we may therefore estimate the left-hand side by
$$ \sum_{k,k' \geq 0: k' = k + O(1)} \| u_k \|_{L^2_x} \| v \|_{\L^\infty_R} \| w_{k'} \|_{L^2_x}$$
which is comparable to
$$ \| v \|_{\L^\infty_R} \sum_{k,k' \geq 0: k' = k + O(1)} \| u_k \|_{H^{s_1}_x}  \| w_{k'} \|_{H^{-s_1}_x}.$$
The claim then follows from Cauchy-Schwarz and the almost
orthogonality of the $u_k$ and of the $w_{k'}$ in Sobolev norms.

The claim for permutations follows since the expression $\langle
u(t), v(t) w(t) \rangle$ is essentially invariant under permutations
(the conjugation being irrelevant for the above norms).
\end{proof}

This lemma does not give the decay of $\frac{1}{N^{ ( s - 1/2 )-}}$
directly, but we shall combine it with \eqref{gain} to do so if
there is enough surplus regularity in the $u$ and $w$ variables.
However, even when this surplus regularity is unavailable we can
still obtain this decay if there is a commutator structure. More
precisely:

\begin{lemma}[Commutator estimate]\label{com}  We have
$$ \eqref{com-1} \lesssim N^{-1} \| u \|_{L^{q_1}_t L^2_x} \| v \|_{L^{q_2}_t \L^p_R} \| Iw \|_{L^{q_3}_t L^2_x}$$
for any $u \in L^{q_1}_t L^2_x$, $v \in L^{q_2}_t \L^p_R$, $w \in L^{q_3}_t
L^2_x$ with $1 \leq p,q_1,q_2,q_3 \leq \infty$, $1/q_1 + 1/q_2 +
1/q_3 \leq 1$, and $R = O(1)$.
\end{lemma}

\begin{proof}
By \eqref{lpr-bernstein} it suffices to take $p=\infty$.  By
H\"older's inequality in time, it thus suffices to show the spatial
commutator estimate
$$ |\langle u, I(vw) - vIw \rangle| \lesssim N^{-1} \| u \|_{L^2_x} \| v\|_{\L^\infty_R} \|Iw\|_{L^2_x}$$
for all $u \in L^2_x$, $v \in \L^\infty_R$, $w \in L^2_x$.  (Observe
that $Iv = v$ for $v \in L^\infty_R$).

Let us first assume that $\hat w$ is supported in the annulus $|\xi|
\sim M$ for some dyadic $M$; we will sum in $M$ later.  If $M \ll N$
then $I(vw)-v(Iw) = vw - vw = 0$, so we may assume $M \gtrsim N$.
Under this assumption, we will prove that
\begin{equation}\label{m-comm}
 |\langle u, I(vw) - vIw \rangle| \lesssim M^{-1} m(M) \| u
\|_{L^2_x} \| v\|_{\L^\infty_R} \|w\|_{L^2_x} \lesssim M^{-1} \|
u\|_{L^2_x} \|v\|_{\L^\infty_R} \|w\|_{L^2_x};
\end{equation}
the claim then follows for general $w$ by a Littlewood-Paley
decomposition and the triangle inequality.

It remains to show \eqref{m-comm}.  Fix $M$.  We use Plancherel to
write
$$ (I(vw) - vIw) \hat{}(\xi) = \int_{\xi_1 + \xi_2 = \xi} (m(\xi_1+\xi_2) - m(\xi_2)) \hat v(\xi_1) \hat w(\xi_2).$$
From the support conditions on $v$ and $w$ we may insert some cutoff
functions
$$ (I(vw) - vIw) \hat{}(\xi) = \int_{\xi_1 + \xi_2 = \xi} (m(\xi_1+\xi_2) - m(\xi_2)) a(\xi_1) b(\xi_2) \hat v(\xi_1) \hat w(\xi_2)$$
where $a$ is a bump function adapted to $|\xi| \lesssim 1$ and
$b(\xi_2)$ is a bump function adapted to $|\xi| \sim M$.

From the mean value theorem and the smoothness of $m$ we see that
$$ m(\xi_1+\xi_2) - m(\xi_2) = O(M^{-1} m(M))$$
on the support of $a(\xi_1) b(\xi_2)$.  Moreover, we may write
$$ (m(\xi_1+\xi_2) - m(\xi_2)) a(\xi_1) b(\xi_2) = M^{-1} m(M) c(\xi_1, \xi_2)$$
where $c$ is a bump function of two variables adapted to the region
$|\xi_1| \lesssim 1$, $|\xi_2| \sim M$.

By inverting the Fourier transform again, we obtain
$$ (I(vw) - vIw)(x) = M^{-1} m(M) \int_{\R^3} \int_{\R^3}\check c(y,z) v(x+y) w(x+z)\ dy dz$$
where $\check c$ is the inverse Fourier transform of $c$.  From the
bump function estimates on $c$ and standard integration by parts
computations, we obtain the bounds
$$ |\check c(y, z)| \lesssim M^3 \langle y \rangle^{-100} \langle Mz \rangle^{-100}.$$
Thus by Minkowski's inequality followed by H\"older's inequality,
\begin{equation}
\left\{
\begin{array}{ll}
\| I(vw) - vIw \|_{L^2_x} & \lesssim M^{-1} m(M) \int_{\R^3} \int_{\R^3} M^3 \langle y \rangle^{-100} \langle Mz \rangle^{-100} \| v(\cdot+y)\|_\infty \| w(\cdot+z)\|_{L^2_x}\ dy dz\\
& \lesssim M^{-1} m(M)  \| v\|_\infty \|w\|_{L^2_x}
\end{array}
\right\}
\end{equation}
and the claim \eqref{m-comm} then follows from the Cauchy-Schwarz
inequality.
\end{proof}

Because of this Lemma, the low frequency terms will be quite minor
in comparison to the medium and high frequency interactions,
although they will unfortunately occupy about half of the cases in
the sequel. For the medium and high frequency interactions we shall
often use the estimate \\\\

\begin{lemma}[Trilinear estimate]\label{tri-lemma}

\begin{enumerate}
\renewcommand{\labelenumi}{(\roman{enumi})}
Let $q_{1}$, $q_{2}$, $q_{3}$, be such that $1/q_{1} + 1/q_{2} +
1/q_{3} \leq 1$. Then

\item We have
$$ \eqref{com-1}\lesssim \| u \|_{L^{q_1}_t H^{s_1}_x} \| v \|_{L^{q_2}_t H^{s_2}_x} \| w \|_{L^{q_3}_t H^{s_3}_x}$$
whenever $0 \leq s_1+s_2,s_2+s_3,s_3+s_1$ and $s_1+s_2+s_3
> 3/2$.

\item Under the same assumptions, we also have the variant
$$
\eqref{com-1} \lesssim \| u \|_{L^{q_1}_t H^{s_1}_x} \| I^{-1} v
\|_{L^{q_2}_t H^{s_2}_x} \| Iw \|_{L^{q_3}_t H^{s_3}_x}.$$

\item Let $s_{1}$,$s_{2}$, $s_{3}$, $s'_{1}$, $s'_{2}$,
$s'_{3}$ be such that $s_{1}+ s_{2} \geq 0$, $-s_{3} \leq
\min{(s_{1},s_{2})}$, $-s_{3} < s_{1} + s_{2} -3/2$, $s'_{1} + s'_{2}
\geq 0$, $-s'_{3} \leq \min {(s'_{1},s'_{2})}$ and $-s'_{3} < s'_{1} +
s'_{2} - 3/2$. Then

\begin{equation}
\begin{array}{l}
\eqref{com-1} \lesssim \| I u \|_{H^{s_1}_{x}} \| v \|_{H^{s_2}_{x}}
\| w \|_{H^{s_3}_{x}} + \| u \|_{H^{s'_1}_{x}} \| I v
\|_{H^{s'_2}_{x}} \| I w \|_{H^{s'_3}_{x}}
\end{array}
\end{equation}

\end{enumerate}
\end{lemma}

\begin{proof}
The first inequality is an immediate consequence of H\"older in
time, \eqref{sobolev-mult} and duality.

To prove the second estimate, we use a H\"older in time and the
triangle inequality to reduce to proving the spatial estimates
$$ |\langle u, (Iv) (Iw) \rangle|,
|\langle u, I(vw) \rangle| \lesssim \| u \|_{H^{s_1}_x} \| I^{-1} v
\|_{H^{s_2}_x} \| Iw \|_{H^{s_3}_x}.$$ The former estimate again
follows from \eqref{sobolev-mult} and duality, noting that $I^{-1}
v$ controls $Iv$ in the $H^{s_2}_x$ norm.  To  handle the latter
estimate, we rewrite
$$ \langle u, I(vw) \rangle = \langle Iu, vw \rangle = \langle (Iu) \overline{v}, w \rangle = \langle I^{-1}((I u) \overline v), Iw \rangle,$$
and reduce to showing the bilinear estimate
$$ \| I^{-1}((Iu) \overline v) \|_{H^{-s_3}_x} \lesssim \| u \|_{H^{s_1}_x} \| I^{-1} v \|_{H^{s_2}_x}.$$
We may assume that $u,v$ have non-negative Fourier transform.  From
the pointwise inequality $m(\xi + \eta)^{-1} \lesssim m(\xi)^{-1} +
m(\eta)^{-1}$ and Plancherel, we see that $I^{-1}$ obeys a
fractional Leibnitz rule, so it suffices to show that
$$ \| u \overline v \|_{H^{-s_3}_x} \lesssim \| u \|_{H^{s_1}_x} \| I^{-1} v \|_{H^{s_2}_x}.$$
and
$$ \|(Iu) (I^{-1} \overline v) \|_{H^{-s_3}_x} \lesssim \| u \|_{H^{s_1}_x} \| I^{-1} v \|_{H^{s_2}_x}.$$
But these both follow from \eqref{sobolev-mult} (observing that $u$
controls $Iu$ and $I^{-1} v$ controls $v$; the conjugation is
irrelevant).

To prove the third estimate it is enough to prove

\begin{equation}
\begin{array}{l}
\left| \langle u, I(vw) \rangle \right| \lesssim \| I u
\|_{H^{s_{1}}} \| v \|_{H^{s_{2}}} \| w \|_{H^{s_{3}}}
\end{array}
\nonumber
\end{equation}
and

\begin{equation}
\begin{array}{l}
\left| \langle u, Iv Iw \rangle \right| \lesssim \| u
\|_{H^{s'_{1}}} \| Iv \|_{H^{s'_{2}}} \| I w \|_{H^{s'_{3}}} 
\end{array}
\end{equation}
But, by integration by parts, duality and (\ref{sobolev-mult}) we
have

\begin{equation}
\begin{array}{ll}
\left| \langle u, I(vw) \rangle \right| & = \left| \langle Iu, vw \rangle \right|  \\
& \lesssim \| I u \|_{H^{s_{1}}} \| v w \|_{H^{-s_{1}}} \\
& \lesssim \| I u \|_{H^{s_{1}}} \| v \|_{H^{s_{2}}} \| w
\|_{H^{s_{3}}}
\end{array}
\nonumber
\end{equation}
and

\begin{equation}
\begin{array}{ll}
\left| \langle u, (Iv)(Iw)  \rangle  \right| & \lesssim \| u
\|_{H^{s'_{1}}} \| (I v) (I w) \|_{H^{-s'_{1}}} \\
& \lesssim \| u \|_{H^{s'_{1}}} \| I v \|_{H^{s'_{2}}} \| I w
\|_{H^{s'_{3}}}
\end{array}
\nonumber
\end{equation}
and the claim follows.
\end{proof}

As with Lemma \ref{low-freq}, these estimates when combined with
\eqref{gain} will give the desired decay in $N$ provided that there
is enough surplus regularity in the high frequency factors.

The ``high-high'' interaction, when $b$, $c$ are both high, will
also be relatively easy to handle because there are two high
frequency terms in which one can sacrifice some regularity.  (It
will turn out that the $a$ term usually has no surplus regularity.)
The ``medium-high'' or ``high-medium'' interactions will be more
delicate however, especially if the function associated with the
``high'' frequency is quite rough (e.g. $\nabla_{x,t} \phi$). In
this case there may be no surplus regularity on the high frequency
factor to use, but to compensate for this the medium frequency
factor will have quite a bit of surplus regularity. To exploit this
we will use the commutator structure, and specifically the H\"older
continuity (or mean-value theorem) estimate
\begin{equation}\label{theta}
m(\xi_1 + \xi_2) - m(\xi_2) = O( (\frac{\xi_1}{\xi_2})^\theta
m(\xi_2) )
\end{equation}
for any medium frequency $\xi_1$, high frequency $\xi_2$, and $0
\leq \theta \leq 1$.  Morally speaking, the estimate \eqref{theta}
allows us to transfer\footnote{It is this ability to use the
commutator structure to transfer regularity from smooth factors to
rough ones which distinguishes the methods here from the frequency
truncation method used by Bourgain \cite{borg:book} and later
authors.  In that method one usually has to rely on ``extra
smoothing estimates'' to control the medium-high interactions, but
these estimates are usually only available if there are no
derivatives in the nonlinearity.} up to one full degree of
regularity from the medium frequency factor to the high frequency
factor. (If it were not for the H\"older estimate \eqref{theta}
(which would for instance be the case if $m$ was rough), one would
have to require that $v$ and $w$ can both \emph{individually} come
up with this much surplus regularity; this is roughly equivalent to
the existence of an extra smoothing estimate of the type mentioned
in the introduction.)

We now turn to the specific details for each commutator in turn.

\section{The non-null-form commutators $[I, \N_1]$, $[I, \N_2]$}\label{n1-sec}

Now we prove \eqref{rk-1} for $[I, \N_1]$ and \eqref{rk-2} for $[I,
\N_2]$. Because of the presence of the relatively smooth function
$A_0$, these commutators can be handled by relatively simple tools,
namely H\"older's inequality, the fractional Leibnitz rule, and some
simple commutator estimates.  We will be able to obtain a decay here
of $N^{-1/2+}$, which improves over the claimed decay of
$\frac{1}{N^{ \left(s -\frac{1}{2} \right)-}}$.

We split $\D_0 \Phi$, $A_0$, and $\underline \Phi$ into low, medium,
and high components as in Section \ref{freq-sec}.  It will suffice
to show
\begin{align}
\left|\int_{t_0}^T \bigl \langle (\D_0 \Phi)_a, I(\partial_t A_{0,b} \underline \Phi_c) - (I \partial_t A_{0,b}) (I \underline \Phi_c) \bigr \rangle\ dt\right| &\lesssim N^{-1/2+} \label{a0-1}\\
\left|\int_{t_0}^T \bigl \langle (\D_0 \Phi)_a, I(A_{0,b} \partial_t \underline \Phi_c) - (I A_{0,b}) (I \partial_t \underline \Phi_c) \bigr \rangle\ dt\right| &\lesssim N^{-1/2+}\label{a0-2}\\
\left|\int_{t_0}^T \bigl \langle \nabla I A_{0,a}, I(\underline
\Phi_b \nabla_{x,t} \underline \Phi_c) - (I\underline \Phi_b) (I
\nabla_{x,t} \underline \Phi_c) \bigr \rangle\ dt\right| \lesssim
N^{-1/2+}\label{a0-3}
\end{align}
for all triples $(a,b,c)$ in Figure \ref{freq-fig}.

To prove the above commutator estimates, we will use the following
bounds on the factors $\D_0 \Phi$, $A_0$, $\nabla_{x,t} A_0$,
$\underline \Phi$, $\nabla_{x,t} \Phi$.

\begin{lemma}[Spacetime estimates]\label{ao1-est}
On the slab $[t_0,T] \times \R^3$, we can place the low, medium, and
high components of $\D_0 \Phi$, $A_0$, $\nabla_{x,t} A_0$,
$\underline \Phi$, and $\nabla_{x,t} \underline \Phi$ in the
following spaces:

\begin{tabular}{|l|l|l|l|l|l|}
\hline
a & $\D_0 \Phi_a \in$ & $A_{0,a} \in$ & $\nabla_{x,t} I A_{0,a} \in$ & $\underline \Phi_a \in$ & $\nabla_{x,t} \underline \Phi_a \in$ \\
& $C \Ball(C^0_t \ldots)$ & $C \Ball(L^2_t \ldots)$ & $C \Ball(L^2_t \ldots)$ & $C \Ball(C^0_t \ldots)$ & $C \Ball(C^0_t \ldots)$ \\
\hline
low & $\L^\infty_{10}$ & $\L^\infty_{10}$ & $\L^\infty_{10}$ & $\L^\infty_{10}$ & $\L^\infty_{10}$\\
\hline
med & $IL^2_x$ & $IH^2_x$ & $IH^1_x$ & $IH^1_x$ & $IL^2_x$\\
\hline
high & $L^2_x$ & $I^{-2} H^2_x$ & $I^{-1} H^1_x$ & $I^{-1} H^1_x$ & $I^{-1} L^2_x$ \\
high & $N^{-1/4+} H^{-1/4+}_x$ & $N^{-1/4+} I^{-1} H^{7/4+}_x$ & $N^{-1/4+}
H^{3/4+}_x$ &  $N^{-1/4+} H^{3/4+}_x$  &  $ N^{-1/4+}  H^{-1/4+}_x $ \\
high & $N^{-1/2+} H^{-1/2+}_x$ & $N^{-1/2+} H^{3/2+}_x$ & $N^{-1/2+} IH^{1/2+}_x$ & $N^{-1/2+} IH^{1/2+}_x$ & $N^{-1/2+} IH^{-1/2+}_x$ \\
\hline
\end{tabular}

Thus for instance we have $\partial_t I A_{0,\high} \in C N^{-1/2+}
\Ball(L^2_t IH^{1/2+}_x)$.
\end{lemma}

We may discard the $I$ from the above spaces if desired thanks to
the trivial embeddings $I H^\alpha \subseteq H^\alpha \subseteq
I^{-1} H^\alpha$ for any $\alpha \in \R$.

\begin{proof}
These follow from Lemma \ref{sigma-bound-lemma}, \eqref{phi-bound},
\eqref{phit-bound}, \eqref{A0-bound}, and \eqref{A0t-bound}.  For
the low frequency terms we use \eqref{lpr-bernstein} and
\eqref{hsd-energy}.  For the medium and high frequency terms we use
\eqref{hsd-energy} and \eqref{gain}, as well as the hypothesis
$\frac{\sqrt{3}}{2} < s < 1$ (actually, $ 5/6  <s < 1$ suffices). Note
that the operator $I$ is the identity on medium frequencies, so its
presence there is harmless.
\end{proof}

We can now motivate the numerology behind the decay of $N^{-1/2+}$.
Consider the commutators \eqref{a0-1}, \eqref{a0-3}, which are
roughly of the form $\int_{[t_0,T] \times \R^3} \O( \nabla_{t,x} A_0
\underline \Phi \nabla_{t,x} \underline \Phi )$.  From the above we
see that the three factors are in $H^1_x$, $L^2_x$, and $H^1_x$, for
medium frequencies at least.  Lemma \ref{tri-lemma} then allows us
to estimate the above trilinear expression.  In fact we have about
half a derivative to spare; even if we reduced the regularity of one
of the $H^1$ factors to $H^{1/2+}$, we could still use Lemma
\ref{tri-lemma}.  The idea is to then use \eqref{gain} to convert
this half derivative of room to a $N^{-1/2+}$ factor in the
estimates\footnote{Indeed, one could perhaps improve this factor
even further by exploiting the room available in the time index.
Currently we are estimating one factor in $L^2_t$ and the other two
in $C^0_t$.  By using Strichartz estimates (cf. Section
\ref{cubic-sec}, or \cite{kpv:gwp}) one might be able to sacrifice
integrability in time for regularity in space, which might then be
convertible to further gains in $N$.}.  The case of \eqref{a0-2} is
similar; the three factors are now in $L^2_x$, $H^2_x$, $L^2_x$ but
there is still the half of derivative of surplus regularity which
one can hope to convert to a $N^{-1/2+}$ gain, by using \eqref{gain}
(and in some cases \eqref{theta}).

Unfortunately, there are a number of minor differences between
\eqref{a0-1}, \eqref{a0-2}, and \eqref{a0-3} which require separate
treatment.  To systematize the numerous cases we shall use a number
of tables.

To prove \eqref{a0-1} (which is the easiest case) for each of the
six cases in Figure \ref{freq-fig} we use the norms and Lemmas
indicated in Figure \ref{a01-fig}.  For low frequency interactions
we use Lemma \ref{low-freq}, while for medium and high frequency
interactions we use Lemma \ref{tri-lemma}(i).

\begin{figure}
\begin{tabular}{|l|l|l|l|l|l|l|}
\hline
$a$ & $b$ & $c$ & $\D_0 \Phi_a \in $ & $\partial_t A_{0,b} \in $ & $\underline \Phi_c \in $ & Lemma\\
 & &  & $C\Ball(C^0_t \ldots)$ & $C\Ball(L^2_t \ldots)$ & $ C \Ball(C^0_t \ldots)$ & \\
\hline
low & high & high & $\L^\infty_{10}$ & $N^{-1/2+} H^{1/2+}_x$ & $N^{-1/2+} H^{1/2+}_x$ & \ref{low-freq}\\
high & low & high & $L^2_x$ & $\L^\infty_{10}$ & $N^{-1/2+} H^{1/2+}_x$ & \ref{low-freq} \\
high & high & low & $L^2_x$ & $N^{-1/2+} H^{1/2+}_x$ & $\L^\infty_{10}$ & \ref{low-freq} \\
\hline
med/high & med & high & $L^2_x$ & $H^1_x$ & $N^{-1/2+} H^{1/2+}_x$ & \ref{tri-lemma}(i) \\
med/high & high & high & $L^2_x$ & $N^{-1/4+} H^{3/4+}_x$ & $N^{-1/4+} H^{3/4+}_x$ & \ref{tri-lemma}(i)\\
med/high & high & med & $L^2_x$ & $N^{-1/2+} H^{1/2+}_x$ & $H^1_x$ & \ref{tri-lemma}(i)\\
\hline
\end{tabular}
\caption{List of possible cases for \eqref{a0-1}, the spaces in
which to estimate the three factors, and the Lemma used to obtain
the estimate.  In this case smoothing effect of $I$ or the
commutator structure does not need to be exploited.  Observe that in
all six cases the product of the three norms is $O(N^{-1/2+})$.}
\label{a01-fig}
\end{figure}

The proof of \eqref{a0-2} is a little trickier because of the low
regularity of $\partial_t \phi$.  We tackle the six cases in Figure
\ref{freq-fig} using the spaces and Lemmas in Figure \ref{a02-fig}
(concatenating the fifth and sixth cases).  Five of the cases are
straightforward applications of the Lemmas of the previous section
and will not be discussed further.  The one case which is
interesting is Case 4, when $a$ is medium or high, $b$ is medium,
and $c$ is high. By a H\"older in time it suffices to show the
commutator estimate
$$ |\langle u, I(vw) - v Iw \rangle| \lesssim N^{-1/2+} \| u\|_{L^2_x} \| v \|_{H^2_x} \| I w \|_{L^2_x}$$
where $v$ has medium frequency and $w$ has high frequency.  (Note that Lemma \ref{com} is not available to us here because $v$ is not low frequency.)  Since
all the norms on the right-hand side are $L^2_x$ based we may assume
that $\hat u$, $\hat v$, $\hat w$ are non-negative.  By
\eqref{theta} with $\theta = 1/2-$ we then have
$$ |\langle u, I(vw) - v Iw \rangle| \lesssim |\langle u, (|\nabla_x|^{1/2-} v)  (|\nabla_x|^{-1/2+} Iw) \rangle|.$$
But from \eqref{gain} we have $\| |\nabla_x|^{-1/2+} Iw \|_{L^2_x}
\lesssim N^{-1/2+} \| Iw \|_{L^2_x}$, and from Sobolev embedding we
have $\| |\nabla_x|^{1/2-} v \|_{L^\infty_x} \lesssim \| v
\|_{H^2}$.  The claim follows.

\begin{figure}
\begin{tabular}{|l|l|l|l|l|l|l|}
\hline
$a$ & $b$ & $c$ & $\D_0 \Phi_a \in$ & $A_{0,b} \in$ & $\partial_t \underline \Phi_c \in$ & Lemma\\
 & & & $C \Ball(C^0_t \ldots)$ & $C \Ball(L^2_t \ldots)$ & $C \Ball(C^0_t \ldots)$ & \\

\hline
low & high & high & $\L^\infty_{10}$ & $N^{-1/2+} H^{3/2+}_x$ & $N^{-1/2+} H^{-1/2+}$ & \ref{low-freq}\\
high & low & high & $L^2_x$ & $\L^\infty_{10}$ & $I^{-1} L^2_x$ & \ref{com}\\
high & high & low & $L^2_x$ & $N^{-1/2+} H^{3/2+}_x$ & $\L^\infty_{10}$  & \ref{low-freq} \\
\hline
med/high & med & high & $L^2_x$ & $H^2_x$ & $I^{-1} L^2_x$ & \eqref{theta} \\
med/high & high & med/high & $L^2_x$ & $N^{-1/2+} IH^{3/2+}_x$ &
$I^{-1} L^2_x$ & \ref{tri-lemma}(ii)\\
\hline
\end{tabular}
\caption{List of possible cases for \eqref{a0-2}, the norms in which
to estimate the three factors, and a very brief description of the
techniques used in the estimate.  When Lemma \ref{low-freq} or Lemma
\ref{tri-lemma} is used, the product of the three norms is
$O(N^{-1/2+})$; in the other two cases the decay in $N$ comes
instead from commutator estimates.} \label{a02-fig}
\end{figure}

Finally, we prove \eqref{a0-3}.  We now argue as before, except that
we now must re-shuffle the six cases of Figure \ref{freq-fig}
because $\nabla_x IA_0$ has different behaviour at medium and high
frequencies. Similar to the previous section, the six cases in
Figure \ref{freq-fig} can now be handled using the spaces and Lemmas
in Figure \ref{n2-fig}.

\begin{figure}

\begin{tabular}{|l|l|l|l|l|l|l|}
\hline
$a$ & $b$ & $c$ & $\nabla_x IA_{0,a} \in$ & $\underline \Phi_b \in$ & $\nabla_{x,t} \underline \Phi_c \in$ & Lemma\\
 & & & $C \Ball(L^2_t \ldots)$ & $C \Ball(C^0_t \ldots)$ & $C \Ball(C^0_t \ldots)$ & \\

\hline
low & high & high & $\L^\infty_{10}$ & $N^{-1/2+} H^{1/2+}_x$ & $N^{-1/2+} H^{-1/2+}$ & \ref{low-freq}\\
high & low & high & $N^{-1/2+} H^{1/2+}_x$ & $\L^\infty_{10}$ & $N^{-1/2+} H^{-1/2+}_x$ & \ref{low-freq}\\
high & high & low & $N^{-1/2+} H^{1/2+}_x$ & $N^{-1/2+} H^{1/2+}_x$ & $\L^\infty_{10}$  & \ref{low-freq} \\
\hline
med & med & high & $H^1_x$ & $  H^1_x$ & $N^{-1/2+} H^{-1/2+}_x$ & \ref{tri-lemma}(i)\\
high & med & high & $N^{-1/2+} H^{1/2+}_x$ & $ I H^1_x$ & $I^{-1} L^2_x$ & \ref{tri-lemma}(ii)\\
med & high & med/high & $H^1_x$ & $N^{-1/2+} IH^{1/2+}_x$ &
$I^{-1} L^2_x$ & \ref{tri-lemma}(ii)\\
high & high & med/high & $ I^{-1} H^{1} $ and & $ N^{-1/4+}
H^{3/4+}_x$ and
& $ N^{-1/4+} H^{-1/4+}_x $ and & \ref{tri-lemma}(iii) \\
high & high & med/high & $N^{-1/4+} H^{3/4+}_x $ & $ N^{-1/4+}
H^{3/4+}_x $ & $I^{-1} L^{2}$ & \ref{tri-lemma}(iii)\\
\hline
\end{tabular}

\caption{List of possible cases for \eqref{a0-3}, the norms in which
to estimate the three factors, and a very brief description of the
techniques used in the estimate. In all cases the product of the
three norms is $O(N^{-1/2+})$.} \label{n2-fig}
\end{figure}

\section{The null form commutator $[I, \N_0]$}\label{n0-sec}

We now prove the most difficult commutator estimate, namely the
estimate \eqref{rk-1} for the null form commutator $[I, \N_0]$. Here
we shall need the full strength of the $H^{s,b}$ spaces, and in
particular the fact that the ``$b$'' index is $s-$ and not just
$1/2+$.  Here is the one case where we will only be able to obtain a
decay of $\frac{1}{N^{\left( s-\frac{1}{2} \right)-}}$ instead of
$N^{-1/2+}$. Unsurprisingly we shall also need null form estimates
for these spaces.  The time localization to the interval $[t_0, T]$
has been ignored up until now (because we have always done a
H\"older in time anyway) but is now a major technical nuisance, as
multiplication by sharp time cutoffs destroys the ``$b$'' index of
regularity.

The major difficulty with this estimate is with the $\nabla^{-1}
Q(\underline \Phi, \underline \Phi)$ component of $\N_0$, because
there is no extra regularity in the $s$ index in any of the factors
to be sacrificed to obtain the decay in $N$.  However, there is some
extra regularity in the $b$ index which can (after much work) be
exploited as a substitute.  Informally, the strategy is as follows.
If at least one of the factors is low or medium frequency then one
can obtain the decay in $N$ through commutator estimates.  Now
suppose all factors are high frequency.  We look at the spacetime
Fourier transform of all three factors.  If at least one of them is
far away ($\gtrsim N$) from the light cone then one can exploit the
additional room in the $b$ index to obtain the gain.  The only
remaining possibility is when all three frequencies are close to the
light cone, but this means that their frequencies must be close to
parallel (since they must add up to zero), at which point one can
obtain some gain from the null structure.

It will suffice to prove that
$$ \left| \int_\R 1_{[t_0,T]}(t) \bigl\langle (\D_0 \Phi)_a, I\N_0(\underline \Phi_b, \underline \Phi_c) - \N_0(I\underline \Phi_b, I\underline \Phi_c)\bigr\rangle \ dt\right| \lesssim
\frac{1}{N^{ ( s - 1/2 ) -} }$$ for all $a,b,c$ as in Figure
\ref{freq-fig}.

We first deal with the low frequency terms when one of $a$, $b$, $c$
is low frequency.  For this term we shall abandon the null structure
and just prove
$$ \left| \int_\R 1_{[t_0,T]}(t) \bigl\langle (\D_0 \Phi)_a, I(\underline \Phi_b \nabla_x \underline \Phi_c) - (I\underline \Phi_b) \nabla_x (I\underline \Phi_c)\bigr\rangle \ dt\right|
\lesssim N^{-1/2+} \lesssim \frac{1}{N^{\left( s - \frac{1}{2}
\right) -}} $$

We argue using the following table; the estimates on $\D_0 \Phi$ and
$\underline  \Phi$ come from Lemma \ref{sigma-bound-lemma} and
\eqref{phi-bound} respectively. Note that the arguments are almost
identical to those in the previous section.

\begin{tabular}{|l|l|l|l|l|l|l|}
\hline
$a$ & $b$ & $c$ & $(\D_0 \Phi)_a \in$ & $\underline \Phi_b \in$ & $\nabla_{x,t} \underline \Phi_c \in$ & Lemma\\
 & & & $C \Ball(C^0_t \ldots)$ & $C \Ball(C^0_t \ldots)$ & $C \Ball(C^0_t \ldots)$ & \\
\hline
low & high & high & $\L^\infty_{10}$ & $N^{-1/2+} H^{1/2+}_x$ & $N^{-1/2+} H^{-1/2+}$ & \ref{low-freq}\\
high & low & high & $L^2_x$ & $\L^\infty_{10}$ & $L^2_x$ & \ref{com}\\
high & high & low & $N^{-1/2+} H^{-1/2+}_x$ & $N^{-1/2+} H^{1/2+}_x$ & $\L^\infty_{10}$  & \ref{low-freq} \\
\hline
\end{tabular}

Now suppose that $a, b, c$ are in one of the three remaining cases
in Figure \ref{freq-fig}.  Since we have dealt with all low
frequency issues there will no longer be a need to distinguish
between $|\xi|$ and $\langle \xi \rangle$, or between homogeneous
and inhomogeneous Sobolev norms, etc.

The Sobolev estimates which sufficed for all the other commutators
will not work here, and we must exploit the null structure.  Since
$[t_0,T]$ is contained inside $[t_0-\delta/2, t_0+\delta/2]$, it
will suffice from \eqref{phi-bound}, Lemma \ref{sigma-bound-lemma},
to prove the global spacetime estimate
\begin{equation}\label{glob-abc}
\begin{split}
&\left|\int_\R 1_{[t_0,T]}(t) \langle u_a, I\N_0(v_b, w_c) - \N_0(Iv_b, Iw_c)\rangle \ dt\right| \\
&\lesssim \frac{1}{N^{ ( s - 1/2 )-}} \| u_a \|_{H^{0,s- }} \| Iv_b
\|_{H^{1,s-}} \| Iw_c \|_{H^{1, s- }}
\end{split}
\end{equation}
where $u_a$, $v_b$, $w_c$ are supported in the Fourier regions
corresponding to $a$, $b$, $c$.

Now that we are working globally in spacetime we are able to use the
spacetime Fourier transform.  We may assume that the spacetime
Fourier transforms of $u_a$, $v_b$, $w_c$ are all real and
non-negative.

From \eqref{n1} we have
$$ |\widetilde{I\N_0(v_b,w_c)}(\tau,\xi)| \lesssim \int_{\xi = \xi_1 + \xi_2; \tau = \tau_1 + \tau_2} m(\xi_1 + \xi_2)
(\frac{|\xi_1 \wedge \xi_2|}{|\xi_1|} + \frac{|\xi_1 \wedge
\xi_2|}{|\xi_1+\xi_2|}) \widetilde{v_b}(\xi_1,\tau_{1}) \widetilde{
w_c}(\xi_2,\tau_{2})
$$

and
$$ |\widetilde{\N_0(Iv_b,Iw_c)}(\tau,\xi)| \lesssim \int_{\xi = \xi_1 + \xi_2; \tau = \tau_1 + \tau_2} m(\xi_1) m(\xi_2)
(\frac{|\xi_1 \wedge \xi_2|}{|\xi_1|} + \frac{|\xi_1 \wedge
\xi_2|}{|\xi_1+\xi_2|}) \widetilde{v_b}(\xi_1,\tau_{1}) \widetilde{
w_c}(\xi_2,\tau_{2})
$$
where $\xi_1 \wedge \xi_2$ is the anti-symmetric tensor with
components $\xi_1^i \xi_2^j - \xi_1^j \xi_2^i$.  Moreover, we have
\begin{align*} &|(I\N_0(v_b,w_c) - \N_0(Iv_b, Iw_c))^\sim(\tau,\xi)| \\
&\lesssim \int_{\xi = \xi_1 + \xi_2; \tau = \tau_1 + \tau_2}
|m(\xi_1 + \xi_2)-m(\xi_1) m(\xi_2)| (\frac{|\xi_1 \wedge
\xi_2|}{|\xi_1|} + \frac{|\xi_1 \wedge \xi_2|}{|\xi_1+\xi_2|}) \hat
v_b(\xi_1,\tau_1) \hat w_c(\xi_2,\tau_2).
\end{align*}

Also, we have
$$ |\hat 1_{[t_0,T]}(\tau)| \lesssim \langle \tau \rangle^{-1}.$$
By Parseval, we can therefore estimate the left-hand side of
\eqref{glob-abc} by
$$ \int_* \frac{w(\xi_1,\xi_2) |\xi_1 \wedge \xi_2| (|\xi_0|^{-1} + |\xi_1|^{-1}) \hat u_a(\xi_0,\tau_{0})
\hat v_b(\xi_1,\tau_{1}) \hat w_c(\xi_2,\tau_{2})}{ \langle \tau_0 +
\tau_1 + \tau_2 \rangle} $$ where $\int_*$ denotes an integration
over all $(\tau_0,\xi_0), (\tau_1, \xi_1), (\tau_2, \xi_2) \in \R
\times \R^3$ with $\xi_0 + \xi_1 + \xi_2 = 0$, and $w(\xi_1,\xi_2)$
is the symbol
$$ w(\xi_1,\xi_2) := \frac{|m(\xi_1+\xi_2) - m(\xi_1)m(\xi_2)|}{m(\xi_1) m(\xi_2)}.$$

We perform a Littlewood-Paley decomposition. Let $N_{j}, \, j=0,1,2$
range over dyadic numbers, with $N_j$ the nearest dyadic number to $|\xi_j|$, and let $\{\min,\med,\max\}$ be a
permutation of $\{0,1,2\}$ such that $N_{\min} \lesssim N_{\med}
\lesssim N_{\max}$.  Observe that $N_{\med} \sim N_{\max}$ since
$\xi_0 + \xi_1 + \xi_2 = 0$, and that $N_{\min} \gtrsim 1$ since
none of $a$, $b$, $c$ are low.  Also we have $N_{\max} \gtrsim N$
since at least one of $b, c$ is high. Clearly we have $|\xi_0|^{-1}
+ |\xi_1|^{-1} \lesssim N_{\min}^{-1}$.  Also write $\lambda_i :=
\langle |\xi_i| - |\tau_i| \rangle$ for $i=0,1,2$. Therefore, from
these notations and the definition of the $H^{s,b}$ norms, it will
thus suffice, by the Cauchy-Schwarz inequality, to prove
\begin{equation}\label{monster}
\begin{split}
\int_{*; N_{\min} \gtrsim 1,N_{max} \sim N_{med} \gtrsim N} & \frac{
|w(\xi_1, \xi_2)| |\xi_1 \wedge \xi_2| \prod_{j=0}^2
F_j(\xi_j,\tau_j)} {N_{\min} N_1 N_2 \lambda_0^{s-} \lambda_1^{s-}
\lambda_2^{s-} \langle \tau_0 + \tau_1 + \tau_2 \rangle}
\\
&\lesssim \frac{1}{N^{ \left( s-\frac{1}{2} \right)-}} \prod_{j=0}^2
\| F_j \|_{L^2_{\tau_j} L^2_{\xi_j}}
\end{split}
\end{equation}
for all non-negative functions $F_j$. 

Next, we recall the standard estimate for the null form symbol $|\xi_1 \wedge \xi_2|$.

\begin{lemma}[Symbol bound]\label{nf}
We have
$$|\xi_1 \wedge \xi_2| \lesssim N_{0}^{1/2} N_{1}^{1/2}
N_{2}^{1/2} (\langle \tau_{0} + \tau_{1} + \tau_{2} \rangle +
\lambda_{0} + \lambda_{1} + \lambda_{2})^{1/2}.
$$
\end{lemma}

\begin{proof}  See \cite[Proposition 1]{kl-mac:hsd} or \cite[Proposition 8.1]{kl-sel:hsd}.
\end{proof}

From this lemma we see that to prove \eqref{monster} it will suffice
to show that the expressions
\begin{equation}\label{monster-2}
\int_{*; N_{\min} \gtrsim 1, N_{max} \sim N_{med} \gtrsim N} \frac{
w(\xi_1,\xi_2) N_0^{1/2} \prod_{j=0}^2 F_j(\xi_j,\tau_j) } {N_{\min}
N_1^{1/2} N_2^{1/2} \lambda_0^{s-} \lambda_1^{s-} \lambda_2^{s-}
\langle \tau_0 + \tau_1 + \tau_2 \rangle^{1/2}}
\end{equation}
and
\begin{equation}\label{monster-3}
\int_{*; N_{\min} \gtrsim 1, N_{max}\sim N_{med} \gtrsim N} \frac{
w(\xi_1,\xi_2) N_0^{1/2} \prod_{j=0}^2 F_j(\xi_j,\tau_j) } {N_{\min}
N_1^{1/2} N_2^{1/2} \lambda_j^{( s- 1/2 )-} \lambda_k^{s-}
\lambda_l^{s-} \langle \tau_0 + \tau_1 + \tau_2 \rangle }
\end{equation}
are both $O \left( \frac{1}{N^{ ( s - 1/2 )-}} \right) \prod_{j=0}^2
\| F_j \|_{L^2_{\tau_j} L^2_{\xi_j}}$ for all permutations
$\{j,k,l\}$ of $\{0,1,2\}$. In fact, we shall study several cases
and show that all of them (except one) are $ O \left(
\frac{1}{N^{\frac{1}{2}-}} \right) $.

For any $\alpha > 0$, let $a_\alpha(t)$ denote the Fourier transform
of $\langle \tau \rangle^{-\alpha}$; this function appears
implicitly in \eqref{monster-2}, \eqref{monster-3} for $\alpha =
1/2, 1$.  We shall need the following $L^p_t$ estimates:

\begin{lemma}\label{a-ft}  Let $0 < \alpha \leq 1$.  Then $a_\alpha \in L^p_t$ for all $p < 1/(1-\alpha)$.
\end{lemma}

\begin{proof}
We dyadically decompose $\check{a_\alpha}$ as $\sum_{k \geq 0}
2^{-k\alpha} \psi_k \check{a_\alpha}$ where each $\psi_k$ is a bump
function adapted to $[-2^k, 2^k]$ (possibly depending on $\alpha$).
Observe that $\hat \psi_k(t) = O( \langle 2^k t \rangle^{-M})$ for
any large $M > 0$. If one sums this, we see that $a_{\alpha}(t)$ is
rapidly decreasing for $|t| \gtrsim 1$. For $|t| \lesssim 1$, we
have $a_\alpha(t) = O(|t|^{\alpha-1})$ if $\alpha < 1$ and
$a_\alpha(t) = O( |\log|t||)$ if $\alpha=1$. The claim follows.
\end{proof}

We now estimate \eqref{monster-2}, \eqref{monster-3} separately in
the cases $\min = 0$ and $\min \neq 0$, giving four cases. In all
cases except Case 4, we will be able to obtain a decay of
$O(N^{-1/2+})$.

\textbf{Case 1: The bounding of \eqref{monster-2} when $\min = 0$.}

We use the crude bound (valid for all $\xi_0, \xi_1, \xi_2$)
\begin{equation}\label{m-bound}
w(\xi_1,\xi_2) \lesssim \frac{1}{m(N_1) m(N_2)} \lesssim
\frac{1}{m(N_{\max})^2} \lesssim N^{-1/2+} N_{\max}^{1/2-}
\end{equation}
to bound \eqref{monster-2} by
$$
N^{-1/2+} \int_{*; N_{\min} \gtrsim 1, N_{max} \sim N_{med} \gtrsim
N} \frac{ \prod_{j=0}^2 F_j(\xi_j,\tau_j) } {N_0^{1/2+} N_1^{1/2}
N_2^{0} \lambda_0^{s-} \lambda_1^{s-} \lambda_2^{s-} \langle \tau_0
+ \tau_1 + \tau_2 \rangle^{1/2}} ;$$

undoing the Fourier transform, it thus suffices to show
$$
 \left|\int_{\R \times \R^3} u_0 u_1 u_2 a_{1/2}\ dx dt\right| \lesssim \| u_0 \|_{H^{1/2+,s- }}
\| u_1 \|_{H^{1/2,s- }} \| u_2 \|_{H^{0,s- }}.
$$
But this follows from the Strichartz estimates $H^{1/2+,s-}
\subseteq L_{t}^{4+} L_{x}^{4}$, together with the embedding $H^{0,s-} \subseteq
L_{t}^{\infty-} L_{x}^{2}$, which follows from interpolation between $H^{0,0} = L^2_t L^2_x$ and $H^{0,1/2+} \subset C^0_t L^2_x$.

\textbf{Case 2: The bounding of \eqref{monster-2} when $\min \neq
0$.}

By symmetry we may assume $\min = 1$.  In this case we can bound
\begin{align*}
w(\xi_1,\xi_2) &\lesssim \frac{m(N_0) + m(N_1)m(N_2)}{m(N_1) m(N_2)} \\
&\lesssim \frac{m(N_{\max}) + m(N_{\min}) m(N_{\max})}{m(N_{\min}) m(N_{\max})} \\
&\lesssim m(N_{\min})^{-1} \\
&\lesssim N^{-1/2+} N_{\min}^{1/2-}.
\end{align*}
We can thus estimate \eqref{monster-2} by
$$
N^{-1/2+} \int_{*; N_{\min} \gtrsim 1} \frac{ \prod_{j=0}^2
F_j(\xi_j,\tau_j) } {N_1^{1+} \lambda_0^{s-} \lambda_1^{s-}
\lambda_2^{s-} \langle \tau_0 + \tau_1 + \tau_2 \rangle^{1/2}}
$$
so it would suffice to show that

$$ \left| \int_{\R \times \R^3} a_{1/2} u_0 u_1 u_2  dx dt
\right| \lesssim \| u_0 \|_{H^{0,s-}} \| u_1 \|_{H^{1+,s- }} \| u_2
\|_{H^{0,s-}}.
$$
But this follows from the Strichartz embeddings $H^{0,s-} \subset
L_{t}^{\infty} L_{x}^{2}$ and $H^{1+,s-} \subset L_{t}^{2+}
L_{x}^{\infty}$ (from \eqref{hsd-strichartz}), and Lemma \ref{a-ft}.

\textbf{Case 3: The bounding of \eqref{monster-3} when $\min = 0$.}

Since the $j=1$ and $j=2$ cases are symmetric, we may assume that $j
= 0$ or $j=1$.

We again use \eqref{m-bound} to bound \eqref{monster-3} by
$$
N^{-1/2+} \int_{*; N_{\min} \gtrsim 1,N_{max} \sim N_{med} \gtrsim
N} \frac{ \prod_{j=0}^2 F_j(\xi_j,\tau_j) } {N_0^{1/2} N_2^{1/2+}
\lambda_j^{\left( s - \frac{1}{2} \right) -} \lambda_k^{s- }
\lambda_l^{s-} \langle \tau_0 + \tau_1 + \tau_2 \rangle} .
$$
Undoing the Fourier transform and considering the $j=0$ and $j=1$
cases separately, it thus suffices to show
$$ \left|\int_{\R \times \R^3} u_0 u_1 u_2 a_1\ dx dt\right| \lesssim \| u_0 \|_{H^{1/2, ( s-1/2 )-}}
\| u_1 \|_{H^{0,s- }} \| u_2 \|_{H^{1/2+,s- }}. $$ and
$$ \left|\int_{\R \times \R^3} u_0 u_1 u_2 a_1\ dx dt\right| \lesssim \| u_0 \|_{H^{1/2,s- }}
\| u_1 \|_{H^{0, ( s- 1/2 )-}} \| u_2 \|_{H^{1/2+,s- }}.
$$
For the latter estimate, we place $u_0$ in $L^4_t L^4_x$, $u_1$ in $L^2_t L^2_x$, and $u_2$ in $L^{4+}_t L^4_x$, using \eqref{hsd-strichartz} and Lemma \ref{a-ft}.  Now we turn to the former.  To avoid excessive notation, we will pretend that $a_1$ is compactly supported rather than rapidly decreasing; the rapidly decreasing case can then be handled by a routine dyadic decomposition.  We may now assume that $u_0,u_1,u_2$ are compactly supported in time.  From \eqref{no-null} we have
$$
\| u_1 u_2 \|_{H^{-1/2, 0}} \lesssim \| u_1 \|_{H^{0, s-}} \| u_2 \|_{H^{1/2, s}}
$$
so it suffices to show that
$$ \left|\int_{\R \times \R^3} u v a_1\ dx dt\right| \lesssim \| u \|_{H^{1/2, ( s-1/2 )-}}
\| v \|_{H^{-1/2,0}}.$$ 
But this is easily established from Plancherel's theorem.

\textbf{Case 4: The bounding of \eqref{monster-3} when $\min \neq
0$.}

By symmetry we may take $\min=1$.  We may assume that $N_1 \ll N_0,
N_2$ since otherwise we could take $\min=0$ and be in Case 3.  We
can then bound \eqref{monster-3} by
$$ \int_{*; 1 \lesssim N_1 \ll N_0,N_2}
\frac{ w(\xi_1,\xi_2) \prod_{j=0}^2 F_j(\xi_j,\tau_j)} {N_1^{3/2}
\lambda_j^{( s-1/2)-} \lambda_k^{s-} \lambda_l^{s-} \langle \tau_0 +
\tau_1 + \tau_2 \rangle} . $$
Using Cauchy-Schwarz, we may fix $N_0$ and $N_2$; but we will still need to sum in $N_1$.

There are two subcases.

\begin{itemize}

\item Case $4.a$: $N_{1} \ll N$.

We will take $j=1$, which is the most difficult case; the other cases are treated similarly or are easier.
In this case $m(\xi_{1})=1$, and by the mean value theorem
\begin{equation}
\begin{array}{ll}
|w(\xi_{1},\xi_{2})| & \lesssim \frac{N_{1}}{N_{2}}.
\end{array}
\label{Eqn:BoundMult}
\end{equation}
Therefore we are reduced to proving that
\begin{equation}
\begin{array}{ll}
\int_{*; 1 \lesssim  N_{1}  \ll N_{0} \sim N_{2}}  \frac{
\prod_{j=0}^{2} F_{j}(\xi_{j},\tau_{j})}{N_2 |\xi_{0} +
\xi_{2}|^{\frac{1}{2}-} \lambda_{0}^{s-} \lambda_{1}^{ \left(  s
-\frac{1}{2} \right)-  }
 \lambda_{2}^{s -} \langle \tau_{0}+ \tau_{1} + \tau_{2} \rangle } & \lesssim
 \frac{1}{N^{\frac{1}{2}-}} \prod_{j=0}^{2} \| F_{j} \|_{L_{t}^{2} L_{x}^{2}}
\end{array}
\end{equation}
since $|\xi_{1}| =|\xi_{0} + \xi_{2}|$. In view of the assumptions,
it is enough to prove

\begin{equation}
\begin{array}{ll}
\left| \int_{\mathbb{R} \times \mathbb{R}^{3}} a_{1} D^{- \left(
\frac{1}{2}- \right)} ( u_{0} u_{2} ) u_{1} \, dx dt \right| &
\lesssim \| u_{0} \|_{H^{\frac{1}{4}+, s -}} \| u_{1} \|_{H^{0,
\left( s-\frac{1}{2} \right) -}} \| u_{2} \|_{H^{\frac{1}{4}+, s-}}
\end{array}
\end{equation}
But this inequality follows from Lemma \ref{a-ft}, H\"older's inequality and the
embeddings $ H^{0, \left( s- \frac{1}{2} \right) -} \subset
L_{t}^{2+} L_{x}^{2}$, $D^{- \left( \frac{1}{2} - \right)}  \left(
H^{\frac{1}{4}, s-}, \, H^{\frac{1}{4}, s-}  \right) \subset
L_{t}^{2} L_{x}^{2}$. The former embedding follows from interpolation between the embeddings $H^{0,0} = L^2_t L^2_x$ and $H^{0,1/2+} \subset C^0_t L^2_x$; the latter is
the $H^{s,b}$ version of the bilinear Strichartz estimates (see e.g. \cite{kl-tar:ym} or \cite[Corollary 1.4]{kl-rod-tao}).

\item Case $4.b$: $N_{1} \gtrsim N$.

Again we will take the most difficult case $j=1$; the other cases are treated similarly or are easier.
In this case, if $\beta > 1-s$, then we have

\begin{equation}
\begin{array}{ll}
|w(\xi_{1},\xi_{2})| & \lesssim \frac{1}{m(N_{1})} \\
& \lesssim \left( \frac{N_{1}}{N} \right)^{1-s} \\
& \lesssim \left( \frac{N_{1}}{N} \right)^{\beta --}
\end{array}
\end{equation}
where the $--$ exponent will dominate the various $-$ and $+$ exponents appearing elsewhere.

We are seeking for $\beta > 1-s$ as large as possible such that

\begin{equation}
\begin{array}{ll}
\int_{*; 1 \lesssim  N_{1}  \ll N_{0} \sim N_{2}}  \left(
\frac{N_{1}}{N} \right)^{\beta --} \frac{ N_{1}^{0+} \prod_{j=0}^{2}
F_{j}(\xi_{j},\tau_{j})} { N_{1}^{\frac{3}{2}} \lambda_{0}^{s -}
\lambda_{1}^{ \left( s -\frac{1}{2} \right) - }
 \lambda_{2}^{s -} \langle \tau_{0}+ \tau_{1} + \tau_{2} \rangle } & \lesssim
 \frac{1}{N^{\beta-}} \prod_{j=0}^{2} \| F_{j} \|_{L_{t}^{2} L_{x}^{2}}
\end{array}
\label{Eqn:IncreaseLastCase}
\end{equation}
for fixed $N_1$; we use the $N_1^{0+}$ factor to sum in $N_1$.

It is enough to find $\beta$ as large as possible such that

\begin{equation}\label{threes}
\begin{array}{ll}
\left| \int_{\mathbb{R} \times \mathbb{R}^{3} } a_{1}   u_{0}
P_{N_{1}}u_{1}  u_{2} \, dx dt \right| & \lesssim N_{1}^{0+} \|
u_{0} \|_{H^{0,s-}} \| u_{1} \|_{H^{\left( \frac{3}{2} -\beta
\right) -, \left( s- \frac{1}{2} \right)- }} \| u_{2} \|_{H^{0,s-}};
\end{array}
\end{equation}
the $--$ exponent in $N_1^{\beta--}$ can safely absorb the various $N_1^{0+}$ factors present with enough room to spare to still allow a summation in $N_1$.

But if $\beta < s -\frac{1}{2}$ then \eqref{threes} comes from Holder
inequality, the embedding $H^{0,s -} \subset L_{t}^{\infty -}
L_{x}^{2}$ and the embedding $ N_{1}^{0+} H^{ \left( \frac{3}{2} -
\beta \right)-, \left( s- \frac{1}{2} \right) -} \subset L_{t}^{2+}
L_{x}^{\infty} $. The former comes from the interpolation of $
H^{0,s -} \subset L_{t}^{2} L_{x}^{2}$ and $ H^{0,s-} \subset
L_{t}^{\infty} L_{x}^{2}$. The latter comes from the interpolation
of  \\
$ \| P_{N_{1}} v \|_{ L_{t}^{2++} L_{x}^{\infty}} \lesssim
N_{1}^{0++} \| P_{N_{1}} v \|_{ L_{t}^{2++} L_{x}^{\infty --}}
\lesssim N_{1}^{0++} \| v \|_{H^{1--,\frac{1}{2}+}}$ and the Sobolev
embedding $H^{\frac{3}{2}+,0} \subset L_{t}^{2} L_{x}^{\infty}$.
Therefore, by choosing $\beta =\left( s -\frac{1}{2} \right) -$
\footnote{Notice that this implies that $\beta > 1-s$, since $s >
\frac{3}{4}$ }, we get an increase of $O \left( \frac{1}{N^{ \left(
s -\frac{1}{2} \right)-}} \right)$ in (\ref{Eqn:IncreaseLastCase}).

\end{itemize}

\appendix

\section{Lack of extra smoothing}\label{lack-sec}

In the study of global well-posedness for the cubic nonlinear wave
equation (NLW) below the energy norm, see \cite{kpv:gwp}, an
important role is played by the \emph{extra smoothing phenomenon},
that the nonlinear component of the evolution is in the energy
class even if the initial data is not.  For instance, if $\Phi \in
C^0_t H^s_x \cap C^1_t H^{s-1}_x$ solves (NLW) for some $3/4 < s <
1$, then we have $\Phi - \Phi_\lin \in C^0_t H^1_x \cap C^1_t
L^2_x$, where $\Phi_\lin$ is the solution to the free wave equation
$\Box \Phi_\lin = 0$ with initial data $\Phi[0]$.  See
\cite{kpv:gwp} for details.  The reliance on such a smoothing phenomenon
is not unique to the work \cite{kpv:gwp}, but is rather an important feature
of all applications of the truncation method put forward by Bourgain (e.g. \cite{borg:book}).

In this appendix we show why this extra smoothing phenomenon fails
in a certain quantitative sense for the system (MKG-CG).
Specifically, we show

\begin{theorem}[Lack of extra smoothing]  Let $3/4 < s < 1$, and let $\eps > 0$ be sufficiently small depending on $s$.  Then for any $M > 0$ there exists a solution $\Phi \in C^0_t H^s_x \cap C^1_t H^{s-1}_x$ to (MKG-CG) with $\| \Phi[0] \|_{[H^s]} \leq \eps$ such that
$$ \| \phi[1] - \phi_\lin[1] \|_{[H^1]} \geq M,$$
where $\phi_\lin$ is the solution to the free wave equation $\Box
\phi_\lin = 0$ with initial data $\phi[0]$.
\end{theorem}

\begin{proof}  Suppose this were not the case.  Then we could find $3/4 < s < 1$ and arbitrarily small $\eps$ such that
$$ \| \phi[1] - \phi_\lin[1] \|_{[H^1]} \lesssim 1$$
whenever $\Phi$ solved (MKG-CG) with $\| \Phi[0] \|_{[H^s]} \leq
\eps$, where we allow the implicit constant to depend on $s$ and
$\eps$.  By Green's theorem for the d'Lambertian we thus see that
\begin{equation}\label{phif}
\left|\int_0^1 \int_{\R^3} \Box \phi(t,x) f_\lin(t,x)\ dx dt\right|
\lesssim 1
\end{equation}
whenever $f_\lin$ is a solution to the free wave equation with
$\|f[1]\|_{[L^2]} \sim \|f[0]\|_{[L^2]} \lesssim 1$.  As we shall
see, it is the component $-2i(\P \underline{A}) \cdot \nabla_x
\phi$, with the derivative of $\phi$ present, which causes
\eqref{phif} to fail, even with the presence of null structure.

We now work on the spacetime slab $[0,1] \times \R^3$. Modifying
slightly the proof of Theorem \ref{hs-lwp} in Section
\ref{local-sec}, we see (if $\eps$ is small enough) that $\Phi$ will
obey the estimates
$$\| \underline{\Phi} \|_{H^{s,3/4+}_{[0,1]}}
+ \| \nabla_{x,t} \underline{\Phi} \|_{H^{s-1,3/4+}_{[0,1]}} + \|
\nabla_{x,t} A_0 \|_{L^2_t H^{1/2+}_x} \lesssim \eps.$$ In fact,
from \eqref{ao} we have
$$ \| \nabla_{x,t} A_0 \|_{L^2_t H^{1/2+}_x} \lesssim \eps^2$$
while a small modification of the proof of \eqref{uphi} yields
$$ \| \underline{\Phi} - \underline{\Phi}_\lin \|_{H^{s,3/4+}_{[0,1]}} \lesssim \eps^2.$$
Now, from \eqref{bphi-eq} we have
$$ \Box \phi = - 2i (\P \underline{A}) \cdot \nabla_x \phi
+ \O( \N_1 ) + \O( \N_3 ).$$ But from the above estimates and the
arguments in Section \ref{local-sec} we have
$$ \| \O( \N_1 ) + \O( \N_3 ) \|_{H^{s-1,s-1}_{[0,1]}} \lesssim \eps^3$$
and similarly from the above estimates and \eqref{null} we have
$$ \| (\P \underline{A}) \cdot \nabla_x \phi - (\P \underline{A}_\lin) \cdot \nabla_x \phi_\lin
\|_{H^{s-1,s-1}_{[0,1]}} \lesssim \eps^3.$$ We thus conclude that
\begin{equation}\label{bpi}
\| \Box \phi + 2i (\P \underline{A}_\lin) \cdot \nabla_x \phi_\lin
\|_{H^{s-1,s-1}_{[0,1]}} \lesssim \eps^3.
\end{equation}
Now let $N \gg 1$ be a large parameter to be chosen later, and let
$f_\lin$ be a solution to the free wave equation with initial data
$f[0]$ and Fourier transform supported on the region $|\xi| \sim 1$,
and normalized so that $\|f[0]\|_{[H^{1-s}]} \lesssim 1$.  Then by
energy estimates and \eqref{bpi} we conclude that
$$\left|\int_0^1 \int_{\R^3} [\Box \phi + 2i (\P \underline{A}_\lin) \cdot \nabla_x \phi_\lin ] f_\lin\ dx dt\right| \lesssim \eps^3$$
while from \eqref{phif} we have
$$\left|\int_0^1 \int_{\R^3} [\Box \phi] f_\lin\ dx dt\right| \lesssim N^{s-1}.$$
If $N$ is sufficiently large depending on $\eps, s$, we conclude
that
\begin{equation}\label{intor}
\left|\int_0^1 \int_{\R^3} [ (\P \underline{A}_\lin) \cdot \nabla_x
\phi_\lin ] f_\lin\ dx dt\right| \lesssim \eps^3.
\end{equation}
We now construct $\underline{A}[0]$, $\phi[0]$, $f[0]$ so that the
above statement fails, yielding the contradiction.  Note that we
need to obey the compatibility conditions \eqref{compat}, but due to
Proposition \ref{a0-est} we can always do so by selecting $A_0[0]$,
so long as we have the divergence-free condition $\div
\underline{A}[0] = 0$, in which case we can drop the Leray
projection $\P$ in \eqref{intor}.

We shall select $\underline{A}[0]$, $\phi[0]$, $f[0]$ to have zero
initial velocity:
$$ \partial_t \underline{A}(0) = \partial_t \phi(0) = 0 = \partial_t f(0) = 0.$$
Then using the Fourier transform, \eqref{intor} becomes
\begin{equation}\label{intor-2}
\left|\int_0^1 \int_{\R^3} \int_{\R^3} \cos(t |\eta|) \cos(t |\xi|)
\cos(t |\xi+\eta|) \xi \cdot \hat{\underline{A}}(0,\eta) \hat
\phi(0,\xi) \hat f(0,-\xi-\eta)\ d\xi d\eta\right| \lesssim \eps^3.
\end{equation}
We let $e_1,e_2,e_3$ be the standard basis for $\R^3$.  We choose
\begin{align*}
\hat \phi(0,\xi) &:= \eps N^{-s} 1_{|\xi - N e_1| \leq 1} \\
\hat f(0,\xi) &:= \eps N^{s-1} 1_{|\xi + N e_1| \leq 1} \\
\hat A_1(0,\eta) &:= 1_{|\eta - e_2/10| \leq 1/100} \\
\hat A_2(0,\eta) &:= - \frac{\eta_1}{\eta_2} 1_{|\eta - e_2/10| \leq 1/100} \\
\hat A_3(0,\eta) &:= 0
\end{align*}
then $\underline{A}$ is divergence-free, and we see for $\xi, \eta$
in the support of $\hat \phi(0)$, $\hat{\underline{A}}(0)$
respectively that
$$ \frac{1}{100} \leq \int_0^1 \cos(t |\eta|) \cos(t |\xi|) \cos(t |\xi+\eta|)\ dt \leq 1$$
and that
$$ \xi \cdot \hat{\underline{A}}(0,\eta) \hat \phi(0,\xi) =
(1 + O(N^{-1})) \eps N^{1-s} 1_{|\eta - e_2/10| \leq 1/100} 1_{|\xi
- N e_1| \leq 1} $$ and so \eqref{intor-2} simplifies to
$$\int_{\R^3} \int_{\R^3} 1_{|\eta - e_2/10| \leq 1/100} 1_{|\xi - N e_1| \leq 1}
1_{|\xi + \eta - N e_1| \leq 1}\ d \xi d\eta \lesssim \eps.$$ But
this can easily be seen to be false for sufficiently small $\eps$,
giving the desired contradiction.
\end{proof}

\begin{remark}  The above construction shows that $\phi[1] - \phi_\lin[1]$ can be made arbitrarily large in the energy norm even when $\| \Phi \|_{[H^s]}$ is small.  It is possible to modify the above construction (using multiple frequency scales $N$) to in fact make $\phi[1] - \phi_\lin[1]$ have infinite energy; we leave the details to the reader.
\end{remark}

\end{document}